\documentclass{article}
\usepackage{amssymb}
\usepackage{amsmath}

\usepackage{hyperref}
\newcommand{\ti}{\tilde}
\newtheorem{theorem}{Theorem}[section]

\newtheorem{corollary}[theorem]{Corollary}

\newtheorem{lemma}[theorem]{Lemma}
\newtheorem{notation}[theorem]{Notation}

\newtheorem{proposition}[theorem]{Proposition}
\newtheorem{remark}[theorem]{Remark}

\newenvironment{proof}[1][Proof]{\noindent\textbf{#1.} }{\ \rule{0.5em}{0.5em}}
\newcommand{\f}{\frac}
\usepackage{color}%
\usepackage{graphicx}
\newcommand{\diff}{\mathrm{d}}

\definecolor{vert}{rgb}{0,0.58,0}

\usepackage{marvosym}

\newcommand{\ep}{\varepsilon}

\usepackage[normalem]{ulem}
\usepackage{caption}
\usepackage{subcaption}
\usepackage{manfnt}
\usepackage{cancel}
\usepackage{comment}
\usepackage{authblk}
\input{tcilatex}

\providecommand{\keywords}[1]
{
  \small	
  \textbf{\textit{Keywords---}} #1
}

\begin{document}

\title{Asymptotic Analysis of a bi-monomeric nonlinear Becker-D\"{o}ring
system}
\author[1]{Marie Doumic}
\author[2]{Klemens Fellner}
\author[3]{Mathieu Mezache} 
\author[4]{Juan J.L. Vel\'azquez}

\affil[1]{\footnotesize Ecole Polytechnique, Inria, CNRS, Institut Polytechnique de Paris, route de Saclay, 91128 Palaiseau Cedex, France, marie.doumic@inria.fr}
\affil[2]{\footnotesize Department of Mathematics and Scientific Computing, University of Graz,
Heinrichstrasse 36, Graz 8010, Austria, klemens.fellner@uni-graz.at}
\affil[3]{\footnotesize Université Paris-Saclay, INRAE, MaIAGE (UR 1404), 78350 Jouy-en-Josas, France, mathieu.mezache@inrae.fr}
\affil[4]{University of Bonn, Institute for Applied Mathematics, Endenicher Allee 60, D-53115 Bonn, Germany, velazquez@iam.uni-bonn.de}

\maketitle

\begin{abstract}
{To provide a mechanistic explanation of sustained {then} damped oscillations observed in a depolymerisation experiment,  a bi-monomeric variant of the seminal Becker-D\"oring system has been proposed in~\cite{DFMR}. 
When all reaction rates are constant, the equations are the following:
\begin{align*}
\frac{dv}{dt} & =-vw+v\sum_{j=2}^{\infty}c_{j}, \qquad
\frac{dw}{dt} =vw-w\sum_{j=1}^{\infty}c_{j},  \\
\frac{dc_{j}}{dt} & =J_{j-1}-J_{j}\ \ ,\ \ j\geq1\ \ ,\ \ \
J_{j}=wc_{j}-vc_{j+1}\ \ ,\ \ j\geq1\ \ ,\ J_{0}=0,
\end{align*}
where $v$ and $w$ are two distinct unit species, and $c_i$ represents the concentration of clusters containing $i$ units. 

We study in detail the mechanisms leading to such oscillations and characterise the different phases of the dynamics, from the initial high-amplitude oscillations to the progressive damping leading to the convergence towards the unique positive stationary solution. We give quantitative approximations for the main quantities of interest: period of the oscillations, size of the damping (corresponding to a loss of energy), number of oscillations characterising each phase. We illustrate these results by numerical simulation,  in line with the theoretical results,  and provide numerical methods to solve the system.}
\end{abstract}

\tableofcontents
\section{Introduction}

We are interested in describing the damped oscillations of the following
model: 
\begin{align}
\frac{\diff v}{\diff t} & =-vw+v\sum_{j=2}^{\infty}c_{j}, \qquad v(0)=v^0,  \label{A1} \\
\frac{\diff w}{\diff t} & =vw-w\sum_{j=1}^{\infty}c_{j},  \qquad w(0)=w^0, \label{A2} \\
\frac{\diff c_{j}}{\diff t} & =J_{j-1}-J_{j}, \quad J_{j}=wc_{j}-vc_{j+1},\; c_j(0)=c_j^0,\; j\geq1,\;  J_{0}=0. \label{A3} 
\end{align}
{In this system, $v(t)$ and $w(t)$ denote the concentrations of two monomeric species
and 
$c_j(t)$ represents the concentrations of polymers, clusters or aggregates containing $j$ units/monomers at time $t$. Clusters grow by polymerisation events adding a $w$-monomer and shrink by catalytic depolymerisation induced by $v$-monomers.} 
System (\ref{A1})-(\ref{A3}) is a particular case of a system with more
general coefficients that was introduced in~\cite{DFMR}. Specifically, we
are assuming in (\ref{A1})-(\ref{A3}) that the reaction rates are size-independent. The goal of introducing this model is to explain sustained, though damped, oscillations,
experimentally observed during the time-course of protein fibrils
depolymerisation experiments~\cite{DFMR,fornara2024dynamics}, {which are also displayed when simulating~\eqref{A1}--\eqref{A3}, see \cite{DFMR}. 
These experiments raised significant interest since the classical Becker-D\"oring model for the polymerisation/depolymerisation of polymers 
features a well-known Ljapunov functional, which conflicts with sustained oscillations. The need to propose a new mathematical model like \eqref{A1}--\eqref{A3} was therefore pointed out by the experimentalists in the hope to gather new insights into the mechanics of prion dynamics.}  The question of how variants of standard aggregation-fragmentation systems  can give rise to persistent or persistent then damped oscillations has attracted  increasing interest in recent years. For example, another modification of the Becker-D\"oring system has been proposed in~\cite{pego2020temporal,niethammer2022oscillations}. This system leads to sustained oscillations due to a very different mechanism, namely the atomisation of large polymers into monomers, see also~\cite{budzinskiy2021hopf,fortin2023stability}. Such rapid shortening events  also appear to be responsible for damped oscillations observed numerically in a model of microtubule dynamics~\cite{honore2019growth}.

Notice that {there are} two remarkable differences between
the model (\ref{A1})-(\ref{A3}) and the classical Becker-D\"oring  model{{~\cite{BeckerDoring_1935}}}, as it
can be found for instance in~{{\cite{Ball-Becker-Doring,Thierry,JabinNiethammer,velazquez1998becker}}}. On the one hand, there are  two
different types of monomers, whose concentrations are given by $v,\ w,$ whereas $c_1$ represents the  concentration of the smallest cluster, which is not necessarily a monomeric species. This leads to the  conservation of the total number of clusters $\sum_{j=1}^\infty c_j,$ a quantity which is not preserved in the Becker-D\"oring system.

As a second difference, the reaction yielding the depolymerisation of a cluster of size $j$ into  a $v$-monomer
and a cluster of size $\left( j-1\right) $ is not modelled as a 
linear process, but as a nonlinear reaction catalyzed by a $v$-monomer itself.
In other words, the depolymerisation reaction has the form%
\begin{equation}
c_{j}+v\rightarrow c_{j-1}+2v.  \label{eq:CatRe}
\end{equation}
We emphasise that the presence of several types of oligomers~\cite{armiento2017mechanism} and monomers as well as the possibility of the
uncommon catalytic depolymerisation reaction (\ref{eq:CatRe}) have been then
confirmed by other observations~\cite{torrent2019pressure}. 

A first numerical and theoretical study has been carried out in~\cite{DFMR}. It gave necessary and sufficient conditions for the existence of a positive steady state and studied the stability of boundary steady states (Proposition~4 in~\cite{DFMR}). Numerical simulations and comparison with well-known oscillatory systems (subsection~5.3. in~\cite{DFMR}) provided evidence for sustained damped oscillations converging towards the positive steady state, when it exists.

The aim of the present study is to characterise in full detail the mechanisms driving the system from pronounced oscillations to equilibrium. The results  will be justified using formal asymptotic expansions (see for instance \cite{guckenheimer2013nonlinear}, \cite{kevorkian2013perturbation} or \cite{bender2013advanced}), complemented with numerical simulations yielding results analogous to those obtained in the asymptotic formulae. Some rigorous mathematical results that support the scenario for the solutions described in this paper are proved in the companion paper~\cite{DFMV2}.

The plan of this paper is the following. In section~\ref{sec:coupling}, we describe in detail the model studied and provide some important notations and lemmas which are used throughout the paper. In subsection~\ref{subsec:LV}, we summarize the form of the solutions of a rescaled version of the Lotka-Volterra model, which we will identify as approximation to the behaviour of the monomeric concentrations {$(v,w)$}, and which will be shown to drive the oscillatory behaviour of the solutions of~\eqref{A1}--\eqref{A3} (subsection~\ref{subsec:sizedistrib}). Finally, in subsection~\ref{subsec:sketch}, we sum-up the main results and describe the processes that evolve the initial distribution of clusters to the final equilibrium distribution. We have found it convenient to decompose this evolution into four different Phases I, II, III, and IV, which are characterised by the range of values of the Lotka-Volterra energy associated to the monomer concentrations $(v,w)$ as well as the range of a typical cluster size $k$ containing most of the mass of the cluster distribution $\{c_k\}_{k\in\mathbb{N}}$.

Section~\ref{sec:phaseIandII} describes in detail the evolution of the cluster concentrations $\{c_k\}$ during one Lotka-Volterra cycle of the concentrations $(v,w).$ This analysis is relevant throughout the phases that we denote as Phases~I and~II. The evolution of the concentrations $\{c_k\}$ along several Lotka-Volterra cycles, and the slow changes occuring from one cycle to the next during Phases~I and~II, is then described in section~\ref{sec:IandII:overall}. Phases~III and~IV are described in sections~\ref{sec:PhaseIII} and~\ref{sec:PhaseIV} respectively. Section~\ref{sec:discussion} summarises and discusses the main conclusions of this article as well as possible generalisations. Finally, three appendices at the end of the paper some detailed calculations about the form of the solutions of the Lotka-Volterra problem, on analysis of the linearised problem around the steady states of~\eqref{A1}--\eqref{A3} and a description of the numerical methods we used.

\begin{notation}
We will use extensively the classical asymptotic notation $f\sim g$ as $%
x\rightarrow x_{0}$ to indicate that $\lim_{x\rightarrow x_{0}}\frac{f\left(
x\right) }{g\left( x\right) }=1.$ When $x_0$ is not specified it means that $f\sim g$ when $\ep \to 0.$ We will use also the notation $f\ll g$ as $%
x\rightarrow x_{0}$ to indicate that $\lim_{x\rightarrow x_{0}}\frac{f\left(
x\right) }{g\left( x\right) }=0.$ We will use the notation $f\simeq g$ to
indicate informally that the two functions $f$ and $g$ can be expected to be
approximately equal for a suitable range of values. We will use the notation 
$f\approx g$ to indicate that the functions $f$ and $g$ are of the same
order of magnitude for some range of values of their argument. More
precisely, there exists a constant $C>1$ such that $\frac{1}{C}f\leq g\leq
Cf.\ $We will use the notation $f\lesssim g$ to indicate that there is a
constant $C>0$ such that $f\leq Cg.$  
\end{notation}

\bigskip

\section{Coupling Lotka-Volterra and Becker-D\"oring systems}
\label{sec:coupling}
{\subsection{A first overview of the dynamics}}
\label{subsec:overview}
{When looking at~\eqref{A1}--\eqref{A3},} we first observe that there are two conserved quantities, namely
\begin{equation}
M=v+w+\sum_{j=1}^{\infty}jc_{j},\ \quad { \varepsilon}=\sum_{j=1}^{\infty
}c_{j}. \label{eq:MandConc}
\end{equation}
The first conserved quantity represents the total mass $M>0$  whereas the second is the number of clusters.
This second conservation law is an important difference between our system
and the classical Becker-D\"{o}ring system, {in which the smallest clusters and monomers are identical. Our study will show its importance to the dynamical
behaviour: Since our system preserves the number of clusters, the question is how its size distribution evolves with time.} 

We remark that we can assume $M=1$ without loss of generality because
the equations \eqref{A1}--\eqref{A3} are invariant under the change of
variables $v\rightarrow Mv,\ w\rightarrow Mw,\ c_{j}\rightarrow Mc_{j},$ {and $\varepsilon\rightarrow M\varepsilon$,}
which transforms the equations (\ref{eq:MandConc}) in%
\begin{equation}
1=v+w+\sum_{j=1}^{\infty}jc_{j},\ \ \varepsilon=\sum_{j=1}^{\infty}c_{j}.
\label{eq:Mresc}
\end{equation}

Notice that since $v\geq0$
and $w\geq0$ these equations imply that $\varepsilon\leq1$ and the identity $%
\varepsilon=1$ takes place only if $v=w=0$ and $c_{j}=\delta_{j,1}.$

The aim of our study is to quantitatively analyse the asymptotic behaviour
of the solutions of the problem \eqref{A1}--\eqref{A3}, \eqref{eq:Mresc}  for $%
\varepsilon\rightarrow0.$ The smallness of $\varepsilon$ may have two
biological interpretations: either only few polymers exist initially and the initial mass lies mainly in
the two monomeric species $v$ and $w,$ or the mass is distributed among monomers
and clusters with a very large average cluster size $i_{M}$ so that $%
i_{M}(0):=\frac{\sum ic_{i}(0)}{\sum c_{i}}=O(\frac{1}{\varepsilon}).$
This last case is typical of amyloid fibrils depolymerisation or fragmentation experiments~\cite{XueRadford2013,beal2020division}.

Using the conservation of the number of clusters, we can rewrite \eqref{A1}--\eqref{A3} in a simpler form:%
\begin{align}
\frac{\diff v}{\diff t} & =-vw+v\left( \varepsilon-c_{1}\right),  \label{eq:BD2} \\
\frac{\diff w}{\diff t} & =vw-\varepsilon w , \label{eq:BD2a} \\
\frac{\diff c_{j}}{\diff t} & =J_{j-1}-J_{j},\ \ \forall j\geq1,\qquad J_{0}=0, \quad
J_{j}=wc_{j}-vc_{j+1},\ \ j\geq1.  \label{eq:BD3}
\end{align}

\noindent The global-in-time {existence and uniqueness of a nonnegative solution to} this system has been proved in~\cite%
{DFMR} (Theorem~2), under the assumption that the initial nonnegative state $(v^0,w^0,c_j^0)$ satisfies
\[v^0+w^0+\sum\limits_{j=1}^\infty j^2 c_j^0 <\infty.\]
There is a family of boundary steady states for which $w=0.$ 
We can have either $v=0$ and an arbitrary distribution of
concentrations $c_{j}$ satisfying the constraint $\sum_{j\geq1}jc_{j}=1,$
or alternatively, $v=1-\varepsilon,$ $%
c_{1}=\varepsilon$, and $c_{j}=0$ for $j\geq2$. All the boundary steady states happen to be unstable {as soon as $\ep <\f{1}{2},$ what is assumed from now on (since we even assume $\ep \ll 1$),}
see Prop.~4 in~\cite{DFMR}. Existence and uniqueness of a unique positive
steady state (Corollary~4 in~\cite{DFMR}) {for $\ep <\f{1}{2}$} is given by 
\begin{equation}
\begin{array}{l}
\theta:=1-{\frac{1}{2\ep}\biggl(1-\sqrt{\bigl(1-2\ep \bigr)^{2}+{4\ep^2}}}\biggr)\sim \varepsilon, \\ 
\bar{c}_{1}:=\varepsilon\theta\sim{\varepsilon^{2}},\qquad\bar{c}%
_{i}:=(1-\theta)^{i-1}\bar{c}_{1},\qquad\bar{v}:=\varepsilon,\qquad\bar {w}%
:=\varepsilon-\bar{c}_{1}.%
\end{array}
\label{eq:steady}
\end{equation}
In Appendix~\ref{app:linear}, we prove that this positive steady state is
locally linearly stable. 

The system~\eqref{eq:BD2}--\eqref{eq:BD2a} for $%
(v,w)$ may be viewed as a $c_1$-perturbation of a Lotka-Volterra {(LV)} system, which
would permanently oscillate in case $c_{1}\equiv 0.$ 
{The unperturbed Lotka-Volterra system, {\it i.e.}~\eqref{eq:BD2}--\eqref{eq:BD2a} with $c_1\equiv0$, is governed by the Hamiltonian 
\begin{equation}
E=v+w-2\varepsilon-\varepsilon\log\left( \frac{vw}{\varepsilon^{2}}\right),
\label{D5}
\end{equation}
and features the steady state $v=w=\varepsilon$, which corresponds to the minimum of the Hamiltonian $E=0$. All other trajectories
of the unperturbed Lotka-Volterra system \eqref{eq:BD2}--\eqref{eq:BD2a} with $c_1\equiv 0$ are closed periodic orbits 
for positive values of $E$ 
and we define $T(E)$ as the associated time period.} 

For the full system \eqref{eq:BD2}--\eqref{eq:BD3}, in which $v,$ $w$ interact
with the clusters $c_j$ and vice versa, the function $E(t)$ is neither a conserved quantity nor a Ljapunov functional. The absence of
a suitable entropy is one difficulty in the study of \eqref{eq:BD2}--\eqref{eq:BD3}. More precisely, the following
formula for the change of the energy holds  
\begin{align}
\frac{\diff E}{\diff t} \label{D5a} 
=\left( \varepsilon-v\right) c_{1}.\  
\end{align}
As we shall demonstrate, solutions to 
the full system \eqref{eq:BD2}--\eqref{eq:BD3} feature 
approximate LV cycles, where $v$ and $w$ oscillate around the steady state values of the full system $\bar v=\ep$ and $\bar w=\ep -\bar c_1$; see Lemmas~\ref{lem:LV1} to~\ref{lem:LV5} for a full description of the dynamics of one LV cycle.  
Hence we see from~\eqref{D5a} that  the sign of $\frac {\diff E}{\diff t}$ changes over every time period. 
The energy level $\bar E$ of the steady state of the full system is computed from~\eqref{eq:steady}:
\begin{equation}\label{eq:steady:energy}
    \bar E=\bar v+\bar w-2\varepsilon-\varepsilon\log\left( \frac{\bar v\bar w}{\varepsilon^{2}}\right)
=-\bar c_1 -\ep\log\left(1-\f{\bar c_1}{\ep}\right)\sim \ep \f{\theta^2}{2} \sim \f{\ep^3}{2}.
\end{equation}
We expect that solutions to our model system \eqref{eq:BD2}--\eqref{eq:BD3}, which depart from $E\approx 1,$ will {slowly {decrease} the mean value of the energy $E$ along successive approximate LV cycle,  each of these cycles being characterised by a given level of energy $E.$} Quantifying precisely this trend is the aim of this article.
{We will show that the decay-on-average of the values of $E$ takes place by means of an intricate mechanism, in which the
concentrations $c_{j}$ oscillate in cluster sizes $j$} {with a transport speed given by $w-v.$} 

For small perturbations of the steady state~\eqref{eq:steady}, we expect $%
c_{1}$ to remain in the order of $\bar c_1 \sim \varepsilon^{2}.$ Else, we can expect $%
c_{1}$ to be very small, except for the times in which  $\left\langle
j\right\rangle $, the mean value of the concentrations $c_j$, reaches its minimum values. This happens when
depolymerisation dominates, namely for $v\gg\varepsilon.$  During these times, even if $c_{1}$ still remains
small, it modifies $E$ the most. This explains why  we
expect that in average $\Delta E =\int_{t}^{t+T(t)}{\f{\diff}{\diff s}} E(s)ds$, where $T(t)$ is the time (pseudo) period, to be negative. 

In fact, the values of $c_{1}$ remain small during most times, except for a possible initial transient state {(for instance if the system departs from $c_j=\ep \delta_1(j)$)}. 
Therefore, the change of $E$ will be small in each cycle of the variables $v,\ w,$ but
eventually its value will decrease and approach the minimum value 
 {$E\sim \frac{\varepsilon^3}{2}$} 
after sufficiently long times. 

A full description of this slow approach is the aim of this article. To do so, we first focus on the dynamics during one (approximately unperturbed) LV cycle. In subsection~\ref{subsec:LV} we begin by analysing the different stages of the dynamics of $(v,w)$ over one LV  cycle (Lemmas~\ref{lem:LV1} to~\ref{lem:LV5}). This done,  in subsection~\ref{subsec:sizedistrib} we can focus on  the size distribution dynamics, which appears close to a drift-diffusion equation with coefficients given in terms of the previously studied $(v,w).$ This strategy is the basis of our asymptotic long-time study and explains how the perturbation slowly modifies the LV cycles; we sketch it in subsection~\ref{subsec:sketch}.

\subsection{Study of the {unperturbed} Lotka-Volterra system\label{subsec:LV}}

According to our above argumentation, $c_{1}$ is expected to be very small during most of the evolution; It will be
shown {\it a posteriori} that this ansatz is self-consistent. In this section, we thus study in detail
the dynamics over one cycle of $\left( v,w\right) $ solutions to the unperturbed LV system, \textit{i.e.},~%
\eqref{eq:BD2}--\eqref{eq:BD2a} with $c_{1}\equiv 0,$ namely
\begin{equation}
\frac{dv}{dt}=-vw+v\varepsilon ,\ \ \frac{dw}{dt}=vw-\varepsilon w{, \ \ w(0)=v(0)\geq \varepsilon},  \label{D8}
\end{equation}
{where the last condition defines a convenient starting- and end-point of a LV cycle.}
Notice that the steady state of (\ref{D8}) is at%
\begin{equation}
v=w=\varepsilon, \ \ {E =0.}  \label{D8Equ}
\end{equation}

We recall that the energy $E$ defined by~\eqref{D5}  is constant along the
solutions of (\ref{D8}): the trajectories associated to $(v,w)$ solution of~(\ref{D8}) are 
the level lines of the function $E.$ The solutions $(v,w)$ are periodic,  oscillating around the equilibrium
value~(\ref{D8Equ}).

In order to describe the behaviour of the solutions of (\ref{D8}) we first
use some rescaling properties of the solutions, summed-up in the following lemma.
\begin{lemma}[Scaling properties]\label{lem:scaling1}
Let $(v,w)(t;E,\ep)$ denote the unique solution $(v,w)$ to~\eqref{D8} with $v(0)=w(0)>\ep$ and $E$ defined by~\eqref{D5}, {that is 
\begin{equation}\label{Ev0}
    E = 2(v(0)-\ep) - 2\ep \log \left( \frac{v(0)}{\ep}\right).
\end{equation}
Note that \eqref{Ev0} is strictly monotone increasing in $v(0)>\ep$ and that trajectories of \eqref{D8} are uniquely defined in terms of values $E>0$.} Let $T(E,\ep)$ denote the period of the corresponding LV cycle. By a scaling argument we have the following identities:
\begin{equation}
\left( w\left( t;E,\varepsilon\right) ,v\left( t;E,\varepsilon\right)
\right) =\left( E w\left( Et;1,\frac{\varepsilon}{E}\right) ,E v\left( Et;1,%
\frac{\varepsilon}{E}\right) \right)  \label{eq:Scal}
\end{equation}
and
\begin{equation}
T\left( E,\varepsilon\right) =\frac{T\left( 1,\frac{\varepsilon}{E}\right) }{%
E} . \label{eq:TScal}
\end{equation}
\end{lemma}

A large part of this paper will be concerned with the dynamics of the
solutions during the range of times in which $\frac{\varepsilon}{E}\ll1.$
Due to (\ref{eq:Scal}) and (\ref{eq:TScal}), in order to compute the
asymptotic behaviour of $w,\ v,\ T$ for this range of values of $\varepsilon$ and $ E,$ it is enough to compute their asymptotic behaviour for $E=1$ and $\varepsilon \rightarrow 0.$
{Given $E=1$, we first note that \eqref{Ev0} implies
\begin{equation}
v(0)=w(0)\sim\frac{1}{2}+\varepsilon\log\left( \frac{1}{\varepsilon }%
\right) +\varepsilon\left( 1-\log\left( 2\right) \right) +O(\ep^2 \log(\ep))\ \ \text{as\ \ }%
\varepsilon\rightarrow0 . \label{D8b}
\end{equation}
}

Figure \ref{fig:phasevw} depicts the $(v,w)$ phase portrait of a LV cycle and the associated time evolution of the concentrations. We divide a LV cycle into five stage subject to the following five lemmas, each describing the corresponding contribution to the period $T$. The proofs of these lemmas by standard asymptotic expansions are carried out in Appendix~\ref{app:LV}.

\begin{figure}[ptb]
\includegraphics[width=0.49\textwidth]{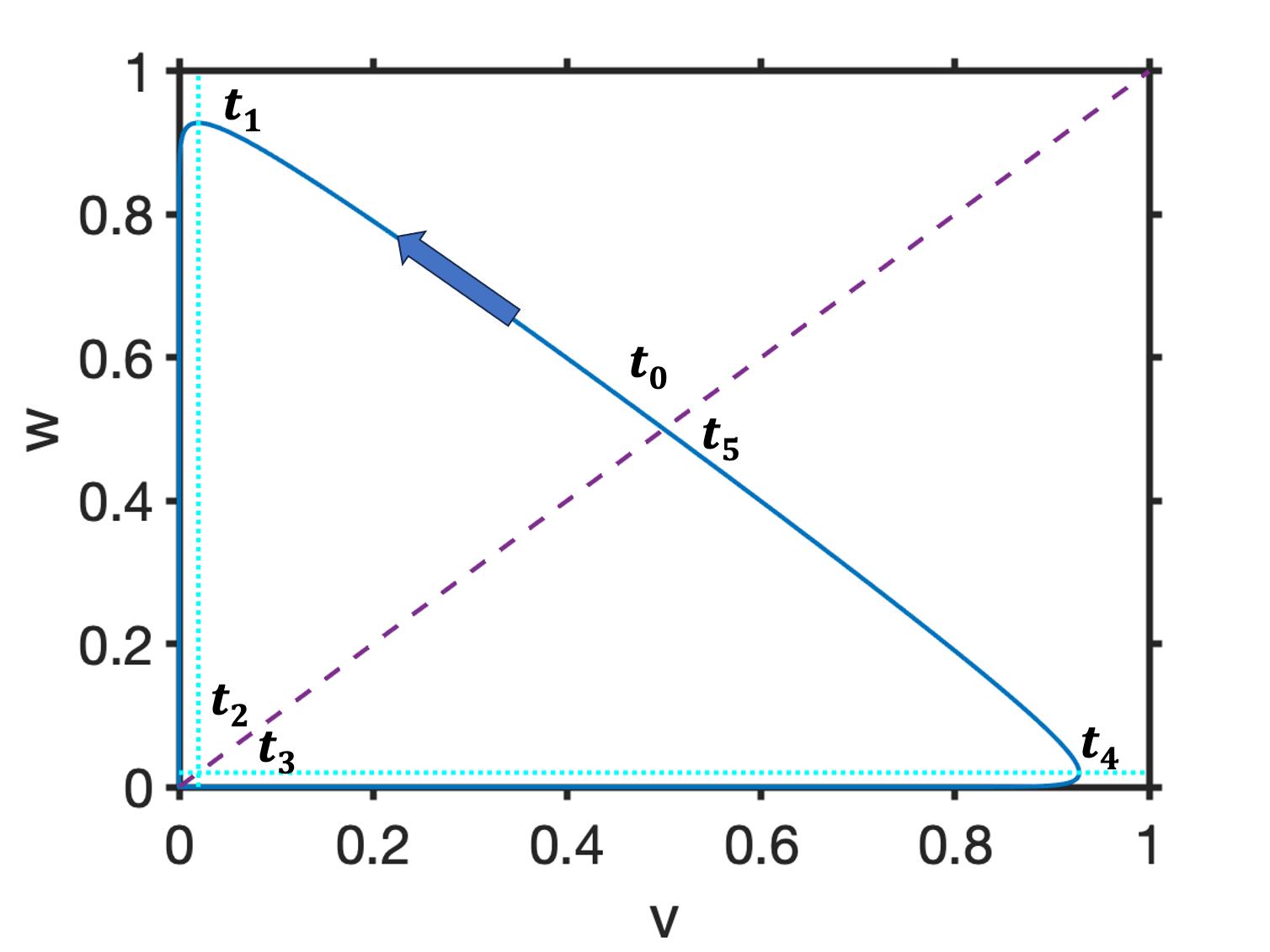} %
\includegraphics[width=0.49\textwidth]{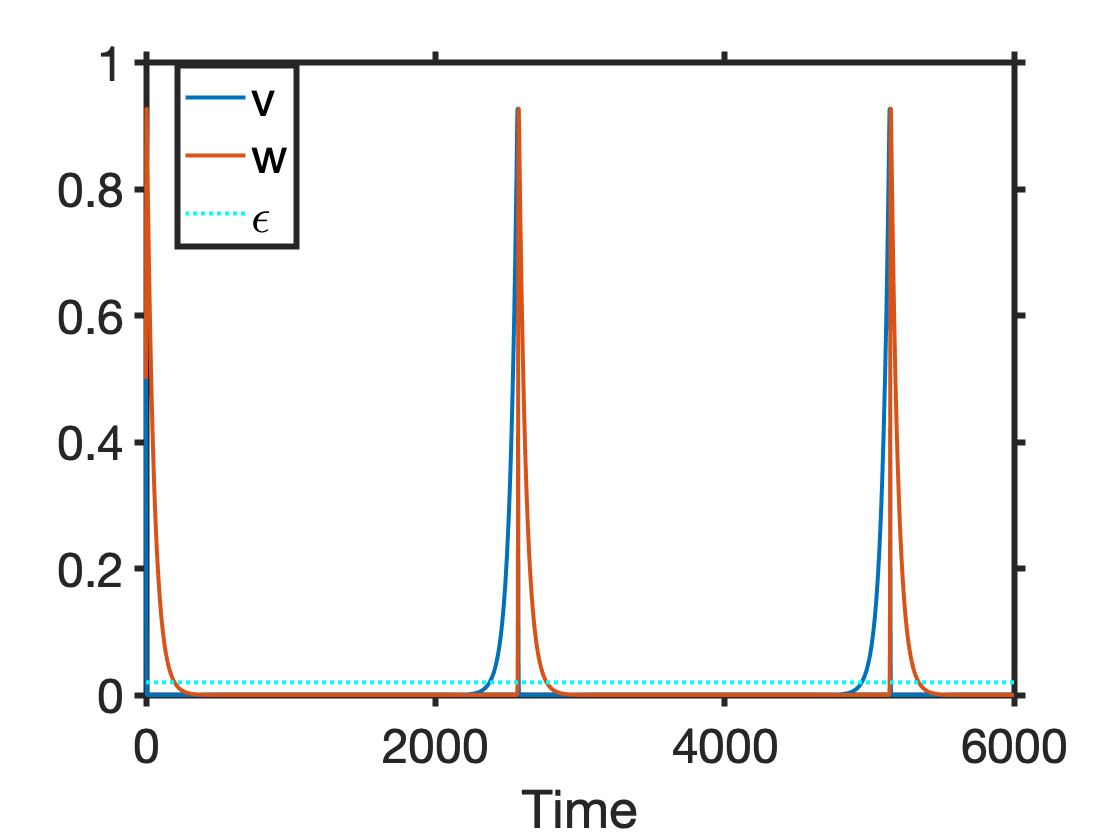}
\caption{Numerical solution of the Lotka-Volterra system \eqref{D8} { for $\ep=0.02$ and $v(0)=w(0)=\f{1}{2}$. To avoid numerical dissipation we rather solved the system for $(\log(v),\log(w))$, see Appendix~\ref{sec:app_PhaseI}, System~\eqref{eq:LV_exp_p1}.} Left: phase portrait for $(v,w)$ illustrating the different stages of the trajectory over one period T. The dashed purple line is the subset where $v=w$, the dotted blue lines are $w=\ep$ or $v=\ep$. Right: time-evolution of the solution $(v,w)$.}
\label{fig:phasevw}
\end{figure}

\begin{lemma}[First stage: fast decay of $v$ from {$v(0)$ to $\ep$, increase of $w$}]
\label{lem:LV1} 
Let $E$ be given by~\eqref{D5} and $\varepsilon\ll 1$. Let  
$(v,w)$ be the unique periodic solution to \eqref{D8} corresponding to $E=1$, i.e. 
subject to initial data at time $t_0=0$ asymptotically characterised in \eqref{D8b}.
We define $t_{1}$ by $$
t_{1}:=\min_{t\geq 0}\left\{ t,\ v(t)=\varepsilon\right\} .$$
On $(0,t_{1}),$ $w$ increases, $v$ decreases, and  as $\varepsilon
\rightarrow0$  the solution
is approximated by 
\begin{equation}\begin{array}{lllll}
v(t)&=&1-\frac{1}{1+e^{-t}} +O(\varepsilon e^{-t}\log(\varepsilon) t),&\\ \\
 w(t) &=& 
\frac{1}{1+e^{-t}} \left(1+O(\varepsilon \log(\varepsilon) t)\right), & t_1\sim-\log(\varepsilon),
\\ \\
w(t_1)&=&1-\ep\log(\ep)+\ep+o(\ep),& v(t_1)= \ep .  \label{estim:t1}
\end{array}
\end{equation}
\end{lemma}
\begin{lemma}[Second stage: decay of $v$ and $w$] 
\label{lem:LV2}\hfill\\
Under the assumptions of Lemma~\ref{lem:LV1},
we define $$t_{2}:=\min_{t\geq t_{0}}\left\{ t,\ w(t)=\varepsilon\right\}.$$
On $(t_{1},t_{2}),$ $w$ and $v$ decrease, and as $\varepsilon
\rightarrow0$ we have
\begin{equation}
 t_{2}-t_{1}\sim\frac{1}{\varepsilon}%
\log\left( \frac{1}{\varepsilon}\right) .\   \label{estim:t2}
\end{equation}
\end{lemma}

\begin{lemma}[Third stage: $v$ increases, $w$ decreases, minimum of $v+w$] \label{lem:LV3}
Under the assumptions of Lemma~\ref{lem:LV1}, we define $$
t_{3}:=\min_{t\geq t_{2}}\left\{ t,\ v(t)=\varepsilon\right\} .$$
On $(t_{2},t_{3}),$ $w$ decreases, $v$ increases to $\varepsilon$, and as $%
\varepsilon\rightarrow0$ we have
\begin{equation}
t_{3}-t_{2}\sim\frac{1}{\varepsilon^{2}},\quad w\sim \varepsilon\exp\left(
-\varepsilon\left( t-t_{2}\right) \right) ,\quad v\sim v\left(
t_{2}\right) e^{-1}\exp\left( \varepsilon\left( t-t_{2}\right) \right).
\label{estim:t3}
\end{equation}
\end{lemma}
\begin{lemma}[Fourth stage: $v$ and $w$ increase, symmetric to Stage 2]
\label{lem:LV4}\hfill\\
Under the assumptions of Lemma~\ref{lem:LV1}, we define $$
t_{4}:=\min_{t\geq t_{3}}\left\{ t,\ w(t)=\varepsilon\right\} .$$
On $(t_{3},t_{4}),$ $w$ and $v$ increase, and as $\varepsilon\rightarrow0$
the time interval  is approximated by 
\begin{equation}
v\sim\varepsilon e^{\varepsilon\left( t-t_{3}\right) },\qquad
t_{4}-t_{3}\sim \frac{1}{\varepsilon}\log\left( \frac{1}{\varepsilon }%
\right) . \label{estim:t4}
\end{equation}
\end{lemma}
\begin{lemma}[Fifth stage: $v$ decreases, $w$ increases, end of the cycle] \label{lem:LV5}
\hfill\\
Under the assumptions of Lemma~\ref{lem:LV1},
we define $$t_{5}:=\min_{t\geq t_{4}}\left\{ t,\ v(t)=w(t)>\varepsilon\right\}.$$
On $(t_{4},t_{5}),$ $w$ increases, $v$ decreases, and as $\varepsilon
\rightarrow0$ the time interval is approximated by 
\begin{equation}
t_{5}-t_{4}\sim\log\left( \frac{1}{\varepsilon}\right) . \label{estim:t5}
\end{equation}
\end{lemma}

We remark that the period $T$ is given as $T=t_5-t_0=t_5$.
We also notice that the third stage is the longest, giving the total cycle its asymptotic length. During this stage almost nothing visible happens since both $v$ and $w$ are exponentially small. By contrast, the first and last stages, where $v$ and $w$ are largest, are extremely fast. These considerations drive all the following analysis, {where we distinguish between a transport part (first and last stages of the cycle) and a transport-diffusion part (second to fourth stages). This is made clearer in the following section where we analyse the dynamics of the cluster size distribution along one LV cycle.}

\subsection{Cluster size distribution dynamics along one LV cycle, initial time}
\label{subsec:sizedistrib}

In subsection~\ref{subsec:LV}, we have described the dynamics of $(v,w)$ along one LV cycle, neglecting the perturbation given by $c_1.$ 
Let us now reason similarly for the cluster size distribution: assuming that $(v,w)$ is given by the unperturbed LV system, how does the cluster distribution evolve?

At this point, we remark that the equations {(\ref{eq:BD3})} for the evolution of the clusters can be
rewritten as%
\begin{align}
\frac{dc_{j}}{dt}& =\left( \frac{w-v}{2}\right) \left(
c_{j-1}-c_{j+1}\right) +\left( \frac{w+v}{2}\right) \left(
c_{j-1}+c_{j+1}-2c_{j}\right),\ \ j\geq 2\ , \label{eq:DiffForm1} \\
\frac{dc_{1}}{dt}& =vc_{2}-wc_{1}.  \label{eq:DiffForm2}
\end{align}
For given $v$ and $w$, the equations (\ref{eq:DiffForm1}) for $j\ge2$ constitute 
discrete-in-size convection-diffusion equations while equation (\ref{eq:DiffForm2})
for the smallest cluster concentration $c_1$ forms a kind of dynamical boundary condition. The approximation of discrete-in-size polymerisation/depolymerisation by a continuous-in-size convection-diffusion 
PDE is well-known in the literature, see e.g. \cite{velazquez1998becker}.
If we assume that the concentrations $c_{j}$ change
slowly in the variable $j$ (in a sense that will be precised later), we can
read (\ref{eq:DiffForm1}), (\ref{eq:DiffForm2}) as a discretized version
of the PDE
\begin{equation}
\frac{\partial c}{\partial t}\left( j,t\right) +{ \tilde{V}}\left( t\right) 
\frac{\partial c}{\partial j}\left( j,t\right) =\frac{{ d}\left( t\right) }{2}%
\frac{\partial^{2}c}{\partial j^{2}}\left( j,t\right).  \label{eq:DiffApp}
\end{equation}
where we introduce 
\begin{equation}
{\tilde{V}(t)} =\left( w-v\right) \left( t\right), \ \,\ \ {d}\left(
t\right) =\left( w+v\right) \left( t\right).\ \,\ \  
\label{eq:DiffCoef}
\end{equation}
In particular, we approximate the discrete-in-size cluster concentrations $c_{j}\left( t\right) $ by a smooth continuous-in-size distribution function $c\left( j,t\right) $ where $j\in(0,\infty)$ denotes now -- with a slight abuse of notation -- a continuum size variable. Hence, we assume  
\begin{equation}\label{eq:contapprox}
c_{j}(t) \simeq c(j,t), \ j\in\mathbb{N}, 
\qquad \vert c_{j+1} - c_{j} \vert 
\ll 1,
\quad\quad \left\vert 
\frac{\partial c}{\partial j}\right\vert \ll1.
\end{equation}
The PDE (\ref{eq:DiffApp}) constitutes a valid approximation for sufficiently large $j$, \cite{velazquez1998becker}. For
cluster sizes $j$ of order one the approximation (\ref{eq:DiffApp}) breaks
down and in order to calculate the concentrations $c_{j}$ we need to use
the system of equations \eqref{eq:DiffForm1}--\eqref{eq:DiffForm2}. 
Hence, we denote by \textit{the outer layer} the range in size where the approximation \eqref{eq:DiffApp} holds and by \textit{the boundary layer} the range where we have to consider the system \eqref{eq:DiffForm1}--\eqref{eq:DiffForm2}.

In the following, we assume that there exists initially a
characteristic cluster size, denoted as $L_0$ (and $L_n$ for the $n-$th following  LV cycle), around which the concentrations $c_{j}$ give a relevant contribution to the cluster distribution $\left\{ c_{j}\right\} _{j\in%
\mathbb{N}}$. 
More precisely, we
assume at $t=0$ when $v=w>\varepsilon$ that the initial concentrations are
given approximately by means of%
\begin{equation}
c_{j}(0) \simeq c(j,0) \simeq \frac{\varepsilon}{L_{0}}\varphi_{0}\left( \frac{j}{L_{0}}\right), \qquad j\in (0,\infty).
\label{eq:clustConc}
\end{equation}
We assume that $\varphi_{0}$ is a measure defined in $\left[
0,\infty\right) $ which {can contain a Dirac delta in $x=0$, modelling the fact that the mass contained in $c_1$ may be of the same order of magnitude as the mass contained in large clusters: we thus  define }
\begin{equation}\label{def:psi}
\varphi_0 (x)=\psi_0(x) + m_0 \delta (x),
\end{equation} 
where $\psi_0$ is a smooth function and supported in $(0,\infty)$ and $m_0 \geq 0$ the mass contained in the Dirac delta. 
Thanks to \eqref{eq:Mresc} and w.l.o.g., we consider $M=1$, which implies $\sum_{j=1}^{\infty
}jc_{j}=O(1)$ as well as $\sum_{j=1}^{\infty}c_{j}=\varepsilon$. For concentrations of the form \eqref{eq:clustConc} it 
follows that  we have $L_0=O(\frac{1}{\varepsilon})$. Assuming $\ep\ll 1$ in \eqref{eq:Mresc}, we typically have $L_0, L_n\gg1.$
We also impose 
\begin{equation}\label{as:phi0:2}
 \int_0^\infty \varphi_0(dx)=\int_0^\infty x \varphi_0(dx) =1,
\end{equation}
which defines uniquely $L_0$ by $L_0=\f{\sum j c_j}{\ep} \in [1,\f{1}{\ep})$. Note that we understand $\int_0^\infty=\int_{[0,\infty)}$, {\it i.e.} we include the Dirac mass at $0.$

The heuristic considerations in this section provides us with the qualitative behaviour of the clusters along one LV cycle: their size distribution is transported at the speed $\tilde{V}(t),$ and is diffused at the diffusion rate $\f{d(t)}{2}.$ 
We remark for instance that if $w>v$, we have that $\tilde{V}\left( t\right) >0$ {(see Figure \eqref{fig:vw4ter_p1} and Figure \eqref{fig:vw1_p1} below as part of a more detailed discussion}). In
particular the convective term in (\ref{eq:DiffApp}) shifts the distribution
of clusters towards larger values of $j$ {(see Figure \eqref{fig:sd4ter_p1} and Figure \eqref{fig:sd1_p1})}. 
Similar to Lemma~\ref{lem:scaling1} in subsection~\ref{subsec:LV}, we obtain the following lemma~\ref{lem:scaling2}, which collects  scaling properties concerning the transport and the total diffusion along one LV cycle.

\begin{lemma}\label{lem:scaling2}
With the notations and assumptions of Lemma~\ref{lem:scaling1}, let us define the characteristic curve $Y(t; E,\ep)$ and the total diffusion $D(E,\ep)$ by
\begin{equation}\label{def:Y}
Y\left( t;E,\varepsilon\right): =\int_{0}^{t}\left( w-v\right) \left(s;E,\varepsilon\right) ds\ , \qquad D(E,\ep):={\f{1}{2}}\int_0^{T(E,\ep)} d(s)ds.
\end{equation}
We have the relations
\begin{equation}Y\left( t;E,\varepsilon\right)=Y\left( Et;1,\frac{\varepsilon}{E}\right), \qquad D(E,\ep)=D(1,\f{\ep}{E}).
\label{eq:YScal}
\end{equation}
\end{lemma}
The next lemma gathers results concerning $D$, $Y$, $T$ and the average values of $v$ and $w$.
\begin{lemma}\label{lem:DYT}
Under the assumptions of Lemma~\ref{lem:LV1}, we have
\begin{equation}
T\left( 1,\varepsilon\right) \sim\frac{1}{\varepsilon^{2}}, \quad \int_0^T v(s)ds=\int_0^T w(s)ds=\ep T \sim \f{1}{\ep},\quad D\left( 1,\varepsilon\right) \sim\frac{1}{\varepsilon}. 
 \label{eq:Tasymp}
\end{equation}
Let us denote $Y(t)=Y(t; 1,\ep)$ and $D=D(1,\ep)$. We have the following behaviour of $Y$ and $D$ along one LV cycle.
\begin{itemize}
\item On $(0,t_2),$ $Y$ is increasing.
\item The main contribution to $D$ is due to $w$ in the interval $(t_1,t_2)$  {(see Figure~\eqref{fig:vw1_p1})}  and $v$ in $(t_3, t_4),$  {(see Figure~\eqref{fig:vw3_p1})}. Elsewhere, the contribution is negligible.
\item On $(t_2,t_3)$, $Y$ is almost constant, its largest value {$Y_{\max} \sim \f{1}{\ep}$} being reached at {$t_2 + \f{t_3 - t_2}{2}$,}
\item On $(t_3,t_5),$ $Y$ is decreasing.
\end{itemize}
\end{lemma}

\subsection{Main results and sketch of the dynamics phases}\label{subsec:sketch}

A main goal of this paper is to determine the evolution of the characteristic cluster length $L$
and its approximate cluster size distribution $\varphi$ over time; In particular for the $n-$th LV cycle, we will denote by $(L_n,\varphi_n)$ the corresponding typical length and size distribution. It will turn out that after
the variables $v,\ w$ have approximately followed several LV cycles \eqref{D8}, the energy $E$ changes and the concentrations $\left\{
c_{j}\right\} _{j\in\mathbb{N}}$ can be approximated by means of a formula
of the form (\ref{eq:clustConc}), for new choices of the length $L_n$ and the
function $\varphi_n.$

We now assume that the initial length $L_0$ is much smaller than $\frac {1}{%
\varepsilon}.$ 
Otherwise, if $L_0$ is of order $\frac{1}{\varepsilon}$ for the initial distribution of clusters, then the first of the four phases described below is skipped by the dynamics. 
We also remark that ${ L_0}\ll\frac{1}{\varepsilon}$ 
in \eqref{eq:Mresc} implies that the sum $\sum_{j=1}^{\infty}jc_{j}$  is
much smaller than $1$ for small $\varepsilon.$ Thus, $\left( v+w\right)
\simeq1$ for small $\varepsilon$ and we can assume $E=1$ and use the analysis carried out in Subsection~\ref{subsec:LV}, Lemmas~\ref{lem:LV1} to~\ref{lem:LV4}. 
Actually,
this approximation holds, not only for the initial time, but whenever we have $L\ll\frac{1}{\varepsilon}.$

In case that initially $L_0\ll\frac{1}{\varepsilon}$, the evolution of the concentrations towards the
equilibrium distribution undergoes four different phases. In each phase a different mechanisms governs the
evolution of the concentrations $\left\{ c_{j}\right\} _{j\in\mathbb{N}},\
v,\ w$ and the energy $E$ and
the characteristic length $L$.
We
describe the {leading} magnitudes associated to these phases as well as the
mechanism yielding the dissipation of energy $E.$ At the end of these four
phases the concentrations $\left\{ c_{j}\right\} _{j\in\mathbb{N}},\ v,\ w$
approach to an equilibrium. Notice that the approach to equilibrium takes place by means of an involved procedure and we do not see (at least currently) that it can be captured by means of a Ljapunov function.

\begin{itemize}
\item {\bf Phase I: The energy $E$ remains near a constant of order $1.$}
The period $T(E,\ep)$ of the LV oscillations  is of order $\frac{1}{\varepsilon^{2}}$%
{(see subsection~\ref{subsec:LV})}.  The evolution of the cluster distributions takes
place by means of a combination of an oscillatory motion of the cluster
distribution towards larger values of $j$ and backwards to clusters with $j$
of order one, combined with a spreading of the cluster concentration
distribution in the space of cluster sizes {(see subsection~\ref{subsec:sizedistrib})}. The distribution spreads an
amount of order $\frac{1}{\sqrt{\varepsilon}}$ in the space of cluster sizes
in each LV cycle, and if we assume that initially $L=L_{0}\approx1$ we require a number of
cycles $n$ of order $\frac{1}{\varepsilon}$ to arrive to $L=L_{\mathrm{fin}}\approx%
\frac {1}{\varepsilon}$ {(see subsection~\ref{subsec:IandII:overall}, Prop.~\ref{prop:phasesIandII})}. {From this point on, the order of magnitude of $L$  remains in the order of $\f{1}{\ep}$.}

\item {\bf Phase II: Decay of the energy $E$ from order $1$ to order $\ep$.} The function $\varphi$ approximating
the concentrations $c_{j}$ has a nontrivial evolution (Corollary~\ref{coroll:sizedistrib} and Prop.~\ref{prop:UM}). The energy $E$
decreases from a value close to $1$ to $E_{\mathrm{fin}}\approx\varepsilon.$ The
distribution of clusters evolves by means of a combination of an oscillatory
displacement towards large cluster sizes combined with their spread in the
space of cluster size as in Phase I (Lemma.~\ref{lem:I:sizedistrib}). The period $T$ of the LV oscillations
is of order $\frac{E}{\varepsilon^{2}}$ (Lemmas~\ref{lem:scaling1} and~\ref{lem:DYT}). {The number of LV oscillations
taking place during this phase is of order $n\approx\frac{\log(\f{1}{\ep})}{\varepsilon}$ (see subsection~\ref{subsec:IandII:overall}, Prop.~\ref{prop:phasesIandII}).}

\item {\bf Phase III: Decay of the energy $E$ from  order $\varepsilon$ to  order $\varepsilon^{3}.$} The concentrations of
clusters $c_{j}$ remain nearly constant for large sizes $j,$ but the
concentrations $c_{j}$ with $j$ of order one oscillate during the LV\
oscillations of the concentrations $v,\ w$ {(see section~\ref{sec:PhaseIII}, Prop.~\ref{prop:PhaseIII}).} The period of the oscillations $%
T$ {remains in the order of $\f{1}{\ep}$.} The number of LV oscillations $%
n $ taking place during this phase is of order {$n\approx\frac{-\log({\ep^2})}{\varepsilon }$ (see subsection~\ref{subsec:PhaseIII:end}). 
The oscillations for the cluster concentrations $j$ of order one are progressively damped and finally become negligible (see subsection~\ref{subsec:PhaseIII:end}, {Prop.~\ref{prop:endPhaseIII}}). }

\item {\bf Phase IV: Final trend to the equilibrium.} The energy $E$
remains of order $\varepsilon^{3}.$ The concentrations $c_{j}$ evolve
towards their equilibrium value and their values can be approximated using a
nonlinear second order parabolic equation {(see section~\ref{sec:PhaseIV}, Prop.~\ref{prop:PhaseIV})}. During this phase, all the
concentrations $c_{j}$ with $j$ of order one can be approximated by a
single function {$C\left(0^+, t\right) $} which changes in the same time scale in
which the whole distribution of concentrations $\left\{ c_{j}\right\} _{j\in%
\mathbb{N}}$ approach to equilibrium. {The oscillations being negligible, the total duration for this last phase is in the order of $\f{1}{\ep^3}.$}
\end{itemize}

\begin{figure}
     \centering
     \begin{subfigure}[b]{0.49\textwidth}
         \centering
         \includegraphics[width=\textwidth]{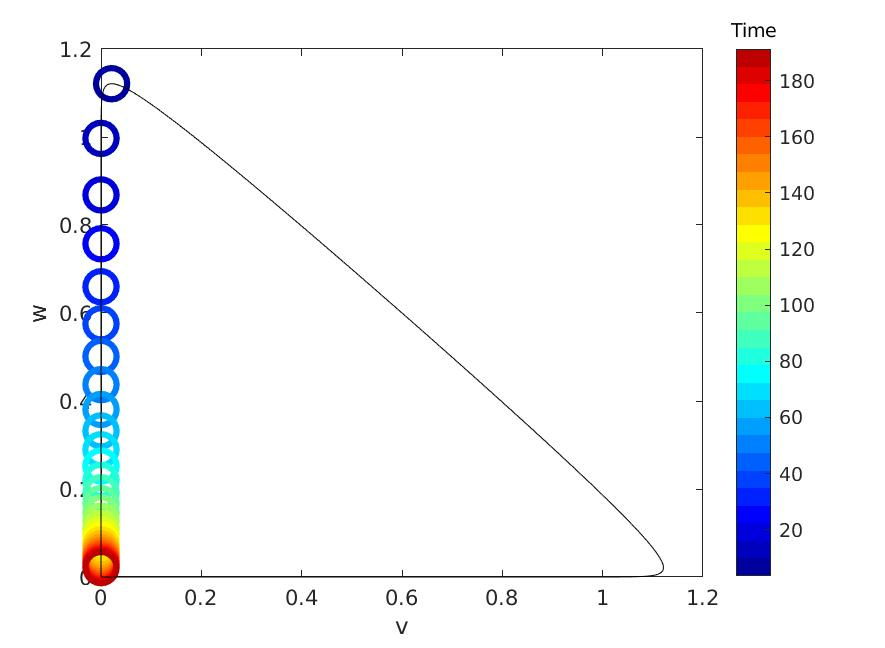}
         \caption{Phase portrait of $(v,w)$ for \\ $t\in [t_1,t_2]=[3,191]$, cf. Lemma \ref{lem:LV2}}
         \label{fig:vw1_p1}
     \end{subfigure}
     \hfill
     \begin{subfigure}[b]{0.49\textwidth}
         \centering
         \includegraphics[width=\textwidth]{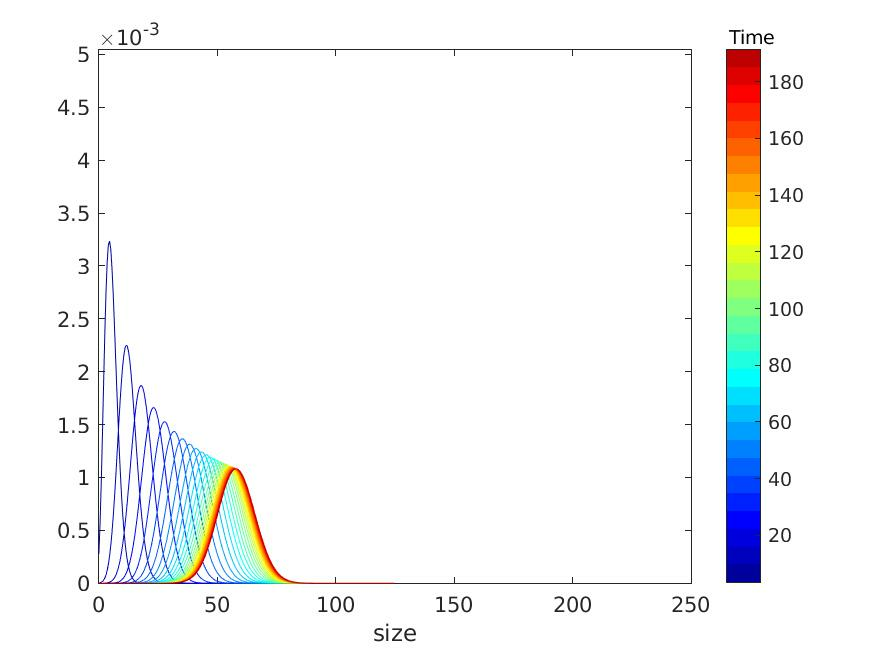}
         \caption{Evolution of cluster distribution $c(j,t)$ for $t\in [t_1,t_2]=[3,191]$}
         \label{fig:sd1_p1}
     \end{subfigure}
     \hfill
     \begin{subfigure}[b]{0.49\textwidth}
         \centering
         \includegraphics[width=\textwidth]{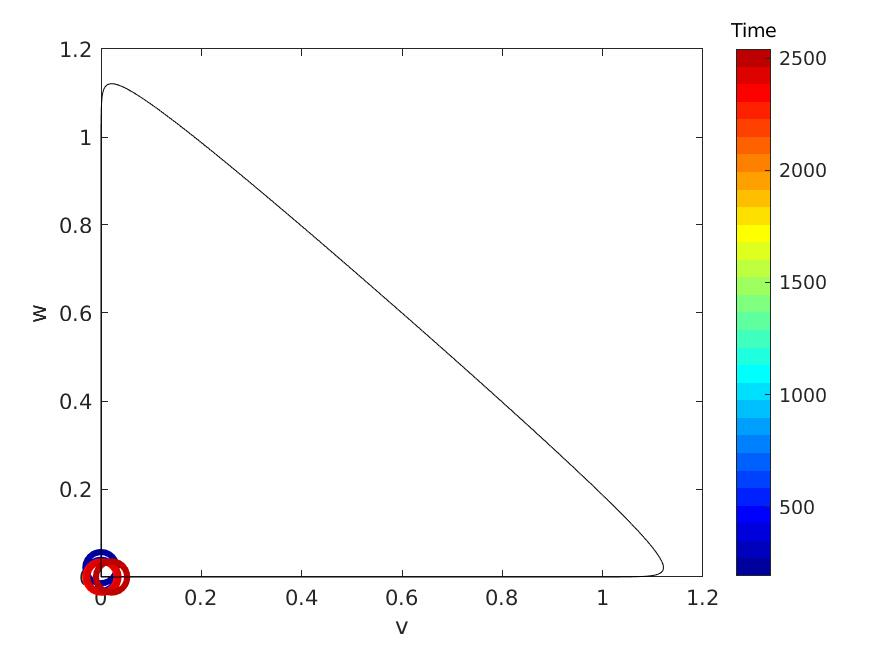}
         \caption{Phase portrait of $(v,w)$ for \\ $t\in [t_2,t_3]=[191,2541]$, cf. Lemma \ref{lem:LV3}}
         \label{fig:vw2_p1}
     \end{subfigure}
     \begin{subfigure}[b]{0.49\textwidth}
         \centering
         \includegraphics[width=\textwidth]{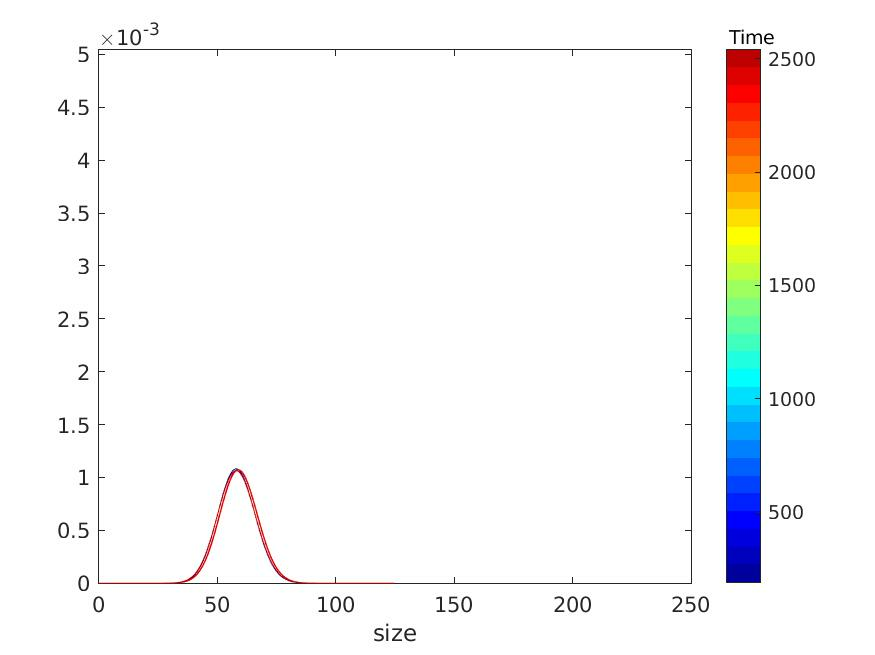}
         \caption{Evolution of cluster distribution $c(j,t)$ for $t\in [t_2,t_3]=[191,2541]$}
         \label{fig:sd2_p1}
     \end{subfigure}
     \begin{subfigure}[b]{0.49\textwidth}
         \centering
         \includegraphics[width=\textwidth]{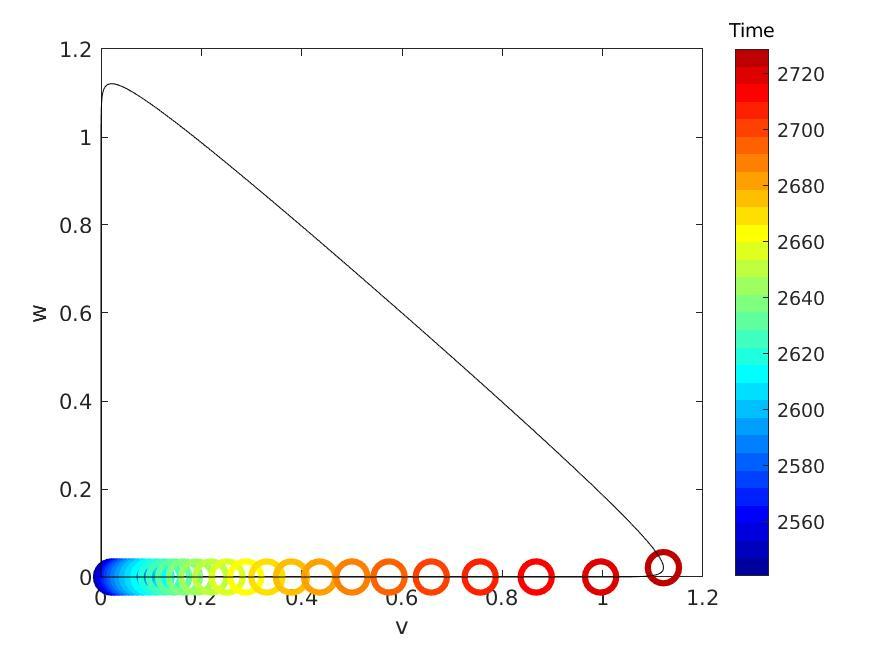}
         \caption{Phase portrait of $(v,w)$ for\\  $t\in [t_3,t_4]=[2541,2729]$, cf. Lemma \ref{lem:LV4}}
         \label{fig:vw3_p1}
     \end{subfigure}
     \begin{subfigure}[b]{0.49\textwidth}
         \centering
         \includegraphics[width=\textwidth]{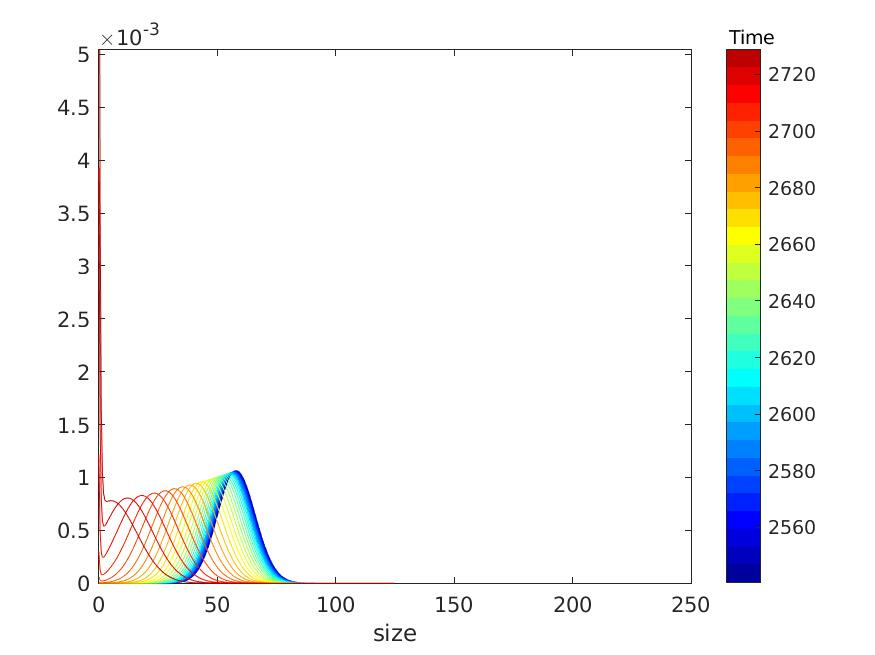}
         \caption{Evolution of cluster distribution $c(j,t)$ for $t\in [t_3,t_4]=[2541,2729]$}
         \label{fig:sd3_p1}
     \end{subfigure}
        \caption{Phase portraits and snapshots of the cluster size distribution for one cycle of the LV system during Phase I in the case $\ep=0.0212$. See Appendix \ref{sec:app_num} for the details about the numerical simulations.}
        \label{fig:size_distr_p1}
\end{figure}

\section{Phases I and II: from one LV cycle to the next} \label{sec:phaseIandII}

As pointed out in subsection~\ref{subsec:sizedistrib}, the dynamics of the $\left\{
c_{j}\right\} _{j\in \mathbb{N}}$ is a discrete-in-size convection-diffusion process. 
We now describe {in detail} the evolution of the cluster concentrations $\left\{
c_{j}\right\} _{j\in \mathbb{N}}$ during one LV cycle of the monomer
concentrations $\left( v,w\right).$
We will illustrate how this mechanism leads to an approximation of the cluster distribution in terms of the ansatz \eqref{def:psi} consisting of an aggregated mass for the smallest cluster and a smooth profile function for large cluster sizes during the Phases I and II. For this reason, we assume in all of this section that 
the energy satisfies $%
\varepsilon \ll E\lesssim 1,$ 
hence all the lemmas of subsections~\ref{subsec:LV} and~\ref{subsec:sizedistrib} are valid. We use them in order to find out what is the decay of energy along one LV cycle, and how to go from a cycle $n$ to a cycle $n+1$. 

Figure~\ref{fig:size_distr_p1} plots phase portraits of the monomeric concentrations $(v,w)$ and snapshots of the approximate cluster size distribution $c(j,t)$ on time intervals $[t_1,t_2]$, $[t_2,t_3]$ and $[t_3,t_4]$ according to Lemmas \ref{lem:LV2}, \ref{lem:LV3} and \ref{lem:LV4}. It is for didactic reasons that Figure~\ref{fig:size_distr_p1} starts with Lemma \ref{lem:LV2} describing stage two: The corresponding evolution of the size distribution in Figure \ref{fig:sd1_p1} shows very nicely the convective transport to larger cluster sizes ($\tilde V \simeq w$ in \eqref{eq:DiffApp}--\eqref{eq:DiffCoef} due to $v$ being small on $[t_1,t_2]$) combined with the diffusive spreading occurring in this stage according to $d\simeq w$. Convection and diffusion slow down as $w$ decays in the phase portrait Figure \ref{fig:vw1_p1}.

Next in Figure \ref{fig:vw2_p1}, the monomer concentrations $(v,w)$ are small compared to $\ep$ on the interval $[t_2,t_3]$ of Lemma \ref{lem:LV3}. Hence, the cluster size distribution undergoes only negligible convection and diffusion in Figure \ref{fig:sd2_p1}, making $[t_2,t_3]$ the longest time interval within one LV cycle.

The dynamics picks up again in stage four described in Lemma \ref{lem:LV4} and plotted in Figure \ref{fig:vw3_p1}, which sees the growth of $v$ on the time interval $[t_3,t_4]$ while $w$ remains small. 
 Accordingly, Figure \ref{fig:sd3_p1}
shows convection of the size distribution towards smaller clusters combined with diffusive spreading. We highlight that the resizing of cluster towards smaller sizes is 
naturally limited by the smallest cluster size $c_1$ where we observe the formation of an aggregate. 

\begin{figure}[htb]
     \begin{subfigure}[b]{0.49\textwidth}
         \centering
         \includegraphics[width=\textwidth]{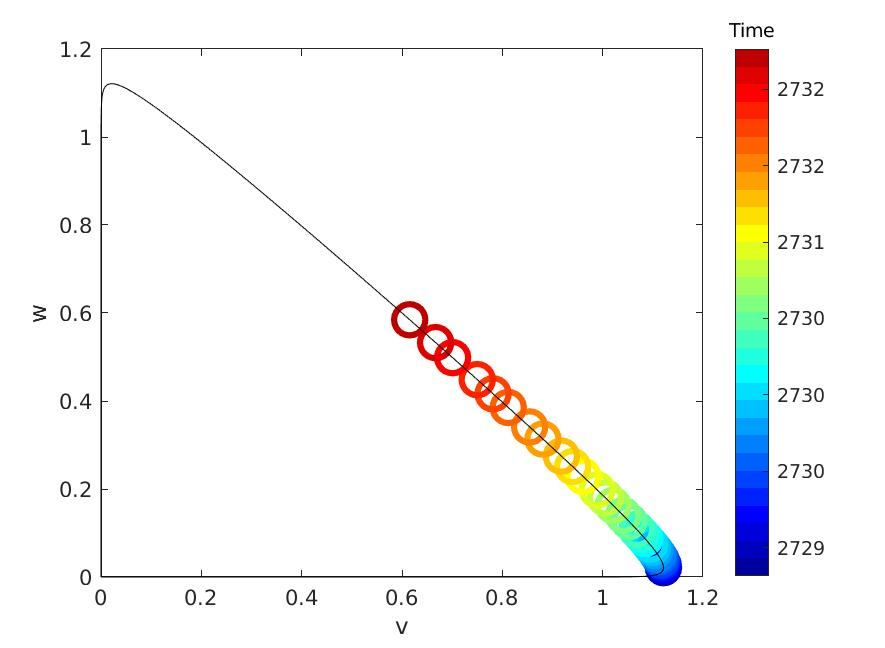}
         \caption{Phase portrait of $(v,w)$ for  $t\in [t_4,t_5=T]=[2729,2732]$, cf. Lemma \ref{lem:LV5}}
         \label{fig:vw4bis_p1}
     \end{subfigure}
     \begin{subfigure}[b]{0.49\textwidth}
         \centering
         \includegraphics[width=\textwidth]{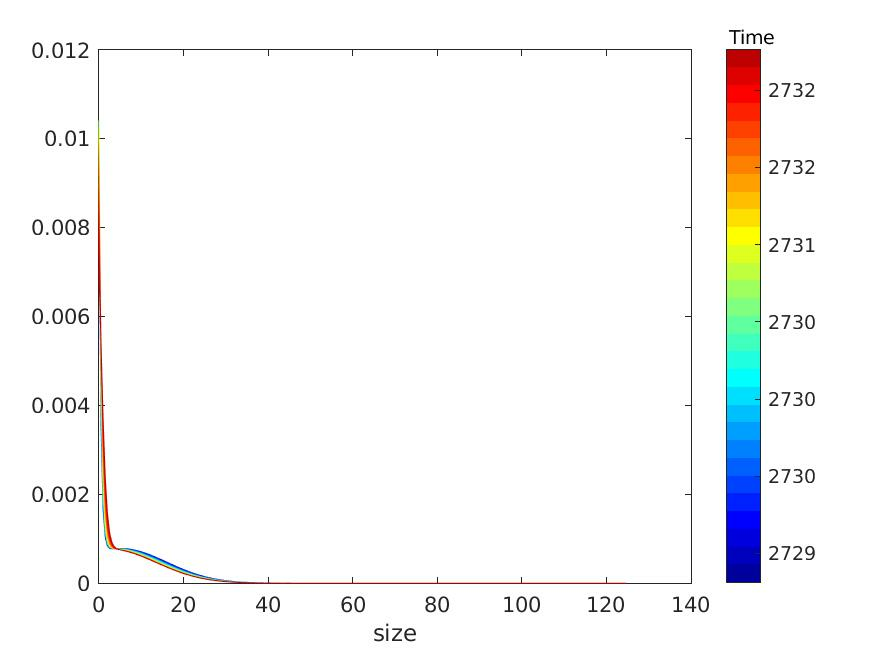}
         \caption{Evolution of cluster distribution $c(j,t)$ for $t\in [t_4,t_5=T]=[2729,2732]$}
         \label{fig:sd4bis_p1}
     \end{subfigure}
     \begin{subfigure}[b]{0.49\textwidth}
         \centering
         \includegraphics[width=\textwidth]{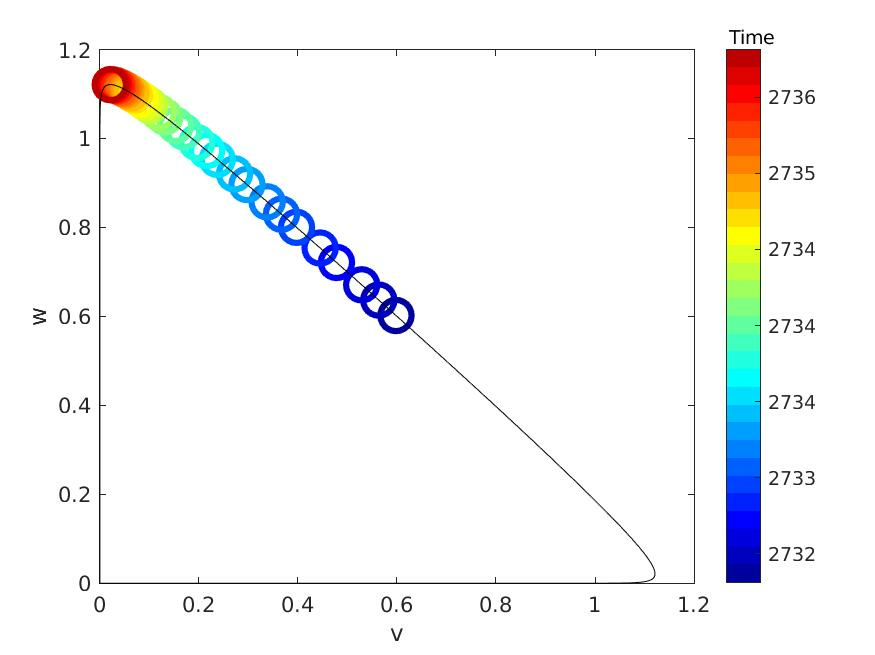}
         \caption{Phase portrait of $(v,w)$ for $t\in [T,T+t_1]=[2733,2736]$, cf. Lemma \ref{lem:LV1}}
         \label{fig:vw4ter_p1}
     \end{subfigure}
     \begin{subfigure}[b]{0.49\textwidth}
         \centering
         \includegraphics[width=\textwidth]{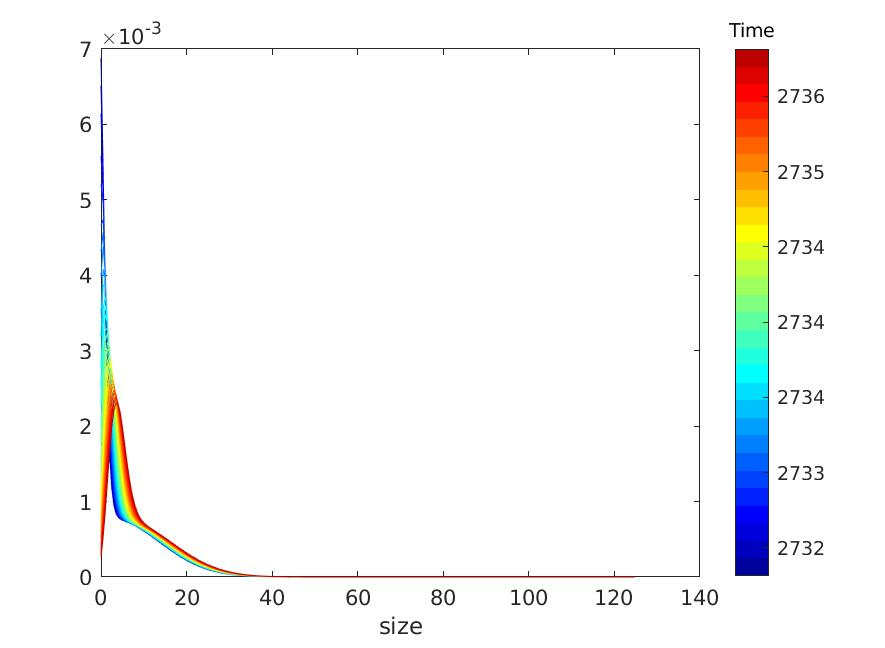}
         \caption{Evolution of cluster distribution $c(j,t)$ for  $t\in [T,T+t_1]=[2733,2736]$}
         \label{fig:sd4ter_p1}
     \end{subfigure}
     \caption{Phase portraits and snapshots of the cluster size distribution during the fastest region of the LV phase portrait in Phase I. See Appendix \ref{sec:app_num} for the details about the numerical simulations.}
        \label{fig:size_distr_p1_bis}
\end{figure}

Figure~\ref{fig:size_distr_p1_bis}
shows in subfigure \ref{fig:sd4bis_p1} that the aggregation at the smallest cluster is further enhanced on the time interval $[t_4,t_5]$ of stage five, Lemma~\ref{lem:LV5}, during which the phase portrait of $(v,w)$ reaches the end time $T$ of a LV cycle in Figure~\ref{fig:vw4bis_p1}.
Finally, Figures~\ref{fig:sd4ter_p1} and \ref{fig:sd4ter_p1} show the first 
stage of a LV cycle as in Lemma~\ref{lem:LV1}.
We remark again that we decided to plots this first time interval last, since the evolution of the cluster distribution during the first stage when the aggregate of the smallest cluster begins to convect towards larger clusters is difficult to read without knowing the previous discussion. 

Combining the time scales characterised in Lemma~\ref{lem:DYT} with the above discussion 
of Figures~\ref{fig:size_distr_p1} and \ref{fig:size_distr_p1_bis} suggests to separate the dynamics of each LV cycle into two principal steps: 1) a "pure transport" step, during which the diffusion is negligible - this is the interval $[0, t_1]$ and symmetrically $[t_4,t_5]$; we could even extend these intervals to the intermediate time points $[0,t_{1,2}]$ and $[t_{3,4},t_5]$ (see the proof of Lemma~\ref{lem:LV2} in the appendix where these times are defined); 2) a "transport and diffusion" step, where both transport and diffusion are important, which is given by the symmetric intervals $[t_{1,2},t_2]$ and $[t_3,t_{3,4}]$. The remaining time interval $[t_2,t_3]$ is a kind of lag time where almost nothing happens, neither transport nor diffusion, due to the {exponential} smallness of both $v$ and $w$.

The following 
proposition characterises one iteration step of the size distribution of clusters as plotted in Figures~\ref{fig:vw4bis_p1} and \ref{fig:sd4bis_p1} (by taking the final depicted time points) from a first appearance (which we shall call $t=0$ rather than $t=T$ for the sake of simplicity) to its second appearance after one period $T$.

\begin{proposition}[Initial cycle and iterative formula]
Let  $M=1,$  $\ep \ll E_0 \leq 1,$ and $v(0)=w(0)>\ep$ such that~\eqref{D5} is satisfied. Let $1\ll L_0\lesssim \f{1}{{\ep}}$, and $c_j(0)$ defined by~\eqref{eq:clustConc} for some measure $\varphi_0$ satisfying~\eqref{def:psi}--\eqref{as:phi0:2}. 
Let $T_0:=T(E_0,\ep)$ be the  time period of the unperturbed LV system~\eqref{D8}, $D_0:=D(E_0,\ep)$ the diffusion defined by~\eqref{def:Y}, and $\sigma_0:=\f{D_0}{(L_0)^2}$. 

\noindent
Defining $(\varphi_1,L_1)$ such that $c_j(T_0)=\f{\ep}{L_1}\varphi_1(\f{j}{L_1})$ and  $\varphi_1$ satisfies~\eqref{def:psi}--\eqref{as:phi0:2}, we have
\begin{align}
\varphi_{1}\left( x\right) & \approx \frac{L_{1}}{L_{0}}\chi_{\left(
0,\infty\right) }\left( x\right) \int_{0}^{\infty}G\left( \frac{L_{1}}{%
L_{0}}x-\eta;\sigma_0^2%
\right) \varphi_{0}\left( \eta\right) d\eta  \notag \\
& +\delta\left( x\right) \int_{-\infty}^{0}d\zeta\int_{0}^{\infty}G\left(
\zeta-\eta;\sigma_0^2%
\right) \varphi_{0}\left( \eta\right) d\eta,  \label{eq:Coninit}
\end{align}
where $G$ is the fundamental solution for the heat equation:
\begin{equation}
G\left( \xi;s\right) =\frac{1}{\sqrt{4\pi s}}\exp\left( -\frac{\xi^{2}}{4s}%
\right)  \label{eq:FundSol}
\end{equation}
and $L_{1}$ is given by
\begin{equation}
\begin{aligned}
\frac{L_{1}}{L_{0}}
&=\int_{0}^{\infty}\varphi _{0}\left(
\eta\right) d\eta\int_{0}^{\infty}xG\left( x-\eta;\sigma_0^2\right) dx
\\
&=1+\int_0^\infty x \diff x \int_0^\infty G\left(x+\eta;\sigma_0^2\right)\varphi_0(\eta)\diff \eta.
\end{aligned}
\label{eq:Coninit2}
\end{equation}
\label{prop:Ln}
\end{proposition}

\begin{remark} If we had that $L_{0}$ is of order one (i.e. the initial
size concentration  is concentrated in cluster sizes $j$ of order
one), then the solution during a LV cycle would spread the
cluster concentrations to regions of size $\sqrt{\frac{E}{\varepsilon}}\gg 1$. Therefore, in a single
cycle the width of the region in which the clusters are concentrated would
become much larger than one; this explains why we have assumed in Prop~\ref{prop:Ln}
that $L_{0}\gg1.$ Notice that with this assumption, if there is a fraction
of clusters in the concentrations $c_{j}$ with $j$ of order one, we should
include a Dirac mass at in $\varphi_{0}$ at $x=0$ containing this fraction
of clusters. 
\end{remark}

\begin{proof}[Heuristic proof of Proposition \ref{prop:Ln}] We provide a heuristic proof of Proposition \ref{prop:Ln} divided into the following steps: In subsection~\ref{subsec:distrib}, we analyse the ODE system as being close to a continuous convection-diffusion equation: its dynamics along one LV cycle is thus described in Lemma~\ref{lem:I:sizedistrib} below. This result is used in subsection~\ref{subsec:iter} to obtain~\eqref{eq:Coninit}--\eqref{eq:Coninit2} by solving the diffusion equation during one cycle with the help of the fundamental solution to the heat equation.
\end{proof}

It is clear that as long as the assumptions of the Proposition \ref{prop:Ln} remain valid,  we can apply it to successive time periods $T_n$ themselves defined by successive energies $E_n$. Hence, the following corollary holds as long as $E_n \gg \ep.$ 

{\begin{corollary}
Under the assumptions and notations of Prop.~\ref{prop:Ln}, the same result applies between two successive cycles $n$ and $n+1$ as long as $E_n \gg \ep$. By denoting accordingly $\varphi_n$, $L_n,$ $m_n$, $\psi_n$ and $\sigma_n^2$, we have the following iterative formulae: 
\label{coroll:sizedistrib}
\begin{align}
\psi_{n+1}\left( x\right)  &=\frac{L_{n+1}}{L_{n}}\int_{0}^{\infty}G\left(
\frac{L_{n+1}}{L_{n}}x-\eta,{\sigma_n^2}\right) \psi_{n}\left( \eta\right) d\eta\notag
\\ &\qquad +m_{n}\frac{L_{n+1}%
}{L_{n}}G\left( \frac{L_{n+1}}{L_{n}}x,\sigma_n^2\right)  \notag, \\
m_{n+1} & =\int_{-\infty}^{0}d\zeta\int_{0}^{\infty}G\left( \zeta-\eta,{
\sigma_n^2} \right)
\psi_{n}\left( \eta\right) d\eta+\frac{m_{n}}{2}  \notag, \\
\frac{L_{n+1}}{L_{n}}  &=1+\int_{0}^{\infty}xdx\int_{0}^{\infty}G\left(
x+\eta,{
\sigma_n^2}\right)
\psi_{n}\left( \eta\right) d\eta \notag\\
&\qquad+m_{n}\int_{0}^{\infty}xG\left( x,{
\sigma_n^2}\right) dx.
\label{eq:SemGroup}
\end{align}

\end{corollary}}

\subsection{An approximate transport and diffusion equation.}
\label{subsec:distrib}
As a first step to establish Proposition~\ref{prop:Ln}, let us study the dynamics of the cluster size distribution during one cycle. We use equations \eqref{eq:DiffForm1}--\eqref{eq:DiffForm2} as written in subsection~\ref{subsec:sizedistrib} and interpret them as a discretised drift-diffusion equation~\eqref{eq:DiffApp}. 
Moreover, it will be useful to rescale time in \eqref{eq:DiffApp} and consider 
the following convection-diffusion equation 
\begin{equation}
\frac{\partial c}{\partial \tau }\left( j,\tau \right) +\frac{2\left(
w-v\right) }{\left( w+v\right) }\frac{\partial c}{\partial j}\left( j,\tau
\right) =\frac{\partial ^{2}c}{\partial j^{2}}\left( j,\tau \right), \qquad  \tau(t)=\int_0^t \f{d(s)}{2}\diff s, 
\label{eq:DiffAppA}
\end{equation}
where we recall that 
$d = w+v$ and Lemmas~\ref{lem:scaling2}, \ref{lem:DYT}, i.e. 
\begin{equation*}
D(E,\ep):={\f{1}{2}}\int_0^{T(E,\ep)} d(s)ds \sim \frac{E}{\ep},
\end{equation*}
which is a consequence of $d = O(\ep)$ for those part of the LV cycle which contribute to the integral the most. Hence, one can image the time-rescaling as roughly $\tau = O(\ep t)$.
The time-rescaled equation~\eqref{eq:DiffAppA} is a key step in understanding these dynamics.
\begin{lemma}\label{lem:I:sizedistrib}
Under the assumptions and notations of Prop.~\ref{prop:Ln}, we define $\Delta t:= \sqrt{\f{1}{E_0 \ep}}$. 
During the time interval $[0,T_0]$, the size distribution $c_j(t)$ approximately evolves according to the following dynamics:
\begin{enumerate}
\item \label{step1:drift} Over the time interval $[0, \Delta t]$, transport dominates diffusion: {\it{i.e.}} $c_j(t)\approx c_{j-Y(t)} (0)$ is transported along $Y(t)>0$ defined by~\eqref{def:Y}, and moves towards larger sizes.
\item \label{step2:diff} Over the time interval $t\in [\Delta t, T_0-\Delta t],$ $c_j(t)$ is close to the solution to~\eqref{eq:DiffAppA}, hence $c_j(t) \approx u(x,\tau(t))$ with  
$j=x+Y(t),$ $\tau$ and $d$ defined by~\eqref{eq:DiffCoef}, and $u$ solution to the pure diffusion equation
\begin{equation}
\label{eq:DiffAppB}
\f{\partial u}{\partial \tau}  (x,\tau) =\f{\partial^2 u}{\partial x^2 } (x,\tau),\qquad E_0\Delta t \leq \tau \leq D_0-E_0\Delta t.
\end{equation}
During this time interval, the total diffusion is in the order of $D_0\approx \f{E_0}{\ep}$: the size distribution diffuses in a range $\Delta x\approx \sqrt{\f{E_0}{\ep}}.$ 
\item \label{step3:drift} In the time interval $t\in [T_0-\Delta t, T_0]$, transport dominates diffusion again. We have $c_j(t)\approx c_{j-Y(t)+Y(T_0-\Delta t)} (T_0-\Delta t)$ and $Y(t)<0$. The mass moving again towards smaller sizes, for $j< \Delta x,$ the mass accumulates in $c_1.$ 
\end{enumerate}

\end{lemma} 
\begin{proof}
W.l.o.g. and for the sake of simplicity, we set $E_0=1$  as the general case follows by the change of variables detailed in Lemmas~\ref{lem:scaling1} and~\ref{lem:scaling2}. 
Let us begin by proving the central behaviour of Step~\ref{step2:diff} given by~\eqref{eq:DiffAppB}. To do so, we use~\eqref{eq:DiffAppA} and identify -- with some abuse of notations -- 
its solution $c(j,{\tau})$ with $c(x,{\tau})$.
We have $\tau(T(E_0,\ep))=D_0$ defined in~\eqref{def:Y}, and it follows that~\eqref{eq:DiffAppA} becomes~\eqref{eq:DiffAppB}. 

The approximation~\eqref{eq:DiffAppB} is valid on the whole size space $(0,+\infty)$ {(and even $(-\infty,+\infty)$ in the sense that we do not need any boundary condition to solve the heat equation)} if the mass is contained in large enough clusters, {\it i.e.} $Y(t)$ large enough, which is true, thanks to Lemma~\ref{lem:DYT}, on an interval $[\Delta t,T_0-\Delta t]$ for some $\Delta t \in (t_1,t_2)$ to be specified. In this case, we expect to have $c_1$ very small, so that the equations for $v$ and $w$ \eqref{eq:BD2}--\eqref{eq:BD2a} may be approximated by the unperturbed  LV system (\ref{D8}), and the approximation of $Y$ done in Lemma~\ref{lem:DYT} is {\it a posteriori} validated.

During this time interval, the range of values of cluster sizes in
which the concentrations $c$ are relevant spreads according to the equation (%
\ref{eq:DiffAppB}). Notice that this spreading process can be thought as
diffusion in the space of cluster sizes, although this diffusion is not
related to any physical diffusion in the physical space, but it is just a
mathematical description that allows to compute the evolution of the
concentrations in the space of cluster sizes.  
Let us denote the heat equation semigroup $\exp\left( \tau
\partial_{x}^{2}\right) $ as $S\left( \tau\right) .$ 
Assuming $\Delta t \ll D_0$ (this will be confirmed {\it a posteriori}), the total time  for diffusion during $[\Delta t,T_0-\Delta t]$ is thus {in the order of $\tau(T_0)=D_0\sim {\f{E_0}{\ep}}$, see Lemma~\ref{lem:scaling2} and~\ref{lem:DYT}. Due to the properties of the heat semigroup,  the characteristic length for diffusion is thus of order $\sqrt{\frac{E_0}{\varepsilon}} \ll D_0.$}

This provides us with the order of magnitude for $\Delta t$: at the final stage $[T_0-\Delta t,T_0]$ of the
LV cycle, when the 
concentrations front
returns to values of $j$ of order
one, the concentrations $c_j$ for $j=O(1)$ got diffused over a region of
order $\sqrt{\frac{E_0}{\varepsilon}}$ cluster sizes (see Figures \ref{fig:sd3_p1} and \ref{fig:sd4bis_p1}).
The transport speed $\f{\diff Y}{\diff t}=w(t)-v(t)$ in this region is of order {$-{E_0}$} by Lemmas~\ref{lem:scaling1} and~\ref{lem:LV4}, thus the time interval during which the size region $\Bigl[1, \sqrt{\frac{E_0}{\varepsilon}}\Bigr]$ hits $c_1$ is of order 
$$
\Delta t=\f{1}{E_0} \sqrt{\f{E_0}{\ep}}=\sqrt{\f{1}{E_0\ep}}.
$$

Concerning the dynamics of the clusters in the boundary layer during the arrival of
the concentrations front, the mass diffused accumulates in $j=1$
(see Figures \ref{fig:sd3_p1} and \eqref{fig:inner_layer_1_p1}).
For clusters of order one, we can approximate the concentrations $c_j$ by means of the solutions of 
\begin{align}
\frac{dc_{j}}{d\tau} & =\left( c_{j+1}-c_{j-1}\right) +\left(
c_{j-1}+c_{j+1}-2c_{j}\right)  ,\ \ j\geq2  \label{eq:BoundLay1} \\
\frac{dc_{1}}{d\tau} & =2\biggl( c_{2}-\frac{\varepsilon}{E_0}c_{1}\biggr)
\label{eq:BoundLay2}.
\end{align}
{To conclude, it remains to prove that the diffusion may be neglected in the intervals $[0,\Delta t]$ and $[T_0-\Delta t, T_0]$: this is due to the fact that the time interval is much smaller than the total diffusion $D_0$. The approximations in Steps~\ref{step1:drift} and~\ref{step3:drift} by a pure transport equation are thus valid.}
\end{proof}

\subsection{From the initial cycle to the next \label{subsec:iter}}

In this section, we (heuristically) obtain~\eqref{eq:Coninit} and~\eqref{eq:Coninit2} of Proposition~\ref{prop:Ln}. These two formulae may be viewed as linking a prototypical size distribution at a first time point $T_0$ (or in general $T_n$, for the $n-$th cycle) to the size distribution after one more cycle at time $T_1$ (respectively $T_{n+1}$) provided that the initial energy $E_0$ (or $E_n$) of this period satisfies $\ep\ll E_0, E_n\ll 1.$ 

We remark that as a first prototypical size distribution at $T_0$, we consider the size distribution obtained after finishing an initial LV cycle, at the time point when the values
of $w\ $and $v$ cross the half-line $\left\{ w=v>\varepsilon\right\}$ and we already 
have aggregated a positive amount of mass in clusters with size $j$ of
order one, and more precisely in the smallest cluster size $j=1$ as shown in Figure \eqref{fig:sd4bis_p1} (see also Figure \eqref{fig:inner_layer_2_p1}) below. This mass comes from two sources, first, a possible contribution of the original size-distribution {$\varphi_0(1/L_0)$} which has been transported  back and forth - and diffused, hence lowered - and secondly, to the diffusion of small clusters $\varphi_0(j/L_0)$ with $j$ within a width $\sqrt{\f{E}{\ep}}\gg 1,$ which will accumulate to some fraction on $j=1$ during the transport back to the origin.

\begin{figure}[ht!]
     \centering
     \begin{subfigure}[b]{0.49\textwidth}
         \centering
		 \includegraphics[width=\textwidth]{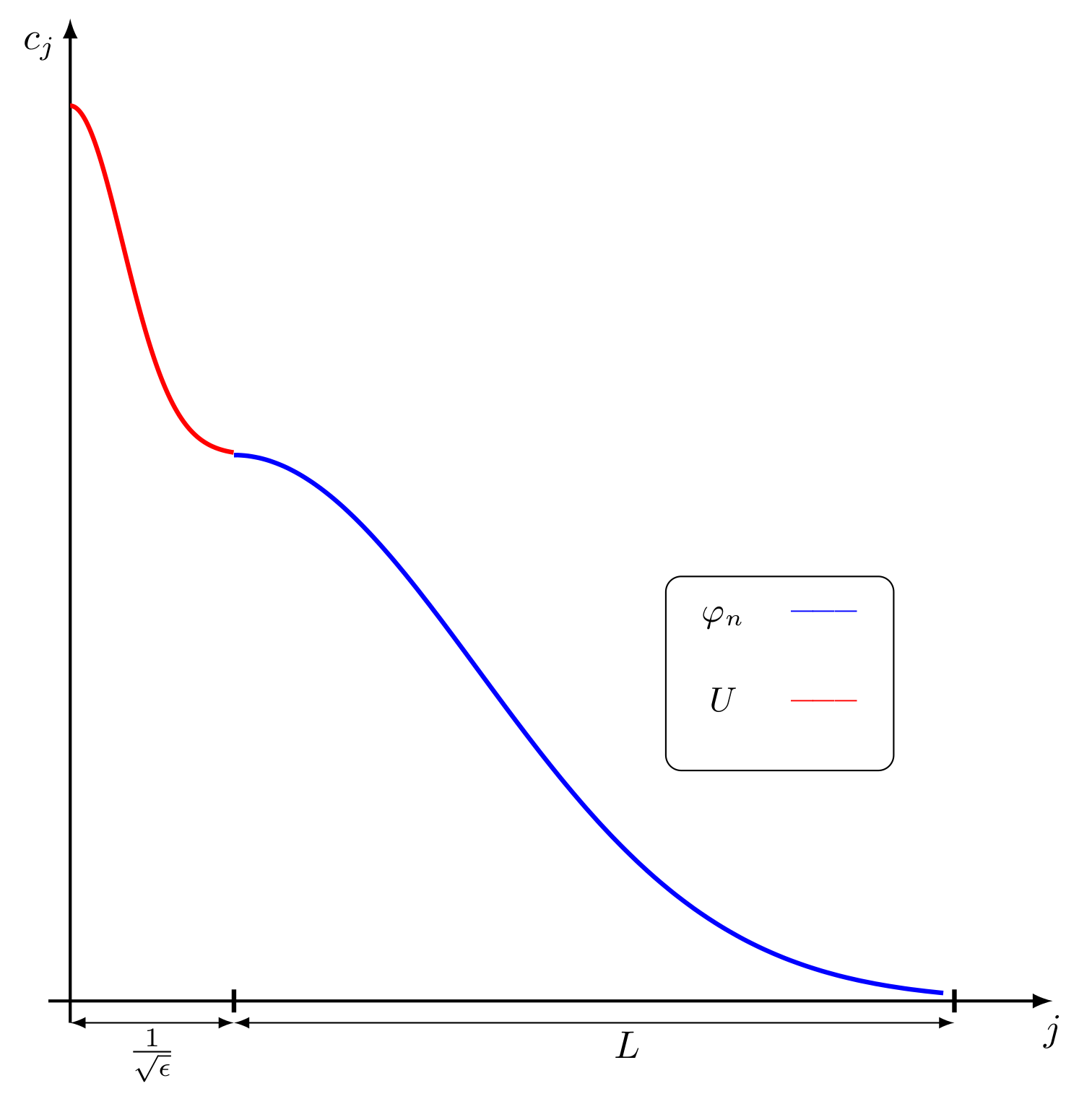}
         \caption{Size distribution when $v>w$, $v>\epsilon$}
         \label{fig:inner_layer_1_p1}
     \end{subfigure}
     \hfill
     \begin{subfigure}[b]{0.49\textwidth}
         \centering
         \includegraphics[width=\textwidth]{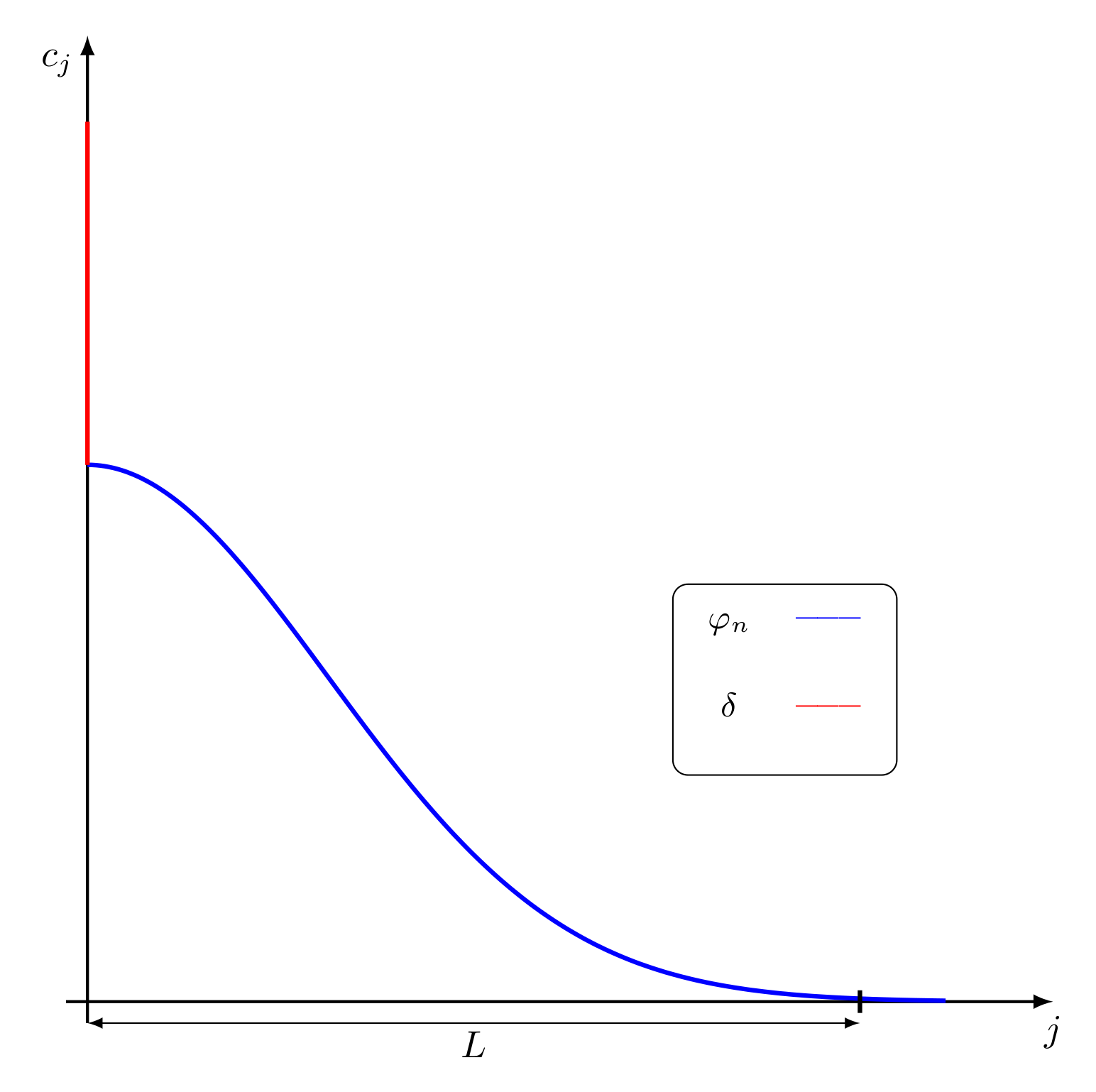}
         \caption{Size distribution when $v=w$, $v>\epsilon$}
         \label{fig:inner_layer_2_p1}
     \end{subfigure}
     \hfill
     \begin{subfigure}[b]{0.49\textwidth}
         \centering
         \includegraphics[width=\textwidth]{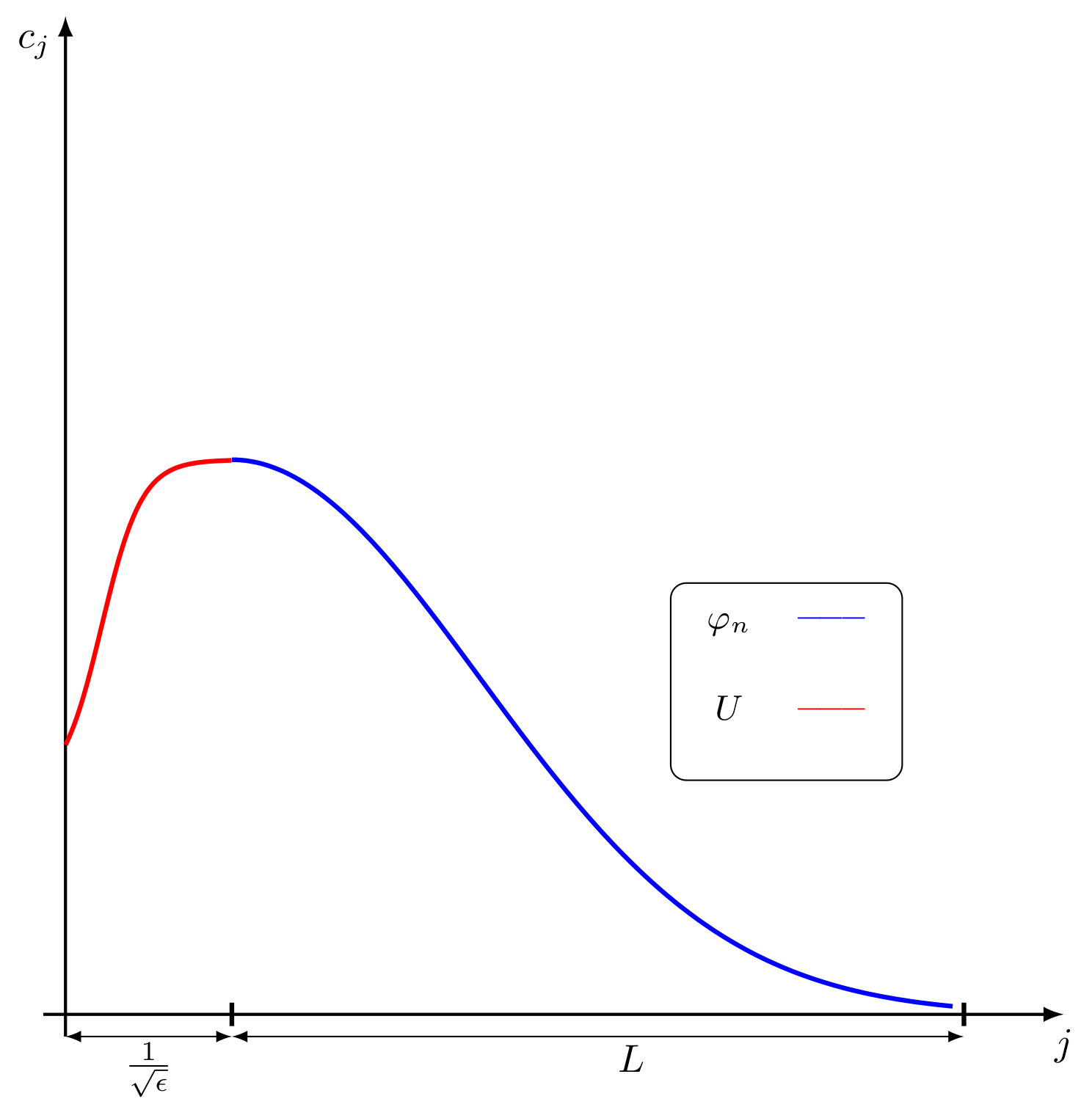}
         \caption{Size distribution when $v<w$, $w>\epsilon$}
         \label{fig:inner_layer_3_p1}
     \end{subfigure}
        \caption{Scheme of the evolution of the cluster sizes of order one.}
        \label{fig:inner_layer_p1}
\end{figure}

We remark that the conditions in (\ref{eq:Mresc}) imply that%
\begin{equation*}
1=v+w+\varepsilon L_{0}\int_{0}^{\infty}x\varphi_{0}\left( x\right) dx\ \ ,\
\ \int_{0}^{\infty}\varphi_{0}\left( x\right) dx=1 
\end{equation*}
where we approximate the sums in (\ref{eq:Mresc}) by means of integrals. In
order to define $L_{0}$ in a unique manner, we assume~\eqref{as:phi0:2} and obtain%
\begin{equation}
1=v+w+\varepsilon L_{0}\ \ ,\ \ \int_{0}^{\infty}\varphi_{0}\left( x\right)
dx=1  \label{eq:clustNorm}
\end{equation}

First, we pretend that we can neglect the effect of the boundary conditions
arising for clusters with size $j=1$. Then, Lemma~\ref{lem:I:sizedistrib} allows us to compute the concentrations after one LV
cycle from the concentrations at the beginning of the
LV cycle (with energy $E$) by evaluating the heat semigroup $S\left(
\tau\right) $ with time step $\tau=D\left( E,\varepsilon\right).$ This would
give the approximation%
\begin{equation}
\frac{\varepsilon}{L_{1}}\varphi_{1}\left( \frac{j}{L_{1}}\right) =%
\frac{\varepsilon}{L_{0}}S\left( D_0 \right) \left[
\varphi_{0}\left( \frac{\cdot}{L_{0}}\right) \chi_{\left( 0,\infty\right)
}\left( \cdot\right) \right](j)\quad{\text{for}} \; j\gg \sqrt{\f{E_0}{\ep}},  \label{eq:SemGroupRough}
\end{equation}
where $\chi_{\left( 0,\infty\right) }\left( \cdot\right) $ is the
characteristic function in the half-line $\left\{ j>0\right\} .$
This approximation can be expected to be valid for cluster sizes  away of the boundary layer, {\it i.e.} for $j\gg \sqrt{\f{E_0}{\ep}}$. We notice that if we were to take negative values of $j,$ the right-hand side would still be positive, which is impossible with the discrete system. As already stated in Lemma~\ref{lem:I:sizedistrib}, 
this mass, instead of being diffused in the region of negative clusters, accumulates in clusters with $%
j=1$ {(see the last frames in Figure \eqref{fig:sd3_p1} and Figure \eqref{fig:sd4bis_p1} and Figure \eqref{fig:inner_layer_2_p1})}.

Then, to define $\varphi_{1}$
in terms of $\varphi_{0}$ we must include all this mass in a Dirac mass at
the origin. This gives

\begin{equation*}
\begin{aligned}
\frac{\varepsilon}{L_{1}}\varphi_{1}\left( \frac{j}{L_{1}}\right) = &%
\frac{\varepsilon}{L_{0}}\chi_{\left( 0,\infty\right) }\left( j\right)
S\left( D_0 \right) \left[ \varphi_{0}\left( \frac{%
\cdot}{L_{0}}\right) \chi_{\left( 0,\infty\right) }\left( \cdot\right) %
\right](j)\\
& +\frac{\varepsilon}{L_{0}}\delta\left( j\right)
\int_{-\infty}^{0}\left( S\left( D_0 \right) \left[
\varphi_{0}\left( \frac{\cdot}{L_{0}}\right) \chi_{\left( 0,\infty\right)
}\left( \cdot\right) \right] \right) \left( \xi\right) d\xi 
\end{aligned}
\end{equation*}
or after some simplifications%
\begin{equation*}
\begin{aligned}
\varphi_{1}\left( \frac{j}{L_{1}}\right) = & \frac{L_{1}}{L_{0}}%
\chi_{\left( 0,\infty\right) }\left( j\right) S\left( D_0 \right) \left[ \varphi_{0}\left( \frac{\cdot}{L_{0}}%
\right) \chi_{\left( 0,\infty\right) }\left( \cdot\right) \right](j)\\
& +\frac{%
L_{1}}{L_{0}}\delta\left( j\right) \int_{-\infty}^{0}\left( S\left(
D_0 \right) \left[ \varphi_{0}\left( \frac{\cdot}{%
L_{0}}\right) \chi_{\left( 0,\infty\right) }\left( \cdot\right) \right]
\right) \left( \xi\right) d\xi .
\end{aligned}
\end{equation*}
We rewrite this equation using  the fundamental solution for the heat
equation $G(\xi;s)$ defined by~\eqref{eq:FundSol},
and we obtain%
\begin{align*}
\varphi_{1}\left( \frac{j}{L_{1}}\right) & =\frac{L_{1}}{L_{0}}%
\chi_{\left( 0,\infty\right) }\left( j\right) \int_{0}^{\infty
}G\left( j-\xi;D_0\right)
\varphi_{0}\left( \frac{\xi}{L_{0}}\right) d\xi \\
& +\frac{L_{1}}{L_{0}}\delta\left( j\right) \int_{-\infty}^{0}dk\int
_{0}^{\infty}G\left( k-\xi;D_0
\right) \varphi_{0}\left( \frac{\xi}{L_{0}}\right) \diff \xi.
\end{align*}
We write $x=\frac{j}{L_{1}}$ and use the change of variables $%
\xi=L_{0}\eta,$ and $k=L_{0}\zeta:$ 
\begin{align*}
\varphi_{1}\left( x\right) & =L_{1}\chi_{\left( 0,\infty\right) }\left(
x\right) \int_{0}^{\infty}G\left( L_{1}x-L_{0}\eta;D_0\right) \varphi_{0}\left( \eta\right) d\eta \\
& +\delta\left( x\right) L_{0}\int_{-\infty}^{0}d\zeta\int_{0}^{\infty
}G\left( L_{0}\zeta-L_{0}\eta;D_0
\right) \varphi_{0}\left( \eta\right) d\eta,
\end{align*}
where we use that $L_{1}\delta\left( L_{1}x\right) =\delta\left(
x\right).$ 

We now use that  $G(\xi,s)= \frac{1}{\lambda }%
G\left( \f{\xi}{\lambda},\f{s}{\lambda^2}\right) $ 
for any $\lambda>0$, and taking $\lambda=L_0$ we obtain%
\begin{align*}
\varphi_{1}\left( x\right) & =\frac{L_{1}}{L_{0}}\chi_{\left(
0,\infty\right) }\left( x\right) \int_{0}^{\infty}G\left( \frac{L_{1}}{%
L_{0}}x-\eta;\frac{D_0}{\left( L_{0}\right) ^{2}}%
\right) \varphi_{0}\left( \eta\right) d\eta  \notag \\
& +\delta\left( x\right) \int_{-\infty}^{0}d\zeta\int_{0}^{\infty}G\left(
\zeta-\eta;\frac{D_0 }{\left( L_{0}\right) ^{2}}%
\right) \varphi_{0}\left( \eta\right) d\eta,  
\end{align*}
which is~\eqref{eq:Coninit}. 
  
In order to conclude the computation of $\varphi_{1}$ we need to determine 
$L_{1}.$ To this end we impose the condition~\eqref{as:phi0:2}, namely $\int_{0}^{\infty}x%
\varphi_{1}\left( x\right) dx=1.$ Notice that since we assume that $\int
_{0}^{\infty}\varphi_{0}\left( x\right) dx=1$ we automatically have $\int_{0}^{\infty
}\varphi_{1}\left( x\right) dx=1.$ We thus have%
\begin{equation*}
1=\int_{0}^{\infty}x\varphi_{1}\left( x\right) dx=\frac{L_{1}}{L_{0}}%
\int_{0}^{\infty}xdx\int_{0}^{\infty}G\left( \frac{L_{1}}{L_{0}}x-\eta;%
\sigma_0^2\right)
\varphi_{0}\left( \eta\right) d\eta\ . 
\end{equation*}
Then%
\begin{equation*}
\frac{L_{1}}{L_{0}}=\int_{0}^{\infty}\int_{0}^{\infty} G\left( x-\eta;%
\sigma_0^2\right)
 x\diff x \varphi_{0}\left( \eta\right) \diff \eta,
\end{equation*}
which is the first equality of~\eqref{eq:Coninit2}. The second one has the advantage of highlighting that $L_1>L_0$. To obtain it, we use the symmetry of $G$  to write
\[
\begin{aligned}
\f{L_1}{L_0}&=\int_{0}^{\infty}\left(\int_{-\infty}^{\infty} G\left( x-\eta;%
\sigma_0^2\right)
x\diff x  - \int_{-\infty}^0  G\left( x-\eta;%
\sigma_0^2 \right)
x\diff x \right)\varphi_{0}( \eta) \diff \eta
\\
&=\int_0^\infty \eta \varphi_0(\eta)\diff \eta - \int_0^\infty \int_{-\infty}^{0}G\left( x-\eta;\sigma_0^2\right) x\diff x
\varphi _{0}\left(
\eta\right) \diff \eta
\\
&=1+ \int_{0}^\infty \varphi _{0}\left(
\eta\right) \int_{0}^{\infty}xG\left( x+\eta;\sigma_0^2 \right) \diff x \diff \eta
\end{aligned}
\]
and this is the second equality of~\eqref{eq:Coninit2}.
It ends the proofs of Prop.~\ref{prop:Ln}. 
This also provides us with the the formulae in~\eqref{eq:SemGroup}.
Taking as the initial step $\psi_n$ and $m_n$ with $Supp(\psi_n)\subset (0,\infty)$ as in \eqref{def:psi} and
denoting accordingly $D_n$ and $\sigma_n^2=\frac{D_n}{(L_n)^2}$ using $E=E_n$ in \eqref{def:Y}, 
we obtain the formulae in \eqref{eq:SemGroup} and finish the proof of Corollary~\ref{coroll:sizedistrib} by using the formulae in Prop. \ref{prop:Ln}, replacing the corresponding terms and iterating over the LV cycles.

\subsection{Energy decay and iterative mapping}
As a first result in Proposition~\ref{prop:energy_first}, we compute the energy decay over a first LV cycle and we obtain a semi-explicit representation formula and provide an estimation in terms of $\sqrt{\ep}$. As a second step, Proposition~\ref{prop:iter} quantifies the change of energy $E$ and the change of typical cluster length scale $L$ over following LV cycle iterations as long as the energy remains relatively big in comparison with $\varepsilon$. Please note that some cases in the proof of Proposition~\ref{prop:iter} require a more detailed analysis compared to the present section and are therefore postponed to subsection \ref{subsec:following}. While this might be sub-optimal in terms of presentation, we felt that as a result, Proposition~\ref{prop:iter} belongs to this section and shouldn't be stated any later. 
\begin{proposition}\label{prop:energy_first}
Under the assumptions and notations of Prop.~\ref{prop:Ln}, the main change of the energy with a LV cycle occurs in the interval $[T_0-\Delta t, t_4]$ and may be approximated as follows: 
\begin{equation}\label{eq:EnerIter}
E_1-E_0\sim -\int_{-\infty}^0 |s| W_0 (s)\diff s,
\end{equation}
with $W_0$ defined by
\begin{equation}
W_{0}\left( \cdot\right) =\frac{\varepsilon}{L_{0}}S\left( D_0 \right) \left[ \varphi_{0}\left( \frac {\cdot}{L_{0}}%
\right) \chi_{\left( 0,\infty\right) }\left( \cdot\right) \right].
\label{eq:UWaves}
\end{equation}
Moreover, since $\sigma_0 := \f{D_0}{(L_0)^2} \approx 1$, the following estimation holds
\begin{equation}
   |\Delta_{E_0}|:=\int_{-\infty}^0 |s| W_0 (s) \diff s\approx \varepsilon L_0 \lesssim\sqrt{\ep}.
   \label{eq:W0}
\end{equation}
\end{proposition}
With Propositions \ref{prop:Ln} and \ref{prop:energy_first}, we can now obtain formulae quantifying the change of energy and the change of the cluster length scale over instances of following LV cycles.
\begin{proposition} \label{prop:iter}
Under the assumptions and notations of Propositions \ref{prop:Ln} and \ref{prop:energy_first} and Corollary \ref{coroll:sizedistrib},  
as long as $ E_n \gg \varepsilon$, the following estimation holds
\begin{equation}
\label{eq:W}
   |\Delta_{E_n}|:=\int_{-\infty}^0 |s| W_n (s) \diff s\approx  m_n \sqrt{E_n \varepsilon} + \f{E_n}{L_n} \psi_n(0) \lesssim \sqrt{\ep} 
\end{equation}
if $\sigma_{n}^2 \ll 1$.
Moreover, iteration formulae for the energy $E_{n+1}$ and the cluster length scale $L_{n+1}$ are given by 
\begin{equation}
|E_{n+1}-E_n|\approx
\frac{E_n}{L_{n}}, \qquad \f{L_{n+1}}{L_n} = 1+ C_n, \qquad C_n = \f{|\Delta_{E_n}|}{\ep L_n}\approx \f{E_n}{\ep L_n^2}.  \label{eq:OrdMagnCh}
\end{equation}
\end{proposition}

\begin{proof}[Proof of Proposition \ref{prop:energy_first}]
As a first step, we prove the expression \eqref{eq:EnerIter}.
{We recall that the energy dynamics is given by~\eqref{D5a}:
$\f{dE}{dt}=(\ep-v)c_1,$
so that we need to compute where the contribution to $\int (\ep -v) c_1\diff t$ is dominant. 
{We also recall the partition of the LV phase space and the corresponding partition of the time period introduced in subsection \ref{subsec:LV} and in Lemma \ref{lem:I:sizedistrib}.}
\begin{itemize}
\item During $[0, t_1],$ by Lemmas~\ref{lem:scaling1} and~\ref{lem:LV1} we have $c_1\leq \ep$ and $v\in [\ep, \f{E_0}{2}]$ hence $\left|\int_0^{t_1} (\ep -v) c_1 \diff t \right| \lesssim  \ep \f{E_0}{2}t_1\sim -\f{\ep}{2}\log(\ep).$
\item During $[t_1,\Delta t]$, we have $v\leq \ep$ hence $|\int_{t_1}^{\Delta t} (\ep -v)c_1 \diff t \lesssim \ep^2 \Delta t\approx \f{\ep^{\f{3}{2}}}{\sqrt{E_0}}.$
\item During $[\Delta t, T_0-\Delta t]$, $c_1 \lesssim e^{-\sqrt{\f{E_0}{\ep}}} \ll \ep^4$ hence $\int_{\Delta t}^{T_0-\Delta t} (\ep - v) c_1\diff t \ll \ep^2$ is negligible.
\item During $[T_0-\Delta t, t_4]$, we compute below the contribution $\int_{T_0-\Delta t}^{t_4} (\ep-v)c_1\diff t.$ 
\item During $[t_4, T_0],$ analogously to the time $[0,t_1],$ we have $\left|\int_{t_4}^{T_0} (\ep -v) c_1 \diff t \right| \lesssim   -\f{\ep}{2}\log(\ep).$ 
\end{itemize}
It remains to compute the contribution to the change of energy during the time $[T_0-\Delta t, t_4]$. We have seen in Lemma~\ref{lem:I:sizedistrib} that we can approximate the size distribution dynamics by the pure transport equation along $Y(t)$,  and moreover in $[T_0-\Delta t, t_4]$ we have $v\sim E_0$ and $w \lesssim \ep,$ hence in the variable $\tau$ we have $\f{dY}{d\tau} \sim -2.$ }
The discontinuity of the characteristic function $\chi_{(0,\infty)}(\cdot)$ results in the formation of a concentration front, that can be described when it approaches to clusters of order one by means of the function
\begin{equation}
c\left( \cdot,\tau\right) =\frac{\varepsilon}{L_{0}}S\left( D_0 \right) \left[ \varphi_{0}\left( \frac{\cdot-2\left(
\tau_{\ast}-\tau\right) }{L_{0}}\right) \chi_{\left[ 0,\infty\right) }\left(
\cdot \right) \right]  \label{eq:cMatching}
\end{equation}
where {$\tau_*=D_0$ is the period of the LV cycle}. 
It is convenient to rewrite this formula defining the function $W_0$ by~\eqref{eq:UWaves},%
then%
\begin{equation*}
c\left( j,\tau \right) =W_0\left( j-2\left( \tau _{\ast }-\tau
\right) \right) .
\end{equation*}
{Here, we use $j$ as continuum variable.}
The flux of monomers towards clusters of order one
can thus be approximated 
as $c\left( 0^{+},t\right) =W_0$ leading~(\ref{eq:BoundLay2}) to be written%
\begin{equation*}
\frac{dc_{1}}{d\tau}=2W_0\left( -2\left( \tau_{\ast}-\tau\right)
\right).
\end{equation*}
Then, using that the function describing the concentrations front $W_0$ decays fast for large negative values of its argument {so that integrating between $\tau(T_0-\Delta t)$ and $\tau$ as the same order of magnitude as integrating between $0$ or even $-\infty$ and $\tau$}, we obtain
\begin{equation}
c_{1}\left( \tau\right) \sim 2\int_{{0}}^{\tau}W_0\left(
-2\left( \tau_{\ast}-s\right) \right) ds\sim \int_{-\infty}^{2\left( \tau
-\tau_{\ast}\right) }W_0\left( s\right) ds\ \ \text{for }%
\tau<\tau_{\ast}  \label{eq:MonConc}.
\end{equation}
In order to estimate the order of magnitude of the change of the energy 
due to the interaction of the concentrations waves with the
regions with cluster sizes $j$ of order one, we use the fact that $%
W_0$ is of order $\frac{\varepsilon}{L_{0}}$ and it has a width
of order $\sqrt{\frac{E_0}{\varepsilon}}$ (cf. (\ref{eq:UWaves})). 
We remark that this scaling properties for $W_0$ are based in the
assumption that the contribution of the Dirac in $\varphi_{0}$ gives a
contribution to $W_0$ of the same order of magnitude that the
part of $\varphi _{0}$\ in $\left\{ x>0\right\} .$\\ 
We can then compute the change of the energy 
using (\ref{D5a}). 
We recall that 
$\frac{dE}{dt}=\left( \varepsilon-v\right) c_{1},$ hence 
 $\frac{dE}{d\tau}=\frac{2\left( \varepsilon-v\right) }{\left(
w+v\right) }c_{1}\approx -2c_1$ in $[T_0-\Delta t,t_4]$ since $v\simeq E_0$
and $w\lesssim\varepsilon.$ 
Then, as long as $E\gg\varepsilon,$ we have the following approximation during the range
of times in which $c_{1}$ contributes significantly to the change of $E $%
\begin{equation}
\frac{dE}{d\tau}\simeq-\frac{2E}{E}c_{1}=-2c_{1}  \label{eq:EnerChang}
\end{equation}
Combining (\ref{eq:MonConc}) and (\ref{eq:EnerChang}) {and claiming again that the integrals are negligible away from $[T_0-\Delta t, t_4]$,} we obtain%
\[\begin{array}{ll}
E_{1}-E_{0}&\approx {- 4 \displaystyle \int_0^{\tau_*} \int_0^\tau W_0 (-2(\tau_*-s))\diff s\diff \tau }=\int_{-\tau_*}^0 W_0 (s) s \diff s 
\\ \\
&\approx -\displaystyle\int_{-\infty}^{0}W_0\left( s\right) \left\vert
s\right\vert ds.
\end{array}\]
\\

In order to prove the estimation \eqref{eq:W0}, we compute
\begin{eqnarray*} 
\Delta_{E_0}&=&
\frac{\varepsilon}{L_{0}}\int_{-\infty}^0 |s| \int_0^\infty \varphi_0\left(\f{y}{L_0}\right)G(s-y;D_0)\diff y \diff s
\\
&=& \f{\ep}{{L_0^2}} \int_0^\infty s \int_0^\infty \varphi_0\left(\f{y}{L_0}\right) G\left(\f{s+y}{L_0};\sigma_0^2\right)\diff y \diff s 
\\
&=& {\ep} L_0 \int_0^\infty s \int_0^\infty \varphi_0(y) G(s+y;\sigma_0^2)\diff y \diff s
\\
&=& {\ep} L_0 \int_0^\infty s \left(m_0 G(s;\sigma_0^2)+ \int_0^\infty \psi_0(y) G(s+y;\sigma_0^2)\diff y  \right)\diff s.
\end{eqnarray*}
Using $\int_0^\infty s G(s;\sigma^2)\diff s=\f{\sigma}{\sqrt{\pi}}$, we obtain
\begin{equation}
  \Delta_{E_{0}}  = {\ep} L_0 \left( \f{m_0 \sigma_0 }{\sqrt{\pi}} + \int_0^\infty s \int_0^\infty \psi_0(y) G(s+y;\sigma_0^2)\diff y  \diff s \right).
  \label{eq:ener_decay_estimate}
\end{equation}
Since $\sigma_0 \approx 1,$ the sum of the first and second terms is in the order of $\ep L_0$ since $m_0+\int_0^\infty \psi_0(x)dx=1,$
which shows \eqref{eq:W0}.
\end{proof}

\begin{proof}[Proof of Proposition \ref{prop:iter}]
{The formula of the energy decay \eqref{eq:EnerIter} as well as its estimate \eqref{eq:ener_decay_estimate} holds over the iterations of the LV cycles as long as $E_n\gg \ep$. 
The only change in comparison to the first cycle is that if $\sigma_n \ll 1$, we have that }
$$\Delta_{E_{n}}  = {\ep} L_n \left( \f{m_n \sigma_n }{\sqrt{\pi}} + \int_0^\infty s \int_0^\infty \psi_n(y) G(s+y;\sigma_n^2)\diff y  \diff s \right), $$
where the first term is in the order of $\ep L_n m_n\sigma=\ep m_n\sqrt{D_n}=m_n\sqrt{E_n\ep}$ {(cf. Lemma \ref{lem:DYT}),} whereas the second term may be approximated by 
$$\ep L_n \psi_n(0) \sigma_n^2 \int_0^\infty\int_y^\infty (z-y) G(z;1)\diff z \diff y \approx \ep \psi_n(0)\f{D_n}{L_n}=\f{E_n}{L_n}\psi_n(0).$$
Concerning the changes in the cluster distribution over the LV cycles, 
for $\f{L_{n+1}}{L_n}$, 
we have
\[\f{L_{n+1}}{L_n} = 1 +C_n,\]
with
\[C_n=\int_0^\infty \!\int_0^\infty x G(x+\eta, \sigma_n^2) \varphi_n(\eta)\diff \eta \diff x=\f{1}{\ep L_n} \int_{-\infty}^0 |s| W_n(s)\diff s  =\f{|E_{n+1}-E_n|}{\ep L_n},\]
so that it only remains to prove either the estimation of $E_{n+1}-E_n$ or the one for $C_n$ to get the other. 

In the cases that $L_n\approx \f{1}{\sqrt{\ep}},$ we easily obtain~\eqref{eq:OrdMagnCh} from \eqref{eq:W0} 
since then $E_n\approx 1,$ so $\ep L_n \approx \sqrt{\ep} \approx \f{E_n}{L_n},$ 
which finishes the proof of Proposition~\ref{prop:iter} in this case.
However, the proof is much more involved in case $L_n \gg \f{1}{\sqrt{\ep}},$ or equivalently, for $\sigma_n \ll 1.$ It then requires the analysis carried out in Subsection~\ref{subsec:following} below.
\end{proof}

\section{{Phases I and II: overall dynamics}}
\label{sec:IandII:overall}
\subsection{Early Phase I: initial increase of the characteristic length\label{subsec:early}}
With an initial energy $E_0\approx 1,$ we have noticed that even if $L_0 \ll \f{1}{\sqrt{\ep}}$, the diffusive effects within one LV cycle yield $L_1\approx \f{1}{\sqrt{\ep}}$, whereas the energy remains $E_1\approx 1$ since $\Delta E_0 \approx \ep L_0 \ll 1$ by Prop.~\ref{prop:energy_first}.
Over following LV cycles, by~\eqref{eq:OrdMagnCh} we have that $L_n$ increases much faster than $E_n$ decreases. 
Departing from $L_1\approx \f{1}{\sqrt \ep}$, we have $C_1\approx 1$. Hence, after  a few cycles (in the order of $-\log (\ep)$) we have $L_n\gg \f{1}{\sqrt \ep},$ so that $\sigma_n^2 =\f{E_n}{\ep L_n^2} \ll 1$ {(see Prop. \ref{prop:iter})}.
The change in the size distribution for the early Phase I is illustrated in the Figure \ref{fig:sd_shape_p1}.  

\begin{figure}[hptb]

  \centering
     \begin{subfigure}[b]{0.49\textwidth}
         \centering
		 \includegraphics[width=\textwidth]{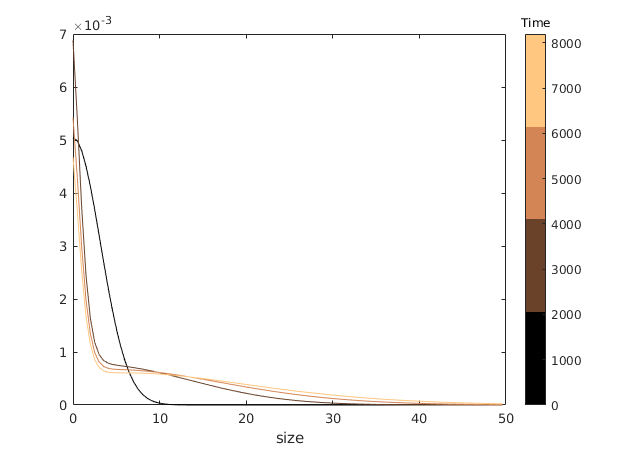}
         \caption{Size distribution $c(j,t)$ when $v=w>\epsilon$, i.e. for $t =0, \, T_0, \, T_1, \, T_2.$}
         \label{fig:sd_shape_1_p1}
     \end{subfigure}
     \hfill
     \begin{subfigure}[b]{0.49\textwidth}
         \centering
         \includegraphics[width=\textwidth]{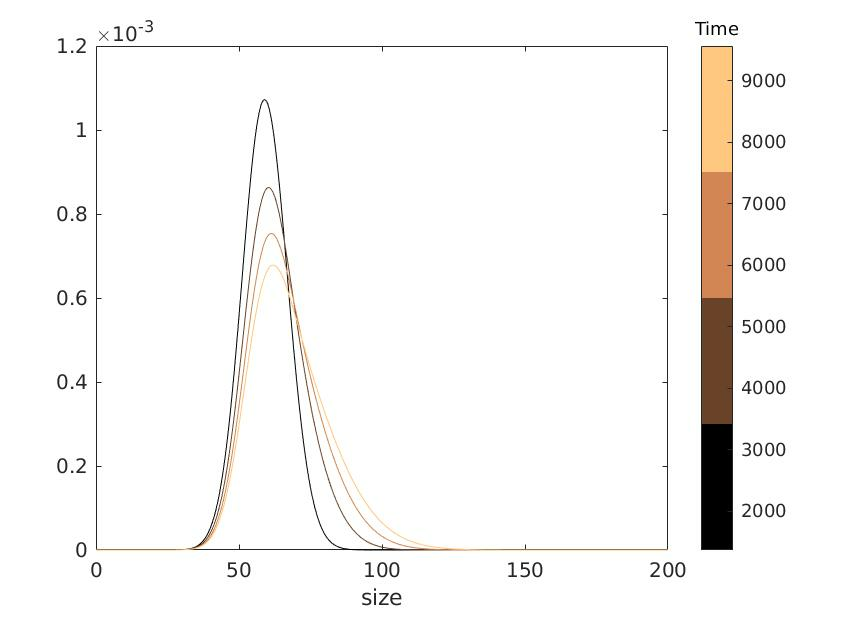}
         \caption{Size distribution $c(j,t)$ when $v=w<\epsilon$, i.e. for $t \approx\Delta t, \, T_0 + \Delta t, \, T_1+ \Delta t, \, T_2+ \Delta t.$}
         \label{fig:sd_shape_2_p1}
     \end{subfigure}

\caption{Evolution of the cluster size distribution over four LV cycles
with $\ep=0.02$ and $1/\sqrt{\ep}\approx 7$. See Appendix \ref{sec:app_num} for the details about the numerical simulations.}
\label{fig:sd_shape_p1}
\end{figure}

\begin{figure}[hptb]
     \centering
     \begin{subfigure}[b]{0.49\textwidth}
         \centering
         \includegraphics[width=\textwidth]{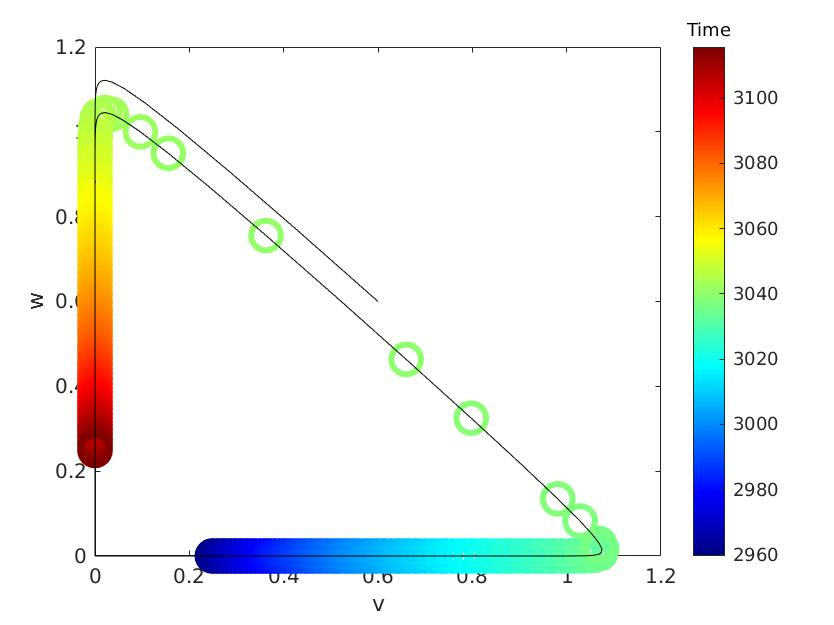}
         \caption{Phase portrait of the monomers}
         \label{fig:vw_ener_p1}
     \end{subfigure}
     \hfill
     \begin{subfigure}[b]{0.49\textwidth}
         \centering
         \includegraphics[width=\textwidth]{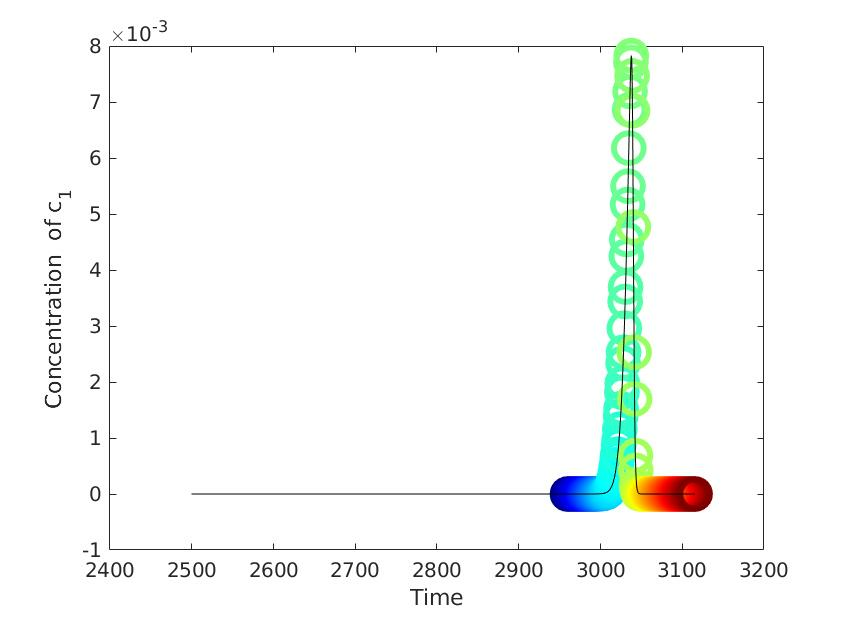}
         \caption{Evolution of the concentration of $c_1$}
         \label{fig:c1_ener_p1}
     \end{subfigure}
     \hfill
     \begin{subfigure}[b]{0.49\textwidth}
         \centering
         \includegraphics[width=\textwidth]{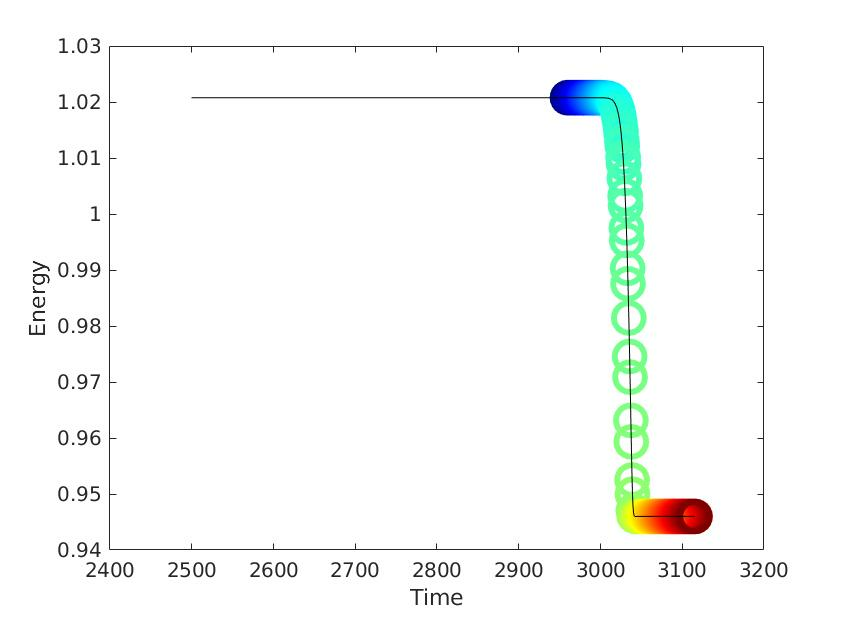}
         \caption{Evolution of the energy $E$ }
         \label{fig:ev_ener_p1}
     \end{subfigure}
        \caption{{Change of the energy $E$ \eqref{D5} for one iteration of the LV cycle when $c_1\neq 0$ in equations \eqref{eq:BD2}--\eqref{eq:BD2a}.}}
        \label{fig:ener_p1}
\end{figure}

{\subsection{Continuation of Phase I and Phase II: size distribution evolution\label{subsec:following}}}

 The change of energy in one LV cycle being of
order {at most } $ \sqrt{\varepsilon}$ 
{(see Prop. \ref{prop:energy_first} and \ref{prop:iter})},
{the energy $E_n$ first remains of order one while $\sigma_n \ll 1$ (Phase I) then decreases to order $\ep$ (Phase II). }

{To describe the dynamics after the early Phase I, 
we are in the case of {$\sigma_n \ll 1$ (see Prop. \ref{prop:iter})}} 
so that we need to evaluate both $m_n$ and $\psi_n(0)$ to prove~\eqref{eq:OrdMagnCh}. 
It is thus convenient to separate the size domain in three parts: the Dirac delta for $c_1$ (providing $m_n$), a boundary layer for sizes $1\ll i\lesssim \f{1}{\sqrt \ep},$ and large sizes $i\gg \f{1}{\sqrt \ep},$ so that $\psi_n(0)$ may be viewed as the limit for $x\to\infty$ of the boundary layer as well as the limit $x\to 0$ of the large-sizes part  of the domain. This is expressed in the following  proposition.

\begin{proposition}\label{prop:UM}
Let  $\ep \ll E_n$ and $\f{1}{\sqrt \ep} \ll L_n,$ so that $\sigma_n^2=\f{E_n}{\ep L_n^2} \ll 1.$ Set $x=\f{j}{L_n}$. Following the notations of Corollary~\ref{coroll:sizedistrib}, we have the following approximations for the size distribution $\psi_n.$
\begin{itemize}
\item For $x>0$ not small, more precisely $x=\f{j}{L_n}$ with $j \gg \sqrt{\f{E_n}{\ep}}$, we have
\begin{equation}\label{eq:PsiAs}
\psi_n(x)\approx \psi(x):=\f{2}{\pi} e^{-\f{x^2}{\pi}},
\end{equation}
from which we also achieve the proof of~\eqref{eq:OrdMagnCh}.
\item For $x=\f{j}{L_n}$ with $j\lesssim \sqrt{\f{E_n}{\ep}},$ defining $x=\sigma_{n+1} \xi$, $m_n=\sigma_{n}M_n$ and $\psi_n (\sigma_n \xi)=U_n(\xi)$, we have that $(U_n,M_n)\approx (U,M)$ solutions of the following system: 
\begin{align}
U\left( \xi\right) & =\int_{0}^{\infty}G\left( \xi-\zeta;1\right) U\left(
\zeta\right) d\zeta+MG\left( \xi;1\right) \ \ ,\ \ \xi >0  \label{BLSt1} 
\\
\frac{M}{2} & =\int_{-\infty}^{0}d\zeta\int_{0}^{\infty}G\left( \zeta
-\xi;1\right) U\left( \xi\right) d\xi  \label{BLSt2} 
 \\
U\left( \infty\right) & =\frac{2}{\pi}  \label{BLSt3} 
\end{align}

\end{itemize}
 \end{proposition}

 \begin{proof}
In all the following dynamics, we have that $\f{|\Delta_{E_n}|}{\ep L_n}\ll 1$ so that \eqref{eq:OrdMagnCh} implies $\f{L_{n+1}}{L_n}\sim 1$ in the limit $\sigma_n \to 0^+$.
For convenience, we recall the first two iteration formulas \eqref{eq:SemGroup} previously stated in Corollary ~\ref{coroll:sizedistrib}:
\begin{align}
\psi_{n+1}\left( x\right)  &=\frac{L_{n+1}}{L_{n}}\int_{0}^{\infty}G\biggl(
\frac{L_{n+1}}{L_{n}}x-\eta,{\sigma_n^2}\biggr) \psi_{n}( \eta) d\eta\notag
 +m_{n}\frac{L_{n+1}%
}{L_{n}}G\biggl( \frac{L_{n+1}}{L_{n}}x,\sigma_n^2\biggr),  \notag \\
m_{n+1} & =\int_{-\infty}^{0}d\zeta\int_{0}^{\infty}G\left( \zeta-\eta,{
\sigma_n^2} \right)
\psi_{n}\left( \eta\right) d\eta+\frac{m_{n}}{2}  
\label{eq:SemGroupbis}
\end{align}
In the following we approximate \eqref{eq:SemGroup} resp. \eqref{eq:SemGroupbis} using the 
expansion formula%
\begin{equation}
\int_{0}^{\infty}G\left( x-y;\sigma^2\right) \psi\left( y\right) dy=\psi\left(
x\right) +\sigma^2\partial_{xx}\psi\left( x\right) +o\left( \sigma^2\right) 
\quad \text{as} \quad \sigma\rightarrow0^{+}
\label{eq:HeatForm}
\end{equation}
which holds not uniformly but provided that $x>0$ is fixed while and $%
\sigma\rightarrow 0^{+}.$
On the other hand, the term $m_{n}\frac{L_{n+1}}{%
L_{n}}G\left( \frac{L_{n+1}}{L_{n}}x;\sigma_n^{2}\right) $ in the first equation of (\ref%
{eq:SemGroup}) resp. \eqref{eq:SemGroupbis} is exponentially small in $\sigma_n^{2}$ if $x$ is of order one. We then obtain using
(\ref{eq:HeatForm}) the following approximation%
\begin{equation}
{ \psi_{n+1}\left( x\right) =\frac{L_{n+1}}{L_{n}}\Bigl[ \psi_{n}+\sigma_n^2 \partial_{xx}\psi_{n}\Bigr] \Bigl( \frac{L_{n+1}}{L_{n}}x\Bigr)}  ,\ \
 \ x>0\ ,  \label{eq:IterApp}
\end{equation}
if $\sigma_n^2$ tends to zero. Notice that the iterative formula (\ref{eq:IterApp}) does
not depend on $m_{n}.$ This suggests in particular that the ratio $\frac{%
L_{n+1}}{L_{n}}$ can be determined  independently on $%
m_{n}$ if $\sigma_n^2\ll1$.
We will assume that $\psi_{n}$ stabilizes to a steady state. In
addition, we use the fact that $\sigma^2_n$ tends to zero. We will assume
that {$\frac{L_{n+1}}{L_{n}}\sim 1+a\sigma_n^2$ as $\sigma_n^2%
\rightarrow0, $} where $a>0$ must be determined. Then, looking for steady
states $\psi$ of (\ref{eq:IterApp}) and neglecting lower order terms, as
well as using the approximation $\psi\left( \left( 1+a\sigma^2\right)
x\right) \simeq \psi\left( x\right) +a\sigma^2 x\partial_{x}\psi\left(
x\right) ,$ we obtain that $\psi$ solves%
\begin{equation*}
\sigma^2\left( a\psi\left( x\right) +ax\partial_{x}\psi\left( x\right)
+\partial_{xx}\psi\left( x\right) \right) =0 
\end{equation*}
where we keep only the lower order terms in $\sigma^2.$ Then%
\begin{equation}
a\psi\left( x\right) +ax\partial_{x}\psi\left( x\right)
+\partial_{xx}\psi\left( x\right) =0\  ,\ \ x>0\ .  \label{eq:psiEq}
\end{equation}

We can expect to have the approximation $\psi_{n}\simeq\psi$ as $%
n\rightarrow \infty.$ Moreover, we remark that we can expect to have $%
m_{n}\rightarrow0$ as $n$ becomes larger enough. Indeed, this is a consequence of
the fact that as $\sigma^2_n\rightarrow0,$ also the amount of mass transferred
to the region $x<0$ by the heat semigroup tends to zero. Therefore, the
normalization condition (\ref{as:phi0:2}) implies that approximately%
\begin{equation}
\int_{0}^{\infty}\psi\left( x\right) \diff x=\int_{0}^{\infty}x\psi\left(
x\right) \diff x=1.  \label{eq:NormPsi}
\end{equation}

The integrable solutions of (\ref{eq:psiEq}) have the form $\psi\left(
x\right) =Ce^{-\frac{ax^{2}}{2}}$ where $C$ is an arbitrary real constant.
The normalization conditions (\ref{eq:NormPsi}) imply $C=a=\frac{2}{\pi }.$ Then we obtain~\eqref{eq:PsiAs},
and also%
\begin{equation*}
\frac{L_{n+1}}{L_{n}}=1+\frac{2}{\pi}\sigma_n^2\ \ ,\ \ \sigma_n^2=\frac{%
D\left( E_n,\varepsilon\right) }{\left( L_{n}\right) ^{2}}.\
\end{equation*}
as $n$ becomes large enough. This also achieves the heuristic proof of~\eqref{eq:OrdMagnCh}.\\

We now consider the description of the functions $\psi_{n}\left( x\right) $
for small $x$,  {\it i.e.} for $x$ in the order of $\sigma_n$. In the first equation of \eqref{eq:SemGroup} resp. \eqref{eq:SemGroupbis},
we use $\frac{L_{n+1}}{L_{n}}\sim 1$
and introduce the new variables $x=\sigma_{n+1} \xi,$ $\psi_{n}\left( \sigma_n %
\xi\right) =U_{n}\left( \xi\right).$ Then%
\[
U_{n+1}\left( \xi\right) \sim \int_{0}^{\infty}G\left( 
\sigma_{n+1} \xi-\eta
;\sigma_n^2\right) \psi_{n}\left( \eta\right) d\eta
 +m_{n}%
G\left(  \sigma_{n+1} \xi;\sigma_n^2\right).
\]
Defining $ \eta=\sigma_n \zeta$ and using $G(y;\sigma^2)=\f{1}{\sigma}G(\f{y}{\sigma};1)$, we obtain%
\begin{equation*}
U_{n+1}\left( \xi\right) \sim \int_{0}^{\infty}G\left( 
\f{\sigma_{n+1}}{\sigma_n}\xi-\zeta;1\right) U_{n}\left( \zeta\right)
d\zeta+\frac{m_{n}}{\sigma_n}G\left( \f{\sigma_{n+1}}{\sigma_n}\xi;1\right).
\end{equation*}

We now remark that $\f{\sigma_{n+1}}{\sigma_n} \sim 1$ as $\sigma_n \to 0$. Moreover, using  
the rescaling {$m_{n}=%
\sigma_{n}M_{n}$}, we introduce the following equation as the definition for the rescaled profile $U_n$ of boundary layer
describing the concentrations near $x=0$:%
\begin{equation}
U_{n+1}\left( \xi\right) = \int_{0}^{\infty}G\left( \xi-\zeta;1\right)
U_{n}\left( \zeta\right) d\zeta+M_{n}G\left( \xi;1\right)   ,\ \ \xi >0\ .
\label{BL1}
\end{equation}

In order to obtain a closed system for both $U_{n}$ and $M_{n}$, we use the
second equation in (\ref{eq:SemGroup}) resp. \eqref{eq:SemGroupbis}. With the definitions of $U_{n}$ and 
$M_{n}$, we obtain the equation%
\begin{equation*}
M_{n+1}=\frac{{L_{n+1}}}{L_{n}}\int_{-\infty}^{0}d\zeta\int_{0}^{\infty}G%
\left( \zeta-\xi;1\right) U_{n}\left( \xi\right) d\xi+{\frac{L_{n+1}}{L_{{n}}}}
\frac{M_{n}}{2} .
\end{equation*}

Approximating then $\frac{L_{n+1}}{L_{n}}$ by $1$, we 
obtain%
\begin{equation}
M_{n+1}=\int_{-\infty}^{0}d\zeta\int_{0}^{\infty}G\left( \zeta-\xi;1\right)
U_{n}\left( \xi\right) d\xi+\frac{M_{n}}{2} . \label{BL2}
\end{equation}

The equations \eqref{BL1}--\eqref{BL2} describe the iterative dynamics of the
boundary layer near $x=0$ yielding the cluster concentrations as well as the
mass of the peaks appearing there. These equations must be solved combined
with the matching condition that results from the fact that $\psi_{n}\left( 
\sigma_n \xi\right) =U_{n}\left( \xi\right) . $
Then, since $\sigma_n\ll 1,$ it is natural to impose the matching condition $%
U_{n}\left( \infty\right) =\psi_{n}\left( 0^{+}\right) .$ Taking into
account (\ref{eq:PsiAs}), we must then impose the following matching
condition for the system \eqref{BL1}--\eqref{BL2}%
\begin{equation}
U_{n}\left( \infty\right) =\frac{2}{\pi}  \label{BL3}.
\end{equation}
It is natural to assume that the solutions of \eqref{BL1}--\eqref{BL3}
approach a stationary solution for large $n.$ Therefore, they become
close to the solutions of the problem~\eqref{BLSt1}--\eqref{BLSt3}{, and it ends the (heuristic) proof.}%
\end{proof}
\bigskip

{In the companion paper~\cite{DFMV2},} it will be proved  that there exists a unique solution of \eqref{BLSt1}--\eqref%
{BLSt2} satisfying the boundary condition (\ref{BLSt3}). The resulting
function $U\left( \xi\right) $ and the rescaled mass $M$ describe the
concentration of clusters $c_{j}$ for $j$  {in the order of $\f{1}{\sigma_n},$ with $1\ll \f{1}{\sigma_n} \ll L_{n}.$}

\bigskip

\begin{remark}
We can use the third equation in (\ref{eq:SemGroup}) to estimate $\frac{%
L_{n+1}}{L_{n}}$ to check the correctness of the approximation~\eqref{eq:OrdMagnCh} 
that has been obtained with a different approach,
approximating the evolution of the cluster sizes with $x$ of order one by
means of a differential operator. We have%
\begin{equation*}
\frac{L_{n+1}}{L_{n}}=1+\sigma_n^2 \int_{0}^{\infty
}y\diff y\int_{0}^{\infty}G\left( y+\xi;1\right) U_{n}\left( \xi\right)
\diff \xi+M_{n}{\sigma_n^2}\int_{0}^{\infty}yG\left( y;1\right)
\diff y\ .
\end{equation*}
Using then the approximations $U_{n}\simeq U$
and $M_{n}\simeq M$, we obtain
the following approximation%
\begin{equation*}
\frac{L_{n+1}}{L_{n}}=1+\sigma_n^2\left( \int
_{0}^{\infty}ydy\int_{0}^{\infty}G\left( y+\xi;1\right) U\left( \xi\right)
d\xi+M\int_{0}^{\infty}yG\left( y;1\right) dy\right). 
\end{equation*}
Therefore, both terms in the formula of $\frac{L_{n+1}}{L_{n}}$ give a
comparable contribution. 
{Another way to see this result is to go back to~\eqref{eq:W} in Prop. \ref{prop:iter}: replacing $\psi(0)$ by $\f{2}{\pi}$ and $m$ by $\sigma_n M$ we have
\[ |\Delta_{E_n}| \approx M \sigma_n \sqrt{E_n \ep} + \f{E_n}{L_n}\f{2}{\pi} \approx \f{E_n}{L_n}\left(M+\f{2}{\pi}\right),
\]
and here again both terms are of similar order.}
\end{remark}

{\subsection{Overall dynamics during phases I and II}}
\label{subsec:IandII:overall}
From~\eqref{eq:OrdMagnCh} we can now detail the overall dynamics during Phases I and II, that we summarize in Prop.~\ref{prop:phasesIandII}.
Let us first define the constant%
\begin{equation}
A:=\int_{0}^{\infty}ydy\int_{0}^{\infty}G\left( y+\xi;1\right) U\left(
\xi\right) d\xi+M\int_{0}^{\infty}yG\left( y;1\right) dy.
\label{eq:Aconstant}
\end{equation}

{\begin{proposition}\label{prop:phasesIandII}
Departing from $E_0\approx 1$ and $L_0\approx \f{1}{\sqrt{\ep}},$  with $A$ defined by~\eqref{eq:Aconstant}, we have the following dynamics.
\begin{itemize}
\item Phase I: as long as $n \ll \f{1}{\ep},$ we have
\begin{equation}
L_{n}\sim\sqrt{\frac{2An}{\varepsilon}},\qquad E_n \sim 1.  \label{eq:LPhI}
\end{equation}
\item Phase II: $L_n$ increases to $\f{1}{\ep}$ and $E_n$ decreases, and at the end of Phase II, {\it i.e.} for $\f{1}{\ep}\ll n\lesssim -\f{\log(\ep)}{\ep}$, we have
\begin{equation}\label{eq:endPhaseII}
L_n \sim \f{1}{\ep}(1-e^{-1}e^{-A\ep n}),\qquad E_n\sim e^{-1}e^{-A\ep n}.
\end{equation}
\end{itemize}
\end{proposition}}
\begin{proof}
{The whole proof is based on studying the sequence $(L_n,E_n)$ defined by~\eqref{eq:OrdMagnCh}. }%
{We have already described the early Phase I in subsection~\ref{subsec:early}: after a few cycles we have $L_n\gg \f{1}{\sqrt \ep}$, and then by Prop.~\ref{prop:UM} }$\psi_{n}$
approaches $\frac{2}{\pi}e^{-\frac{x^{2}}{\pi}}$ and we can approximate $%
M_{n}$ and $U_{n}$ by the solutions of \eqref{BLSt1}--\eqref{BLSt3}. We can
then approximate the evolution of $L_{n}$ using 
{that $C_n\sim A\sigma_n^2$  so that~\eqref{eq:OrdMagnCh} may be written in the following more precise way:}
\begin{equation}
\frac{L_{n+1}}{L_{n}}- 1 \sim \frac{AE_{n}}{\varepsilon\left( L_{n}\right) ^{2}},\qquad
{E_{n+1}-E_{n}\sim -A \f{E_n}{L_n}}.\label{eq:Echn}
\end{equation}

{\subsubsection*{Phase I asymptotics.} }
As long as $E_n \sim 1,$ we can simplify~\eqref{eq:Echn} by writing
\begin{equation*}
L_{n+1}\sim L_{n}+\frac{A}{\varepsilon L_{n}}.
\end{equation*}
Hence%
\begin{equation*}
\frac{\left( L_{n}\right) ^{2}}{2}\sim\frac{\left( L_{0}\right) ^{2}}{2}+%
\frac{An}{\varepsilon} 
\end{equation*}
for {$1\ll n\ll \f{1}{\ep}$, and} we obtain the asymptotics~\eqref{eq:LPhI}.

 \subsubsection*{{Phase II asymptotics}\label{SubPhase2}}

We
define a new set of variables, namely%
\begin{equation}
L_{n}=\frac{1}{\varepsilon}\ell\left( s\right) \ ,\ E_{n}=e\left( s\right) \
,\ s=\varepsilon n  \label{eq:ELDef}
\end{equation}
Then (\ref{eq:Echn}) becomes%
\begin{equation*}
\frac{\ell\left( s+\varepsilon\right) -\ell\left( s\right) }{\varepsilon }%
\sim \frac{Ae\left( s\right) }{\ell\left( s\right) }\  ,\ \ \frac {e\left(
s+\varepsilon\right) -e\left( s\right) }{\varepsilon}\sim -\frac{{A} e\left(
s\right) }{\ell\left( s\right) }\ . 
\end{equation*}
It is natural to approach the left-hand side of these equations by
derivatives. We then obtain the following system of ODEs%
\begin{equation}
\frac{d}{ds}\ell\left( s\right) =\frac{Ae\left( s\right) }{\ell\left(
s\right) }\  ,\ \ \frac{d}{ds}e\left( s\right) =-\frac{{ A}e\left( s\right) }{%
\ell\left( s\right) }\ ,  \label{eq:ODEPhII}
\end{equation}
{and since at the beginning of Phase II we still have $L_n \ll \f{1}{\ep}$ and $E_n\sim 1,$}
 the matching conditions for $L_{n},\ E_{n}$ yield the initial
conditions%
\begin{equation}
{\ell\left( 0\right) =0}\  ,\ \ e\left( 0\right) =1.  \label{eq:ODEInCon}
\end{equation}
The initial conditions stated in \eqref{eq:ODEInCon} makes the system \eqref{eq:ODEPhII} singular, however the initial conditions are compatible since the solution $\ell$ is bounded away from $0$ exponentially fast. 
Equations \eqref{eq:ODEPhII}--\eqref{eq:ODEInCon} yield the evolution of $%
L_{n},\ E_{n}$ during Phase II. These equations can be solved in the
following implicit form%
\begin{equation}
-\frac{1}{{A}}\ell-{\frac{1}{A}}\log\left( 1-{ \ell}\right) =s\  ,\
\ e=1-{ \ell}\ .  \label{eq:EnLePhII}
\end{equation}
{
Notice that this equation (or directly (\ref{eq:ODEPhII})) implies the
asymptotic behaviour%
\begin{equation*}
\ell\left( s\right) \sim\sqrt[\ ]{2As}\ ,\ \ e\left( s\right) \sim1\ \ 
\text{as \ }s\rightarrow0\ . 
\end{equation*}
}
\

We obtain that $\ell$ increases from $\ell=0$ for $s=0$ to $\ell={ 1}$
as $s$ tends to infinity. On the other hand $e$ decreases from $e=1$ for $%
s=0 $ to $e=0$ as $s$ tends to infinity. We have the following asymptotic
behaviour for $\ell$ and $s$%
\begin{equation}
\ell\left( s\right) \sim \left( 1-e^{-1}e^{-{ A}%
s}\right)  ,\ \ e\left( s\right) \sim e^{-1}e^{-{ A} s}\ \ \text{%
as\ \ }s\rightarrow\infty\ .  \label{eq:enelevol}
\end{equation}

Combining (\ref{eq:ELDef}) and (\ref{eq:enelevol}) yields the evolution of $%
L_{n},$ $E_{n}$ during Phase II.  The formula for $\ell\left(
s\right) $ in (\ref{eq:enelevol}) implies that $L_{n}$ remains of order $%
\frac{1}{\varepsilon}$ during the Phase II.
\end{proof}

\begin{remark}
The approximations (\ref{eq:Echn})  have been
computed under the assumption that this displacement $\frac{E_{n}}{%
\varepsilon }$ is very large. However, the formula for $e\left( s\right) $
in (\ref{eq:enelevol}) as well as the fact that $E_{n}=e\left( s\right) $
imply that for $n$ of order ${\frac{1}{A}}\frac{1}{\varepsilon }\log
\left( \frac{1}{\varepsilon }\right) $ the energy $E_{n}$ becomes of order $%
\varepsilon $ and then the approximations~(\ref{eq:Echn}) are not any longer valid. This {marks} the beginning of Phase
III.
We also notice that during this phase the
concentrations $c_{k}$ experience a displacement of order $\frac{E_{n}}{%
\varepsilon}$ in the space of cluster sizes $k$ {(see Lemmas~\eqref{lem:scaling2} and~\eqref{lem:DYT})}.
\end{remark}
\begin{figure}
     \centering
     \begin{subfigure}[b]{0.49\textwidth}
         \centering
         \includegraphics[width=\textwidth]{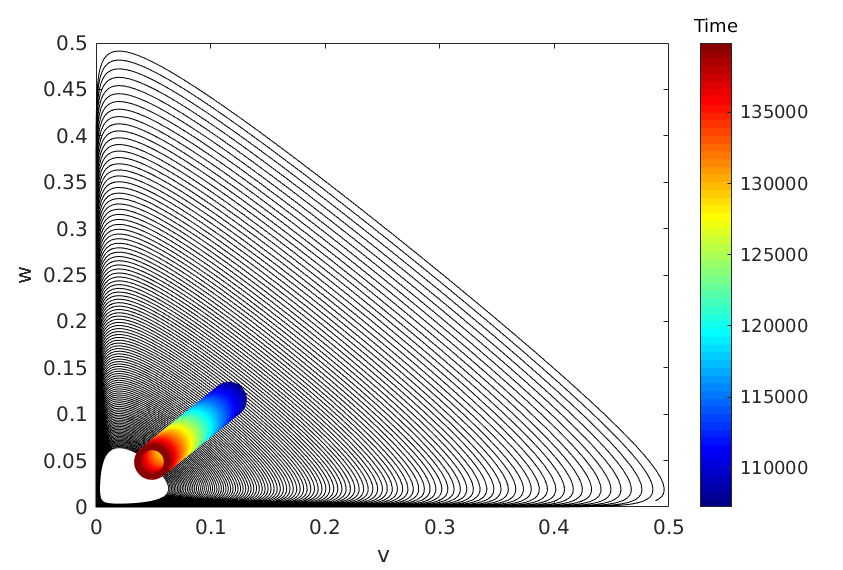}
         \caption{Phase portrait of the monomers}
         \label{fig:vw_ener_p2}
     \end{subfigure}
     \hfill
     \begin{subfigure}[b]{0.49\textwidth}
         \centering
         \includegraphics[width=\textwidth]{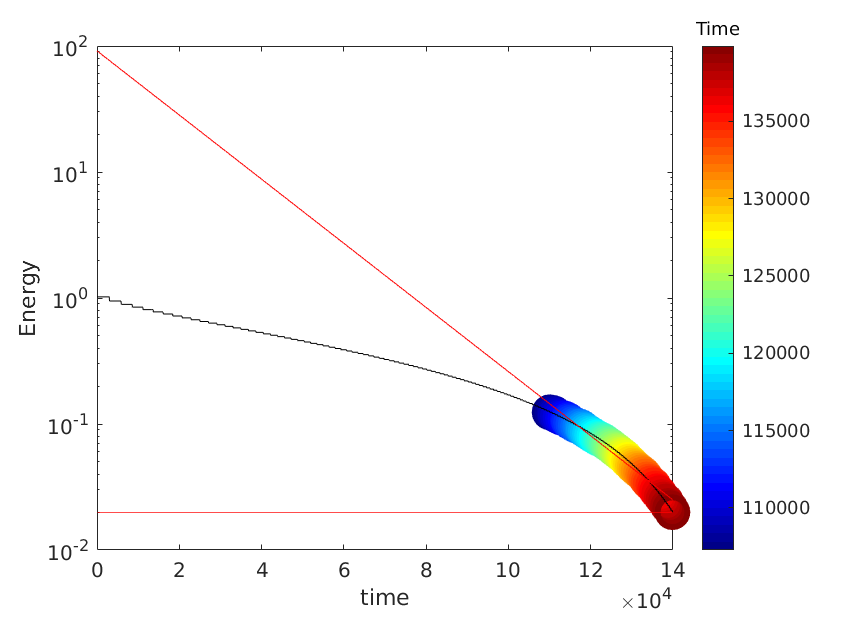}
         \caption{Evolution of the energy $E$ }
         \label{fig:ener_p2}
     \end{subfigure}
        \caption{{Numerical {simulation} of the evolution of the monomers and the energy in Phase II (c.f section~\ref{subsec:IandII:overall}). The color scheme indicates the time evolution in Figures \ref{fig:vw_ener_p2} and \ref{fig:ener_p2}. (Right) The y--axis scale is logarithmic and the horizontal line is $E_n \approx \ep$ where $\ep=0.02$ in the numerical simulation. The decreasing affine line in red is an illustration of the exponential decay of the energy that takes place in Phase II (see Prop. \ref{prop:phasesIandII}). } }
        \label{fig:ener_change_p2}
\end{figure}

\section{Phase III: Energy reduction  from order $\protect\varepsilon$ to  order $\protect\varepsilon^{2}.$}
\label{sec:PhaseIII}
\subsection{Early Phase III: matching condition}
\label{subsec:PhaseIII:early}
Phase III begins when $E_{n}$ becomes of order $\varepsilon,$ {so that all the analysis made for Phases~I and~II, which rely on the assumption $\f{E_n}{\ep}\gg 1,$ is no more valid. However, we can infer the initial state at the beginning of Phase~III as the limit value obtained at the end of Phase~II: this is expressed in the following lemma.
\begin{lemma}\label{lem:initPhaseIII}With the notations of the previous results, at the end of Phase~II and beginning of Phase~III, we have 
\begin{equation}\label{eq:initPhaseIII}
E_n \approx \ep,\quad L_n{\approx} \f{1}{\ep},\quad T_n \approx \f{1}{\ep},\quad Y^n_{\max} \approx 1,\quad c_j(T_n)\sim \ep^2 \psi(\ep j),\quad j \gg 1,
\end{equation}
with $\psi$ defined by~\eqref{eq:PsiAs}. Moreover, we have  $v(T_n)=w(T_n)=O(\ep)$ and $c_j(T_n)\approx U(j)$ for $j=O(1).$
\end{lemma}}
\begin{proof}
{The energy $E_n\approx \ep$ is the definition of the beginning of Phase~III, so that we deduce from Prop.~\ref{prop:phasesIandII}  that $n\approx -\f{\log(\ep)}{\ep}.$ Hence, by~\eqref{eq:endPhaseII} we have $L_n {\approx}\f{1}{\ep}.$ Thanks to Lemmas~\ref{lem:scaling2} and~\ref{lem:DYT}, the maximal displacement for clusters $Y_{\max}^n\approx \f{E_n}{\ep}$ tends to become of order $1$, and the time period for one cycle $T_n\approx \f{E_n}{\ep^2}$ tends to be of order $\f{1}{\ep}$. }
 Since $E_{n}$ is of order $\varepsilon$ we obtain
that $w$ and $v$ are of order $\varepsilon$ too (cf. \eqref{D5}). 

For the size distribution, we recall its description done in Prop.~\ref{prop:UM}: a Dirac mass of weight $m_n\approx \sigma_{n} M,$ a boundary layer of width $\sigma_n$ in the variable $x=\f{j}{L_n}$ and the function $\f{\ep}{L_n} \psi(\f{j}{L_n})$ for $\f{j}{L_n}\gg \sigma_n.$ 
We now have $\sigma_n=\f{\sqrt{D_n}}{L_n}=\f{\sqrt{E_n}}{\sqrt{\ep}L_n} \approx \ep$ so that the boundary layer width vanishes: the approximation by  $\f{\ep}{L_n} \psi(\f{j}{L_n})\sim \ep^2 \psi(\ep j)$ is valid for any $j \gg 1$, and for the boundary layer we have $c_j\approx U(j).$
\end{proof}

\subsection{Dynamics of Phase III:  a new scaling}
\label{subsec:PhaseIII:dynamics}
{From this initial state,} it is
natural to introduce the following change of variables in order to describe
Phase~III:\ 
\begin{equation}
w=\varepsilon W\ \ ,\ \ v=\varepsilon V\ \ ,\ \ c_{k}=\varepsilon^{2}C_{k}\
\ ,\ \ \tau=\varepsilon t.  \label{F1}
\end{equation}
 {Notice that the variable $\tau$ here is different from the variable $\tau$ defined for Phases~I and~II.}
Then the equations~(\ref{eq:BD2})--(\ref{eq:BD3}) become%
\begin{align}
\frac{dV}{d\tau} & =-VW+V\left( 1-\varepsilon C_{1}\right),  \label{D1} \\
\frac{dW}{d\tau} & =VW-W,  \label{D2} \\
\frac{dC_{j}}{d\tau} & =\bar{J}_{j-1}-\bar{J}_{j},\ \ \forall j\geq 1,\qquad \bar{J}_{0}=0, \quad
\bar{J}_{j}=WC_{j}-VC_{j+1},\ \ j\geq1.  
\label{D3}
\end{align}

It is also natural to define a rescaled LV energy associated to the problem \eqref%
{D1}--\eqref{D3}. Notice that if we set $\varepsilon=0$ in \eqref{D1}--\eqref{D2}, we obtain the conservation of the rescaled LV\ energy $\tilde{E}$, i.e.
\begin{equation}\label{def:Etilde}
E=\ep\tilde{E}, \qquad\tilde{E}=V+W-2-\log\left( VW\right). 
\end{equation}
During the Phase III we expect the rescaled energy $%
\tilde{E}$ to be reduced from values of order one to values {of order {$\ep^2,$} which is the order of magnitude of its equilibrium value}. Notice
that the solutions of \eqref{D1}--\eqref{D3} are small perturbations of the
LV equation{, though the perturbation is different from the one in Phases I and II: here, for $\tilde E\approx 1,$ we have that $V$ and $W$ always remain in the order of one, and each phase of the LV cycle lasts for a time of order one in the variable $\tau$.} The perturbative term $-\varepsilon VC_{1}$ yields changes in
the rescaled energy $\tilde{E}$ {that is illustrated numerically in Figures  \ref{fig:sd_vw_ener_p3b} and \ref{fig:energy_p3_decay}}. 

The 
main difference between Phase III, described by means of \eqref{D1}--\eqref{D3}, and Phases I and II is that the change of energy of
the LV oscillations does not take place in the specific interval {$[T-\Delta t, t_4]$} of the LV
cycle, but during the whole cycle. Moreover, in the analysis of Phases I and
II, the main change of the energy during a LV cycle takes
place when the concentration wave is closer to cluster sizes $j$ of order one and during those times $c_{1}$ is much larger than the
concentrations $c_{j}$ with $j\neq1.$ On the contrary, during Phase III, the
rescaled concentrations $C_{j}$ with $j$ of order one have the same order of
magnitude as $C_{1}.$ Due to this, we cannot use a perturbative argument 
to approximate the values of $C_{1}$ as it was made in the analysis of
Phases I and II, but we need to study a problem which involves the whole
sequence $\left\{ C_{j}\right\} _{j\in\mathbb{N}}.$ 
Let us gather the results for Phase~III in the following proposition.

\begin{figure}[h!tb]
     \centering
     \begin{subfigure}[b]{0.48\textwidth}
         \centering
         \includegraphics[width=\textwidth]{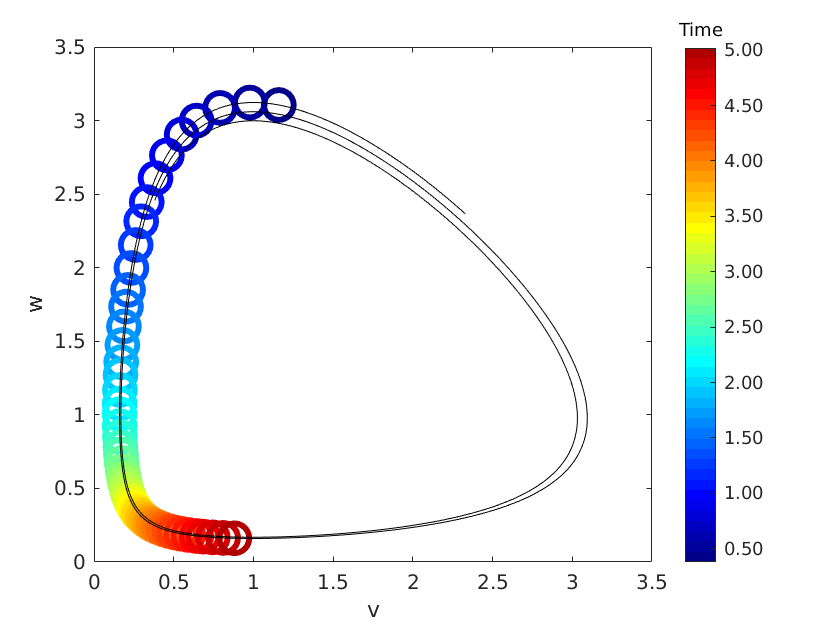}
         \caption{Monomers}
         \label{fig:vw_p3b_1}
     \end{subfigure}
     \hfill
     \begin{subfigure}[b]{0.48\textwidth}
         \centering
         \includegraphics[width=\textwidth]{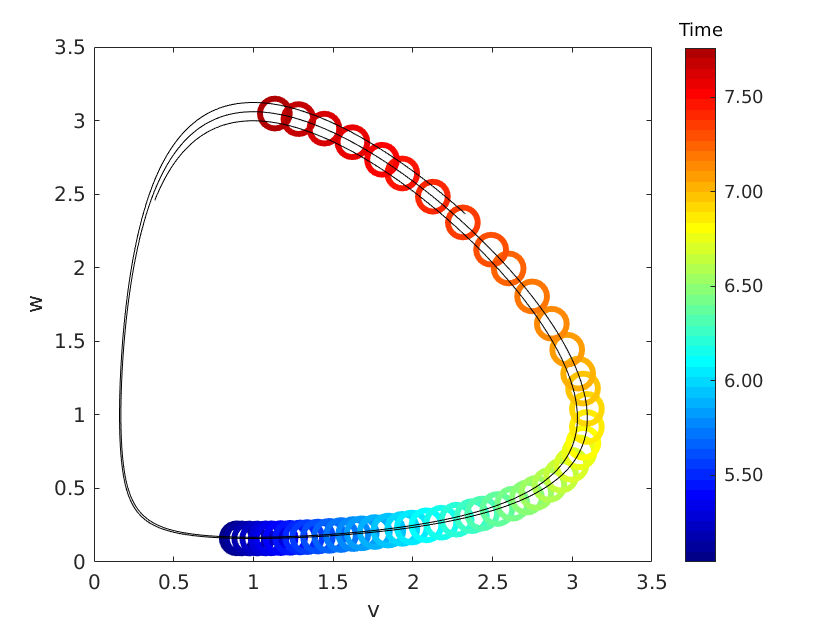}
         \caption{Monomers}
         \label{fig:vw_p3b_2}
     \end{subfigure}
     \hfill
     \begin{subfigure}[b]{0.48\textwidth}
         \centering
         \includegraphics[width=\textwidth]{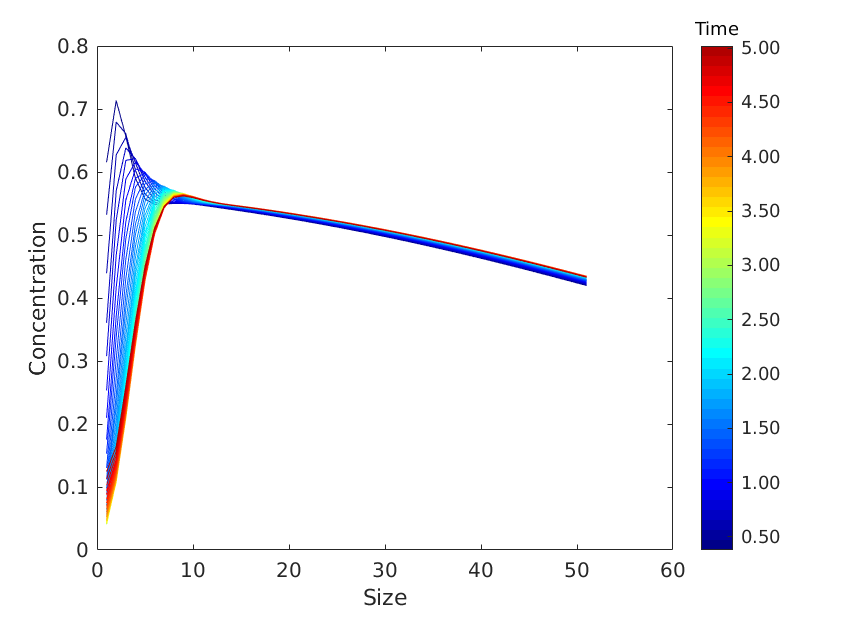}
         \caption{Size distribution}
         \label{fig:sd_p3b_1}
     \end{subfigure}
     \hfill
    \begin{subfigure}[b]{0.48\textwidth}
         \centering
         \includegraphics[width=\textwidth]{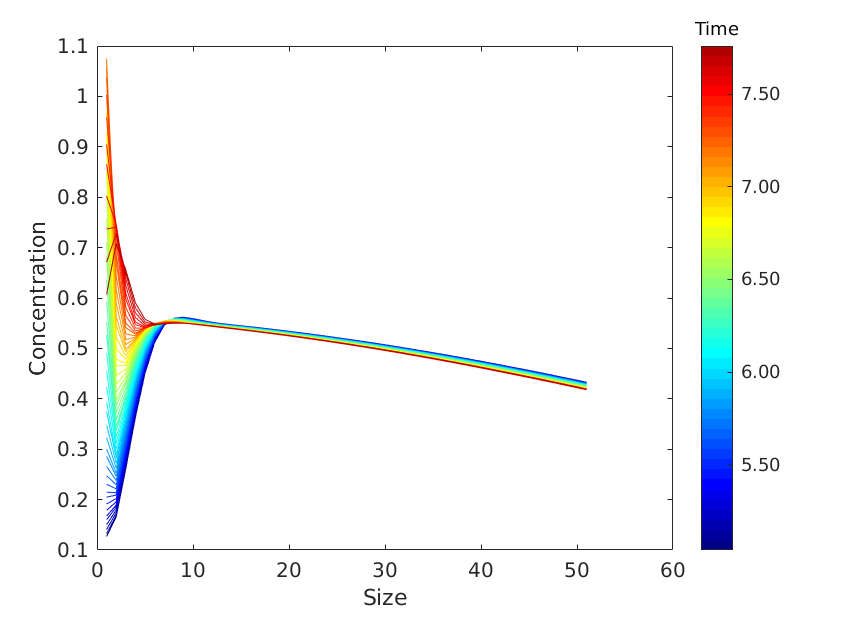}
         \caption{Size distribution}
         \label{fig:sd_p3b_2}
     \end{subfigure}
     \hfill
    \begin{subfigure}[b]{0.48\textwidth}
         \centering
         \includegraphics[width=\textwidth]{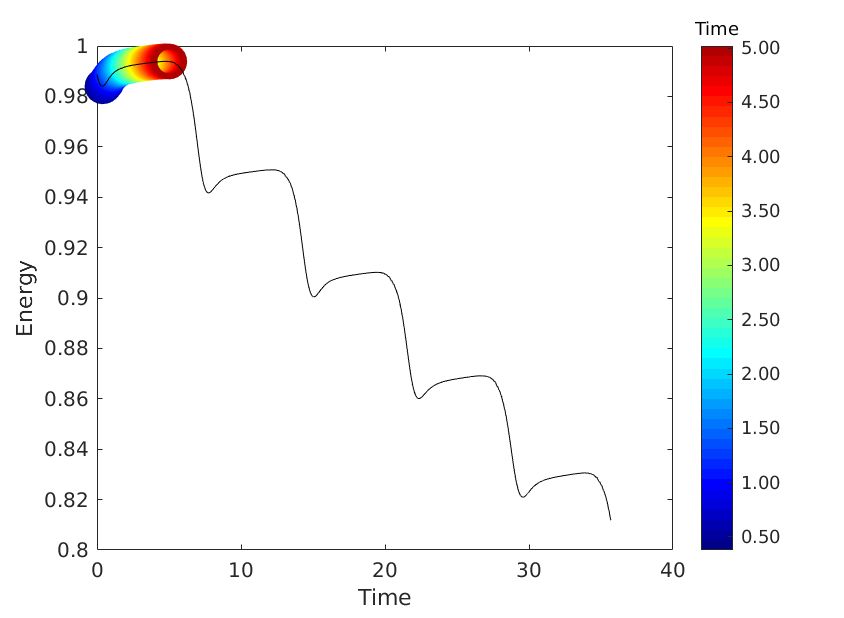}
         \caption{Energy}
         \label{fig:ener_p3b_1}
     \end{subfigure}
     \hfill
      \begin{subfigure}[b]{0.48\textwidth}
         \centering
         \includegraphics[width=\textwidth]{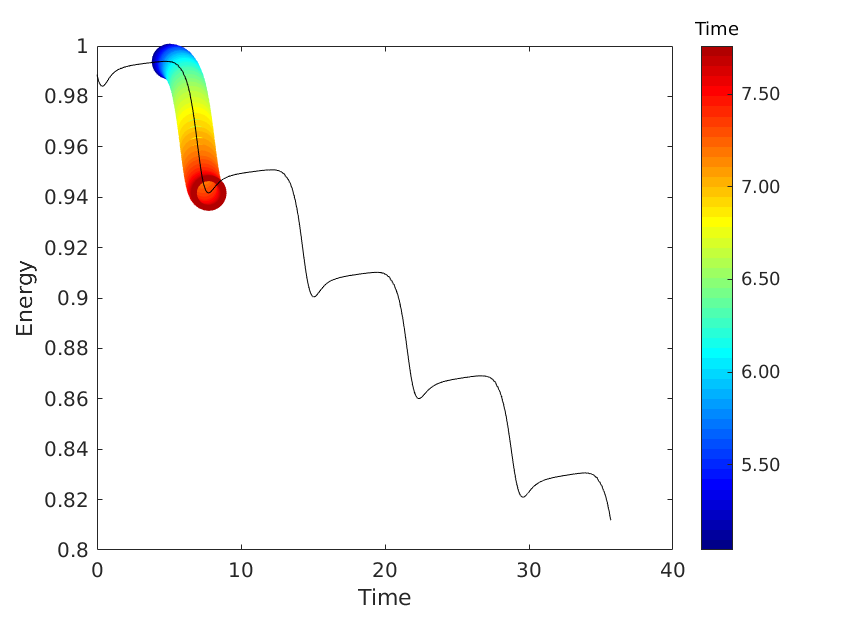}
         \caption{Energy}
         \label{fig:ener_p3b_2}
     \end{subfigure}
        \caption{Numerical illustration of the evolution of the size distribution of the clusters, the monomers and the energy (\ref{def:Etilde}) in Phase III over one cycle (c.f the rescaled system described in the equations \eqref{D1}--\eqref{D3}). The time evolution indicated by the colors is the same for the three Figures \ref{fig:vw_p3b_1}, \ref{fig:sd_p3b_1}  and \ref{fig:ener_p3b_1} (resp. \ref{fig:vw_p3b_2}, \ref{fig:sd_p3b_2} and \ref{fig:ener_p3b_2}). The Figures \ref{fig:sd_p3b_1} and \ref{fig:sd_p3b_2} are truncated and only show the evolution of the size distribution for sizes smaller than 50. The time $t=0$ corresponds to the end of Phase II and the beginning of Phase III in this simulation. }
        \label{fig:sd_vw_ener_p3b}
\end{figure}

{\begin{proposition}
\label{prop:PhaseIII}
Let $T^{in}_3$ denote 
the beginning of Phase III.
Let us depart from $E_n,$ $L_n$, $T_n$ and $c_j(T_n)$ described by~\eqref{eq:initPhaseIII} and define $V,\,W,\;\tau$ and $C_j$ by~\eqref{F1}, $\tilde E_n:={\tfrac{1}{\ep}} E_n$ and $\tau^{in}_n=\ep T^{in}_3$. On each cycle $\tau_n > \tau^{in}_n$, we can approximate the dynamics with the following system taken on $[\tau_{n}, \tau_{n+1}]$:
\begin{align}
\frac{dV}{d\tau} & =-VW+V,\label{G1}
\\
\frac{dW}{d\tau}  &=VW-W  \label{G2}, \\
\frac{dC_{j}}{d\tau} & = \frac{W-V}{2} \left(
C_{j-1}-C_{j+1}\right) +\frac{V+W}{2}\left( C_{j-1}-2C_{j}+C_{j+1}\right),\; \;  j \geq 2,  \label{G3} \\
\frac{dC_{1}}{d\tau} & =VC_{2}-WC_{1}, \quad  {\tilde E_k}= V+W-2-\log(VW),\quad V(0)=W(0)>1,
\label{G3a}
\end{align}
and with the additional matching condition%
\begin{equation}
\lim\limits_{j\to \infty} C_{j}\left( \tau\right)= \frac{2}{\pi }\qquad  \forall \tau \in 
[\tau_{n}, \tau_{n+1}].
\label{G4}
\end{equation}
The change of energy between two successive cycles is then given by
\begin{equation}
\tilde{E}_{n+1}-\tilde{E}_{n}\approx \varepsilon\int_{\tau_{n}}^{\tau_{n+1}}\left( 1-V( \tau) \right) C_{1}\left(
\tau\right) d\tau  \label{eq:EnChPh3},
\end{equation}
with $(C_1,V)$ solutions to \eqref{G1}--\eqref{G4}. 
\end{proposition}}
\begin{proof}
Let us first prove that $C_{j}\approx \psi(\ep j)$ for
$j\approx \f{1}{\ep}$: this implies the matching condition~\eqref{G4}, {the system~\eqref{G1}--\eqref{G3a} being written for $j$ large but $\ep j$ small. This is obtained through a transport-diffusion equation as for Phases I and II.}  
We depart from the initial condition given by Lemma~\ref{lem:initPhaseIII}
\begin{equation}
C_{j}{(\tau^{in}_n)}=\psi\left( {\varepsilon}j\right),\qquad {\tau^{in}_n=\ep T^{in}_3} \ \ \ \text{for }j%
\text{ large}\ ,  \label{eq:CjRes}
\end{equation}
We rewrite
(\ref{D3}) as {a discrete transport-diffusion equation, as we did in~\eqref{eq:DiffForm1}}%
\begin{equation}
\frac{dC_{j}}{d\tau}=\left( \frac{W-V}{2}\right) \left(
C_{j-1}-C_{j+1}\right) +\frac{V+W}{2}\left( C_{j-1}-2C_{j}+C_{j+1}\right). \label{eq:C:transportdiff}
\end{equation}
Denoting as $x$ the variable $\ep j$ and approximating the
discrete derivatives in the previous formula by means of continuous
derivatives  {as we did in~\eqref{eq:DiffApp}} (writing $C_{j}\left( \tau\right) $ as $C\left( x,\tau\right) $
with $x=\ep j$), we obtain%
\begin{equation}
\frac{\partial C\left( x,\tau\right) }{\partial\tau}={\varepsilon}%
\left( V-W\right) \frac{\partial C\left( x,\tau\right) }{\partial x}+\frac{%
V+W}{2}\ep^{2}\frac{\partial ^{2}C\left(
x,\tau\right) }{\partial x^{2}}.  \label{eq:DiffPh3}
\end{equation}
Equation~(\ref{eq:DiffPh3}) allows to estimate the variation of the
concentrations $C_{j}$ with large $j:$  
\begin{itemize}
\item The periods $\tau_{n}>\tau^{in}_{n}$ of each LV cycle is of order one during Phase III.
\item $\left( V-W\right) $ is periodic in each LV cycle, hence  the transport rate $\ep \left(
V-W\right)$ cause an oscillation of $C\left( x,\tau\right)$ in the variable $x$ of order $\varepsilon$    or equivalently, in the variable $j$ the amplitude of the oscillations is of order one. 
\item Since the second order term $%
\frac{V+W}{2}\ep^{2}\frac{%
\partial^{2}C\left( x,\tau\right) }{\partial x^{2}}$ is $O(\ep^2)$, significant changes in
the shape of the concentrations $C\left( x,\tau\right) $ only happen over times $\tau$ of order $%
\frac{1}{\varepsilon^{2}}$.
\end{itemize}
To sum-up, over times $\tau$ much smaller than $\frac{1}{%
\varepsilon^{2}}$, we can neglect the modifications of the
concentrations $C( x,\tau_{n})$, and only small oscillations during each period. Notice that the equations \eqref{D1}--\eqref{D3} suggest that there is a change of the rescaled LV energy $%
\tilde{E}$ of order $\varepsilon$ in each LV cycle {as long as $\tilde E$ is of order one, so that the number of LV cycles needed for $\tilde E$ to become small is of order $\f{1}{\ep}$}. 
{The end of Phase~III, characterised by $\tilde{E}$ small, is studied below in Prop.~\ref{prop:endPhaseIII}, and concerns {a number of cycles in the order $-\f{\log(\ep^2)}{\ep}.$}
Since one LV period is of order one in the variable $\tau$ during the whole Phase~III, we deduce that the total duration of Phase~III is of order $-\f{\log(\ep^2)}{\ep}\ll \f{1}{\ep^2}$ in the variable $\tau,$ hence}
the external
concentrations $C\left( x,\tau\right)\approx \psi (x)$ remain frozen during Phase
III. 

This allows to obtain the matching condition~\eqref{G4} {(cf. Prop \ref{prop:UM})}. 
We then obtain~\eqref{G1}--\eqref{G3a}  to approximate $C_{1}$
to the leading order. 
We expect to have a stable periodic solution of \eqref{G1}--\eqref{G4} for
each value of the rescaled energy $\tilde{E}.$ 
Notice that the solution of \eqref{G1}-- \eqref{G2} can be
computed for each value of $\tilde{E}\geq0$ independently on the values of $%
C_{j}.$ Thus, \eqref{G3}--\eqref{G4} become a system of infinitely many ODEs
with prescribed functions $W,\ V$ for each value of $\tilde{E}.$

We end the proof of Prop.~\ref{prop:PhaseIII} by computing
 the change of rescaled energy $\tilde{E}$ during each
cycle: we use that
\begin{equation*}
\frac{d\tilde{E}}{d\tau}=\varepsilon\left( 1-V\right) C_{1}, 
\end{equation*}
which yields~\eqref{eq:EnChPh3}.
\end{proof}
\\

\begin{remark}Contrarily to Phase~II, where~\eqref{eq:Echn} provided an explicit formula to estimate the decay of energy, we have in Phase III no way to approximate~\eqref{G1}--\eqref{G4} in order to estimate the right-hand side of~\eqref{eq:EnChPh3}: in general this has to be done numerically, except at the end of Phase III, for which we have the following result.
\end{remark}
\subsection{End of Phase~III: {energy decay}}
\label{subsec:PhaseIII:end}
As already said, we expect to have a number of cycles in the order of $\f{1}{\ep}$ before reaching $\tilde E_n \ll 1.$
We now compute the integral on the right-hand side of (\ref{eq:EnChPh3}) at the end of Phase~III, {\it i.e.} for  $%
\tilde{E}_{n}\rightarrow0.$  {We obtain the following result.}
{\begin{proposition}\label{prop:endPhaseIII}Under the assumptions and notations of Prop.~\ref{prop:PhaseIII}, for the range $\ep^2 \ll \tilde E_n \ll 1$, we can approximate the decay of energy by
\begin{equation}
\tilde E_{n+1}-\tilde E_n \approx - \ep a \tilde E_n,
\end{equation}
for a given constant $a>0$ of order one. Hence the energy decays exponentially fast at the end of Phase~III.
The assumptions of Prop.~\ref{prop:PhaseIII} are valid until $\tilde E_n \approx \ep^2,$ which happens after a number of cycles in the order of $\f{1}{a\ep}\log(\f{1}{\ep^2})$. At the end of this period, we thus have $\tilde E_n=O(\ep^2),$ $V-1=O(\ep),$ $W-1=O(\ep),$ $C_j\approx \psi(\ep j)$ for $j\approx \f{1}{\ep}.$
\end{proposition}}
{\begin{remark} We notice that, as at the end of Phase~II and Phase~III, and despite the fact that the scalings differ, the energy decays exponentially, with a rate in the order of $\ep$ in both cases (cf. Figures \ref{fig:ener_p2} and \ref{fig:energy_p3_decay}). However the constants ($A$ and $a$) are different.
\end{remark}}
{\begin{remark}
    The energy decay at the end of Phase III is also characterizing the damping of the oscillations of the trajectories of the monomers' concentrations. Proposition \ref{prop:endPhaseIII} proves that the trajectories $(V,W)$ are enclosed in a ball of center $(1,1)$ and radius $\ep$. 
    However, one can note that the steady-state \eqref{eq:steady} for the monomers is shifted from the point $(1,1)$, hence more precise estimate can be found for $V-1$ after $\f{1}{a\ep}\log(\f{1}{\ep^2})$ cycles.
    The damping of the oscillations are studied more precisely in section \ref{sec:PhaseIV} and the results will be refined in Lemma \ref{lem:endPhaseIII} using a more precise asymptotic expansion when linearizing around the steady-state.
\end{remark}}
\begin{figure}
    \centering
    \includegraphics[width=0.75\textwidth]{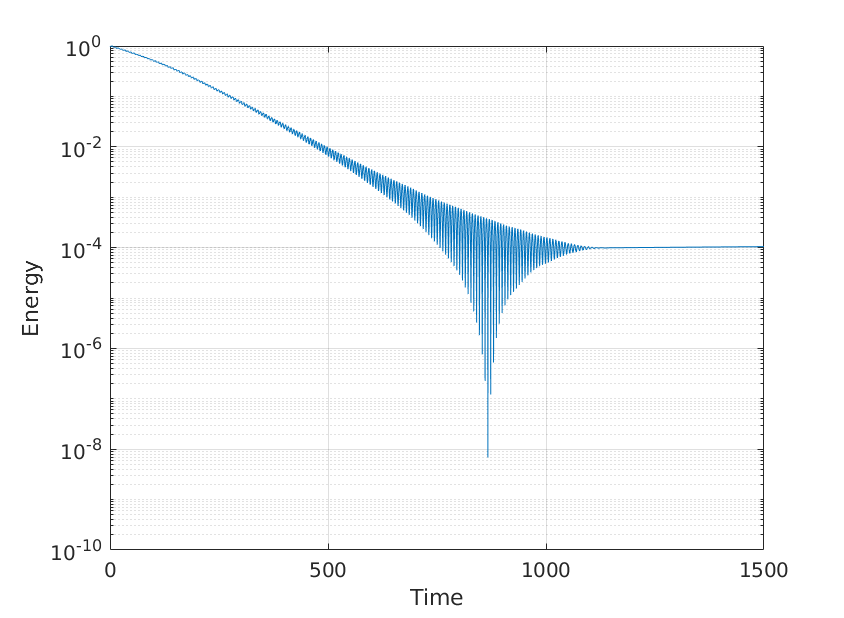}
    \caption{Numerical simulation of the time evolution of the Energy $\tilde E$ defined in~\eqref{def:Etilde} during Phase III and the beginning of Phase IV. The time $t=0$ corresponds to the end of Phase~II and the beginning of Phase~III in this simulation. The y-axis is logarithmic and $\varepsilon = 0.02$. The decay of the energy is exponential-like until it reaches an order of magnitude of $\varepsilon^2\approx 10^{-4}$; { the order of magnitude of the energy  oscillations then become of the same order of magnitude as the energy itself, and then decay (end of Phase~III). 
    }}
    \label{fig:energy_p3_decay}
\end{figure}
\begin{proof}
Let us linearize the set of equations 
\eqref{G1}--\eqref{G4}, writing%
\begin{equation*}
V=1+\alpha\ \ ,\ \ W=1+\beta\ \ ,\ \ C_{j}=\frac{2}{\pi }\left( 1+\eta
_{j}\right) 
\end{equation*}
where $\left\vert \alpha\right\vert ,\ \left\vert \beta\right\vert $ and $%
\left\vert \eta_{j}\right\vert $ are small. Then, keeping only the linear
terms in $\alpha,\ \beta$ and $\eta_{j}$, we obtain the following equations%
\begin{align}
\frac{d\alpha}{d\tau} & =-\beta  \label{H1},\\
\frac{d\beta}{d\tau}  &=\alpha  \label{H2}, \\
\frac{d\eta_{j}}{d\tau} & =\left( \eta_{j-1}-2\eta_{j}+\eta_{j+1}\right),\ \  j \geq 2,  \label{H3} \\
\frac{d\eta_{1}}{d\tau} & =\left( \alpha-\beta\right) +\eta_{2}-\eta _{1},
\qquad \qquad
\lim\limits_{j\to\infty} \eta_{j}\left( \tau\right)  =0\ \ {\text{for all }}%
\tau.  \label{H5}
\end{align}
The rescaled energy $\tilde{E}$ can be approximated for small values as%
\begin{equation*}
\tilde{E}=\frac{1}{2}\left( \alpha^{2}+\beta^{2}\right). 
\end{equation*}
Then, the solution of \eqref{H1}--\eqref{H2} can be written, up to a
translation of the origin of time $\tau$ as%
\begin{equation*}
\alpha=\sqrt{2\tilde{E}}\cos\left( \tau\right), \ \ \ \ \beta=\sqrt {2\tilde{%
E}}\sin\left( \tau\right). 
\end{equation*}
We notice that we now have periods of the LV cycle of $\tau_n\sim 2\pi$.
In fact, it is enough to solve
the reference problem
\begin{align}
\frac{d\varphi_{j}}{d\tau} & =\left(
\varphi_{j-1}-2\varphi_{j}+\varphi_{j+1}\right), \ \ \ \ j\geq2  ,\label{Hm1}
\\
\frac{d\varphi_{1}}{d\tau} & =e^{i\tau}+\varphi_{2}-\varphi_{1},  \label{Hm2}
\\
\lim_{j\rightarrow\infty}\varphi_{j}\left( \tau\right) & =0\ \ \text{for all 
}\tau .  \label{Hm3}
\end{align}
since solutions $\left\{ \eta_{j}\right\} _{j\in\mathbb{N}}$ to
\eqref{H3}--\eqref{H5}
are recovered from solutions $\left\{ \varphi_{j}\right\} _{j\in\mathbb{N}}$ to 
\eqref{Hm1}--\eqref{Hm3} by means
of%
\begin{equation*}
\eta_{j}=\sqrt{2\tilde{E}}\left[ \func{Re}\left( \varphi_{j}\right) -\func{Im%
}\left( \varphi_{j}\right) \right] \ \ ,\ \ j\geq1. 
\end{equation*}
We look for solutions of the problem \eqref{Hm1}--\eqref{Hm3} in the form%
\begin{equation*}
\varphi_{j}=A_{j}e^{i\tau} 
\end{equation*}
where the coefficients $A_{j}$ solve%
\begin{align}
iA_{j} & =A_{j-1}-2A_{j}+A_{j+1},\ \ \ \ j\geq2  \label{Am1} \\
iA_{1} & =A_{2}-A_{1}+1.  \label{Am2}
\end{align}
We can look for particular solutions of (\ref{Am1}) in the form%
\begin{equation*}
A_{j}=\left( r\right) ^{j-1}\ \ ,\ \ j\in\mathbb{N},
\end{equation*}
which yields the two roots
\begin{equation*}
r_{\pm}=\frac{1}{2}\left[ \left( 2+i\right) \pm\sqrt{4i-1}\right].
\end{equation*}
We compute numerically $\left\vert
r_{-}\right\vert \approx 0.48053 <1$ while  $\left\vert r_{+}\right\vert >1$.
Therefore, in order to obtain a solution $\varphi_{j}$
satisfying \eqref{Hm3}, we must have%
\begin{equation*}
A_{j}=K_{0}\left( r_{-}\right) ^{j-1},\ \ j\in\mathbb{N} 
\end{equation*}
for some $K_{0}\in\mathbb{C.}$ 
In order to determine $K_{0}$ we use (\ref%
{Am2}). Then%
\begin{equation*}
K_{0}\left[ \left( 1+i\right) -r_{-}\right] =1. 
\end{equation*}
Therefore,%
\begin{equation*}
\varphi_{j}=\frac{\left( r_-\right) ^{j-1}}{\left[ \left( 1+i\right) -r_{-}%
\right] }e^{i\tau}=\frac{\left( r_-\right) ^{j-1}}{\left[ \frac{i}{2}+\sqrt{%
{ i-\f{1}{4}}}\right] }e^{i\tau} ,\ \ \ j \geq 1.
\end{equation*}
Thus,%
\begin{equation*}
\eta_{j}=\sqrt{2\tilde{E}}\left[ \func{Re}\left( \frac{\left( r_-\right) ^{j-1}%
}{\left[ \frac{i}{2}+\sqrt{{i-\f{1}{4}}}\right] }e^{i\tau}\right) -\func{Im}\left( 
\frac{\left( r_-\right) ^{j-1}}{\left[ \frac{i}{2}+\sqrt{{i-\f{1}{4}}}\right] }%
e^{i\tau}\right) \right], \ \ j\geq 1. 
\end{equation*}
or, equivalently%
\begin{equation*}
\eta_{j}=\sqrt{2\tilde{E}}\func{Re}\left( \frac{\left( r_-\right) ^{j-1}\left(
1+i\right) }{\left[ \frac{i}{2}+\sqrt{{i-\f{1}{4}}}\right] }e^{i\tau }\right) =\sqrt{%
2\tilde{E}}\func{Re}\left( \left( 1+i\right) \varphi_{j}\right) \ \ ,\ \ j\geq 1.\end{equation*}

We can now approximate the integral on the right-hand side of (\ref{eq:EnChPh3}) as follows:

\begin{align*}
\tilde E_{n+1}-\tilde E_n
&= \varepsilon\int_{0}^{\tau_n }\left(
1-V\left( \tau;\tilde{E}_{n}\right) \right) C_{1}\left( \tau;\tilde{E}%
_{n}\right) d\tau \\
& =-\frac{2\varepsilon}{\pi }\int_{0}^{\tau_n}\alpha\bigl( \tau;\tilde{E}_{n}\bigr) \bigl( 1+\eta _{1}\bigl(
\tau;\tilde{E}_{n}\bigr) \bigr) d\tau \\
& =-\frac{2 \sqrt{2\tilde{E}}\ep }{\pi }\int_{0}^{\tau_n }\cos\left( \tau\right) \eta_{1}\bigl( \tau;%
\tilde{E}_{n}\bigr) d\tau \\
& =-\frac{4\varepsilon\tilde{E}}{\pi }\func{Re}\left( \left( 1+i\right)
K_{0}\int_{0}^{2\pi}\cos\left( \tau\right) e^{i\tau}d\tau\right) \\
& =-\frac{4\varepsilon\tilde{E}}{\pi }\func{Re}\left( \left( 1+i\right)
K_{0}\int_{0}^{2\pi}\cos^{2}\left( \tau\right) d\tau\right) 
=-{4\varepsilon\tilde{E}}{}\func{Re}\left( \left(
1+i\right) K_{0}\right)
\end{align*}
with
\begin{equation*}
\func{Re}\left( \left( 1+i\right) K_{0}\right) =\func{Re}\left( \frac{\left(
1+i\right) }{\frac{i}{2}+\sqrt{{ i-\f{1}{4}}}}\right) ={ 0.92505} >0. 
\end{equation*}
Then%
\begin{equation}
\tilde{E}_{n+1}-\tilde{E}_{n}=-{4\varepsilon\tilde{E}_{n}}{}\func{Re}%
\left( \left( 1+i\right) K_{0}\right) =-{ a} \varepsilon\tilde{E}_{n}\ \ 
\text{where } { a \approx 3.6922}>0. \label{eq:EnExp}
\end{equation}
Phase~III continues until the time in which the contribution of the term 
$-\varepsilon VC_{1}$ in (\ref{D1}) becomes comparable to the term $%
-VW+V=-V\beta,$ {so that the approximation of~\eqref{D1} by~\eqref{G1} becomes invalid.} Therefore, we need to compare the terms $\varepsilon C_{1}$
and $\beta.$ Since $C_{1}$ remains of order one during the whole phase,  Phase III ends when $\beta$ becomes of order $%
\varepsilon$. The order of magnitude of $\beta$ is of order $\sqrt{\tilde{E}%
_{n}}.$ Thus, Phase III is valid as long as $\tilde{E}_{n}\gtrsim%
\varepsilon^{2}.$ Notice that \eqref{eq:EnExp} yields an exponential decay
for $\tilde{E}_{n}$ with the form $\exp\left( -a n\varepsilon\right) . $
It then follows that $\tilde {E}_{n}$ becomes of order $\varepsilon^{2}$
after $\frac{1}{a\varepsilon }\log\left( \frac{1}{\varepsilon^{2}}%
\right) $ LV cycles. 
\end{proof}

\section{{Phase IV: Oscillations decay and stabilization} }
\label{sec:PhaseIV}

As outlined in the previous section, the final Phase IV starts once the energy level is in the order of $\ep^3$ (or $\ep^2$ for the rescaled variables of Phase III). We recall that this is also the order of magnitude of our system's equilibrium energy value $\bar E,$ see~\eqref{eq:steady:energy}. We decompose Phase~IV into two successive stages: first, a phase similar to the end of Phase~III, where now the oscillations (yet no longer the energy) decay and finally become negligible. Secondly, the trend to equilibrium occurs through an approximate parabolic equation, with non-oscillatory coefficients.

\subsection{End of Phase III and early 
Phase~IV: decay of the oscillations}
\label{subsec:PhaseIV:decay}
In the previous section, Proposition \ref{prop:endPhaseIII} gives the decay of the energy until the end of Phase III as well as the characterization of the monomers and the cluster distribution until the end of Phase III, i.e. until $\tilde E_n \approx \ep^2$, resp. $E_n \approx \ep^3$.
The following lemma also characterizes the end of Phase III, but allows us to describe more precisely the damping of the oscillations occurring at the very end of Phase III and the beginning of Phase IV.

\begin{lemma}\label{lem:endPhaseIII}
Under the assumptions of Prop. \ref{prop:endPhaseIII}, the oscillations become negligible after a number of cycles in the order of $\frac{\log\left(\ep^{-1}\right)}{\ep}$ and we have 
\begin{equation*}
    V=1+O(\ep^{3/2}),\quad W=1-\ep C_\infty + O(\ep^{3/2}),\quad C_1=C_\infty + O(\ep). 
\end{equation*}

During these cycles, the changes in the size distribution  $C_j$ have been negligible, so that $C_\infty =\f{2}{\pi}$ and $C_j \approx \psi (\ep j)$.
\end{lemma}
\begin{proof}
In~\eqref{D1}, {when linearizing and taking $W = 1 + \beta$ (cf. Proof of Prop. \ref{prop:endPhaseIII}),} the term $-\varepsilon VC_{1}$ becomes comparable to $-V\beta$
and it cannot be any longer ignored to the leading order. In order to
understand the evolution of $W,\ V,\ C_{j}$ during this phase, 
it is natural to introduce
new variables $\tilde{\alpha}$, $\tilde \beta$ {and $\eta_j$} defined by
\begin{equation*}
V=1+\varepsilon\tilde{\alpha},\qquad  W=1{ -\ep C_\infty }+\varepsilon\tilde{\beta},\qquad {C_j=C_\infty(1+\ep\tilde \eta_j),}
\end{equation*}
{where $C_\infty=\f{2}{\pi}$ at the  "beginning of the very end" of Phase~III but is expected to slowly change during Phase IV.}
Then, keeping the leading terms in \eqref{D1}--\eqref{D3} we obtain the
following approximate model 
\begin{align}
\frac{d\tilde{\alpha}}{d\tau} & =-\tilde{\beta}{ -\varepsilon \tilde{%
\alpha}\tilde \beta -\ep C_\infty \tilde \eta_1 }\label{Da1a} \\
\frac{d\tilde{\beta}}{d\tau} & =\tilde{\alpha} {(1-\ep C_\infty)+\ep\tilde\alpha \tilde\beta}  \label{Da2a} \\
{ \frac{d\ti\eta_{j}}{d\tau} }& ={\varepsilon\biggl( \frac{\tilde{\beta}-\tilde{\alpha }%
}{2}\biggr) \left( \ti\eta_{j-1}-\ti\eta_{j+1}\right) +\left(
\ti\eta_{j-1}-2\ti\eta_{j}+\ti\eta_{j+1}\right) }\\ \nonumber
& {-\ep C_\infty(\ti\eta_{j-1}-\ti\eta_j) +\ep\f{\ti\beta+\ti\alpha}{2}(\ti\eta_{j-1}-2\ti\eta_j+\ti\eta_{j+1}) }\  ,\ \ j\geq2   \\
\frac{d\tilde \eta_{1}}{d\tau} & =\tilde\eta_{2}-\tilde \eta_1 +\tilde\alpha-\tilde\beta+C_\infty+\ep(\tilde \alpha\tilde\eta_2-\tilde\beta\tilde\eta_1+C_\infty\ti\eta_1)\ .  \label{Da4a}
\end{align}
We keep only the leading order terms to get
\begin{align}
\frac{d\tilde{\alpha}}{d\tau} & =-\tilde{\beta}\label{Da1} \\
\frac{d\tilde{\beta}}{d\tau} & =\tilde{\alpha}   \label{Da2} \\
{\frac{d\ti\eta_{j}}{d\tau} }& =
\ti\eta_{j-1}-2\ti\eta_{j}+\ti\eta_{j+1}\ ,\quad j\geq2  \label{Da3} \\
{\frac{d\tilde \eta_{1}}{d\tau} }& ={\tilde\eta_{2}-\tilde \eta_1 +\tilde\alpha-\tilde\beta+C_\infty.}  \label{Da4}
\end{align}
We notice that the system obtained only differs from~\eqref{H1}--\eqref{H5} by the source term $+C_\infty$ in the equation for $\tilde\eta_1,$ {which expresses the fact that there is a constant non-negligible difference between the influx of monomers (due to the depolymerisation rate $V$) and the outflux due to the polymerisation $W\approx V-C_\infty.$ }

We first solve~\eqref{Da1}--\eqref{Da2} as we did for~\eqref{H1}--\eqref{H2}: we notice that solutions to \eqref{Da1}--\eqref{Da2} are  $2\pi$ periodic (as for ~\eqref{H1}--\eqref{H2}), which implies that periods $\tau_n$ of \eqref{Da1a}-\eqref{Da2a} are $2\pi+O(\ep)$. With respect to the energy of \eqref{Da1a}--\eqref{Da2a}, we have
\[\tilde E \approx \ep^2 \f{\tilde \alpha^2 + (\tilde \beta - C_\infty)^2}{2}=: \ep^2 \hat E +\ep^2\f{C_\infty^2 - 2\tilde \beta C_\infty}{2},
\]
where we have defined $\hat E:=\f{\tilde\alpha^2 +\tilde\beta^2}{2},$ so that up to choosing a proper initial time we have $\tilde\alpha=\sqrt{2 \hat E}\cos (\tau)$ and $\tilde \beta=\sqrt{2\hat E}\sin (\tau).$ The oscillations of the energy $\tilde E$ are now of the same order of magnitude as its value {{(see Fig.~\ref{fig:energy_p3_decay}, values around $E=10^{-4}$, for an illustration)}}, as they are expressed by the term $-\ep^2 \tilde\beta C_\infty,$ so that we also define its average over a time period:
\[<\tilde E>=\ep^2 {\Delta_{\hat E}}  + \ep^2 \f{C_\infty^2}{2},
\]
so that it tends to its equilibrium $\ep^2 \f{C_\infty^2}{2}$ when ${\Delta_{\hat E}}$ vanishes. We also compute
\begin{align}
\f{d\hat E}{d\tau}=\ep^{-2} \f{d\tilde E}{d\tau}+C_\infty \f{d\tilde\beta}{d\tau}=
{ - \ep \tilde \alpha C_\infty  \tilde\eta_1 - \ep\tilde \alpha\tilde \beta (C_\infty  +\tilde \alpha-\tilde \beta),}
\label{eq:hatE}
\end{align}

so that over a period $\tau_n = 2\pi +O(\ep)$, we have 
\begin{equation}
    {\Delta_{\hat E}} { \sim } -\ep C_\infty \f{1}{2\pi}{\int_{\tau_n}^{\tau_n+2\pi}} \tilde \alpha(s) \tilde \eta_1 (s) ds, 
    \label{eq:EnExp2}
\end{equation}
and we are left to compute $\tilde \eta_1$ in order to estimate the decay of energy.
\\

\paragraph{Solution to~\eqref{Da3}--\eqref{Da4}.}
Recognizing in~\eqref{Da3}--\eqref{Da4} a linear system with two sources $\tilde\alpha-\tilde\beta$ and $C_\infty$, we superimpose the solution $\tilde \eta_j^{(1)}$ to~\eqref{Da1}--\eqref{Da4} {with} $C_\infty=0$ -- this is exactly the solution to~\eqref{H3}--\eqref{H5} computed above  for $\eta_j$, where we simply replace $\alpha,\,\beta,\,\tilde E$ by $\tilde \alpha,\,\tilde\beta$ and $\hat E$ respectively  -- with the solution $\tilde \eta_j^{(2)}$ of the system with a constant source $C_\infty,$ namely
\begin{align}
{\frac{d\ti\eta_{j}^{(2)}}{d\tau} }& =
\ti\eta_{j-1}^{(2)}-2\ti\eta_{j}^{(2)}+\ti\eta_{j+1}^{(2)},\quad j\geq2  \label{Dabis3} \\
{\frac{d\tilde \eta_{1}^{(2)}}{d\tau} }& ={\tilde\eta_{2}^{(2)}-\tilde \eta_1^{(2)} +C_\infty.}  \label{Dabis4}
\end{align}
To solve~\eqref{Dabis3}--\eqref{Dabis4}, let us compute the system satisfied by its Laplace transform $L_j(z):=\int_0^\infty e^{-z\tau }\tilde\eta_j^{(2)}(\tau)d\tau:$
\begin{align*}zL_1(z)&=L_2(z)-L_1(z)+\f{C_\infty}{z},
\\
zL_j(z)&=L_{j-1}(z)-2L_j(z)+L_{j+1}(z).
\end{align*}
We look for solutions $L_j(z)=B(z)\theta(z)^{j-1}$, to find
\[\theta^2 -\theta (2+z) +1=0,
\]
so that 
\[\theta_{\pm}(z)=1+\f{z}{2} \pm \sqrt{z+\f{z^2}{4}}.
\]
Since we want $|\theta|<1$ to ensure that the solutions vanish when $j\to\infty,$ we have to choose $\theta_-(z)$ which satisfies $|\theta_-(z)|<1$ for $z\notin \{0,4\}.$ We then compute $B(z)$ with the equation for $L_1,$ and find
\[zB(z)=B(z)(\theta_-(z)-1)+\f{C_\infty}{z},
\]
so that finally
\[L_j(z)=\f{C_\infty \theta_-(z)^{j-1}}{z\left(z-\theta_-(z)+1\right)} =\f{C_\infty \left(1+\f{z}{2} - \sqrt{z+\f{z^2}{4}}\right)^{j-1}}{z\left(\f{z}{2}+\sqrt{z+\f{z^2}{4}}\right)}. \]
We can now compute the inverse Laplace transform, and obtain, taking $\delta >0,$
\[\ti\eta_j^{(2)}(z)(\tau)=\f{C_\infty}{2i\pi} \int_{\delta-i\infty}^{\delta+i\infty} e^{z\tau}\f{\left(1+\f{z}{2} - \sqrt{z+\f{z^2}{4}}\right)^{j-1}}{z\left(\f{z}{2}+\sqrt{z+\f{z^2}{4}}\right)}dz. \]
We can deform the contour of integration $(\delta-i\infty,\delta+{i\infty})$ by making $\delta \to 0$ and taking the contour {$(-\infty-i\delta,-i\delta)\cup \{|z|=\delta, Re(z)>0\}\cup(i\delta,-\infty+i\delta)$} so that the main contribution for the inverse Laplace transform comes from the half circle ${\cal C}:= \{|z|=\delta, Re(z)>0\},$ on which we compute
\[\ti\eta_1^{(2)}{\approx} \f{C_\infty}{2i\pi} \int_{\cal C} e^{z\tau} z^{-\f{3}{2}} dz\approx \f{C_\infty \sqrt{\tau}}{2i\pi} \int_{\cal C} \f{e^\eta d\eta}{\eta^{3/2}}{\approx} \f{C_\infty \sqrt{\tau}}{i\pi} \int_{\cal C} \f{e^\eta d\eta}{\eta^{1/2}}=C \f{C_\infty \sqrt{\tau}}{i\pi},
\]
for a given constant $C.$ This allows us to compute the change of energy due to the contribution of $\ti\eta_1^{(2)},$ see below. 

{Let us also verify that for $j\gg 1$ the influence of $\ti\eta_j^{(2)}$ remains negligible on the timecourse of this phase, which is in the order of $\f{1}{\ep},$ so that the large-sizes distribution remains approximately unchanged. For $z\ll 1,$ we compute
\[L_j(z)\sim C_\infty \f{(1-\sqrt{z})^{j-1}}{z^{\f{3}{2}}} \sim C_\infty \f{e^{-j\sqrt{z}}}{z^{\f{3}{2}}}.
\]
 Thus, computing the inverse Laplace transform for $j\gg 1$ -- more precisely $j=\sqrt{\tau} x,$ since the typical timescale for the discrete diffusion equation is $O(\sqrt{\tau})$ -- and differentiating it in time, we obtain
\[\f{\partial}{\partial \tau}\ti\eta_j^{(2)}\approx \f{C_\infty}{2i\pi} \int_{\cal C} e^{z\tau} \f{e^{-j\sqrt z}}{\sqrt{z}} dz =\f{C_\infty}{2i\pi \sqrt{\tau}} \int_{\cal C} e^{\eta} \f{e^{-x\sqrt \eta}}{\sqrt{\eta}} d\eta,
\]
where we have used the change of variables $\eta=z\tau.$ We can see by integrating by parts that this expression decreases faster than any polynomial when $x\to \infty,$ which ensures that the contribution of $\ti\eta_j^{(2)}$ vanishes for $j\to\infty.$ Moreover, since the time derivative of $\ti\eta_j^{(2)}$ is in the order of $\f{1}{\sqrt{\tau}},$
 the {time needed for the 
 {clusters distribution}
 departing from $j\approx 1$ to come to $j\approx \f{1}{\ep}$ is of order $\f{1}{\ep^2},$ which appears much larger than the total time of this period in the $\tau$ variable, which is in the order of $\f{\log(\ep^{-1})}{\ep},$ see below.}}

\paragraph{Computation of the change of energy.}
{Writing $\ti \eta_1=\ti\eta_1^{(1)}+\ti\eta_1^{(2)},$ we can now gather the two contributions for the change of energy {(cf. \eqref{eq:EnExp} and \eqref{eq:EnExp2})}, writing
\begin{align*}
\Delta_{\hat E}&=\Delta_{\hat E}^{(1)}+\Delta_{\hat E}^{(2)},
\\
&\approx -\ep a \hat E -\ep \sqrt{2\hat E}  \f{C_\infty^2}{2\pi}\f{C}{i\pi} \int_{{\tau}}^{{\tau}+2\pi}  \sqrt{\tau} { \cos(\tau)} d\tau.
\end{align*}

As already seen above, we have $a>0$ so that the first term ensures an exponential decay for $\hat E.$ However, the sign of the second term changes according to the exact time-point $\tau\in[\tau_n, \tau_n+2\pi]$, 
and since it is weighted by $\sqrt{\hat E}$ it may after a while reveal dominant. However, at the beginning of this stage, we have $\hat E\gg 1$ so that the second term is negligible compared to the first one. We can thus assume $\tau > \tau_n \gg 1$ without loss of generality. We now estimate the integral in $\Delta_{\hat E}^{(2)}$ for large $\tau.$ 

\begin{align*}\int_{\tau}^{\tau+2\pi}  \sqrt{\tau} {\cos(\tau)} d\tau&=(\sqrt{\tau+2\pi}-\sqrt{\tau}){ \sin (\tau)} { -} \int_{\tau}^{\tau+2\pi}\f{1}{2}\f{{ \sin(s)}}{\sqrt{s}}ds,
\\
&= \sqrt{\tau}\left(\sqrt{1+\f{2\pi}{\tau}}-1\right){ \sin (\tau)}+O\left(\f{1}{\sqrt {\tau}}\right),
\\
&=O(\f{1}{\sqrt{\tau}}).
\end{align*}
We can now argue as in~\eqref{eq:ELDef} for Phase~II, and define $\hat E_n=e(s)$ with $s=\ep n,$ to get 
\[ \f{de}{ds}=-a e+\sqrt{e} F(s),
\]
with $a>0$ and $F(s)=O(\sqrt{\f{\ep}{s}}).$ We solve this equation by writing $y=\sqrt{e}$ which satisfies
\[\f{dy}{ds}=\f{1}{2\sqrt{e}}\f{de}{ds}=-\f{a}{2}y (s)+ \f{1}{2}F(s),
\]
so that 
\[y(s)=y(0)e^{-\f{a}{2}s} + e^{-\f{a}{2}s}\int_0^s F(t)e^{\f{a}{2}t}dt=_{s\to\infty} O\left(e^{-\f{a}{2}s} +\sqrt{\f{\ep}{s}}\right).
\]
Solving this equation we obtain that $e(s)$ is in the order of $e^{-as}+O(\f{\ep}{s}).$ Since the validity of the approximation of this phase ends  when $\ti\alpha,\,\ti\beta \ll 1,$ {\it i.e.} when $e(s)\ll 1,$ we notice that the second term does not really play a role, and the phase ends for $s=O(-\log(\ep))$, {\it i.e.} for a number of cycles in the order of $\f{\log(\ep^{-1})}{\ep}.$  We notice however that after this, the energy decays much more slowly, due to the second term: at the end of this phase, the energy $\hat E$ is still in the order of $\f{\ep}{\log{\ep^{-1}}}$, so that $\ti\alpha,\;\ti\beta = O(\ep^{1/2})$. It is however sufficient to enter the next and final phase, where oscillations have become negligible.}
\end{proof}

{\subsection{{Phase IV:} stabilization by means of a parabolic equation\label{subsec:PhaseIV:stab}}}

During this phase, the concentrations $C_{j}$ stabilize to their equilibrium
values. The oscillations of the concentrations $W,\ V$ cease, $V$ stabilizes to its equilibrium, and finally there is a feedback loop from the concentrations $C_j$ to determine the value of the concentration $W.$ 
{Let us now sum-up the main results of Phase~IV in the following formal proposition.
\begin{proposition}\label{prop:PhaseIV}
Departing at time $T_{4}^{in}$ from the initial conditions described in Lemma.~\ref{lem:endPhaseIII},  the behaviour of $C_j$ is approximated by the following free-boundary problem in the variable $\bar \tau:=\ep^3(t-T_4^{in}):$
\begin{eqnarray}
\frac{\partial C\left( x,\bar{\tau}\right) }{\partial\bar{\tau}}=C\left(
0^{+},\bar{\tau}\right) \frac{\partial C\left( x,\bar{\tau}\right) }{%
\partial x}+\frac{\partial^{2}C\left( x,\bar{\tau}\right) }{%
\partial x^{2}}\  ,\ \ x>0\  ,\ \ \bar{\tau}>0\ ,  \label{eq:CxEvol}
\\
\frac{\partial C\left( 0^{+},\bar{\tau}\right) }{\partial x}%
+\left( C\left( 0^{+},\bar{\tau}\right) \right) ^{2}=0\  ,\ \ \ \bar {\tau}%
>0\ ,  \label{eq:CxBV}
\\
C\left( x,\bar{\tau}=0\right) =\psi\left( x\right) . \label{eq:CIn}
\end{eqnarray}
Moreover, we have $V\sim 1$ and $W\sim 1-{\ep} C\left( 0^{+},\bar{\tau}\right)$ during all of Phase IV. We thus have, for $\bar{\tau}\to\infty$ 
and a time $t$ in the order of $\f{1}{\ep^3}$ in the physical time variable: $\lim\limits_{\bar{\tau}\to\infty} W-1= -\ep,$ $\lim\limits_{\bar{\tau}\to\infty} C_j \sim \exp (-\ep j).$ 
\end{proposition}}
\begin{proof}
At the end of Phase~III, we have seen that $\ti\alpha$ and $\ti\beta$ became negligible compared to $C_\infty$ and to $\ti\eta_j.$ We can approximate system~\eqref{Da3}--\eqref{Da4} by
\begin{align}
\f{d\ti\eta_j}{d\tau}&=\ti\eta_{j-1}-2\ti\eta_{j}+\ti\eta_{j+1},\quad j\geq2  \label{DaIV1} \\
{\frac{d\tilde \eta_{1}}{d\tau} }& ={\tilde\eta_{2}-\tilde \eta_1 +C_\infty,}  \label{DaIV2}
\end{align}
or equivalently, by recalling $C_j=C_\infty(1+\ep\tilde \eta_j)$ and taking the equations for $C_j$ for $j$ of order one, we
can approximate \eqref{Da3}--\eqref{Da4} as%
\begin{align}
\frac{dC_{j}}{d\tau} & =\left( C_{j-1}-2C_{j}+C_{j+1}\right)  ,\ \ j\geq 2
\label{Db3} \\
\frac{dC_{1}}{d\tau} & =C_{2}-C_{1},  \label{Db4}
\end{align}
where, at the beginning of Phase~IV,  $C_{j}$ is approximately constant equal to $C_\infty=\f{\pi}{2}$ for $j$ large.
We notice that Equations~\eqref{Db3}--\eqref{Db4}
yield an evolution for $C_{j}$ independent of $\tilde{\alpha},\ \tilde{\beta}%
.$ Equation (\ref{Db3}) is a discrete diffusion equation, so that solutions of \eqref{Db3}--\eqref{Db4} converge, in times $\tau$ of order one,
to $C_{j}=C_{\infty}$.

To understand what happens for larger times, we now examine the evolution of the concentrations $C_{j}$ with $j$ of order 
$\frac{1}{\varepsilon}.$ Let us recall~\eqref{eq:C:transportdiff}, where we keep only the leading order term for the diffusion and the $\ep$ order for the transport:
\[\f{dC_j}{d\tau}=-\f{\ep C_\infty}{2} (C_{j-1}-C_{j+1}) + (C_{j-1}-2C_j+C_{j+1}).
\]

 {As done for the previous phases (we recall that since the end of Phase~II we have $L_n\sim\f{1}{\ep}$), w}e use the approximation of $C_{j}\left(
\tau\right) $ as $C\left( x,\tau\right) $ with $x=\varepsilon j.$
Then%
\begin{equation}
\frac{\partial C\left( x,\tau\right) }{\partial\tau}=-{\varepsilon ^{2}%
}C_\infty \frac{\partial C\left(
x,\tau\right) }{\partial x}+\ep^2\frac{%
\partial^{2}C\left( x,\tau\right) }{\partial x^{2}}.  \label{eq:CxPh4}
\end{equation}

This equation must be solved with the initial condition~\eqref{eq:CIn} obtained in Lemma~\eqref{lem:endPhaseIII} at the
beginning of Phase IV.
{The value of $C_{\infty}$ can be
approximated identifying it with the value obtained for the outer
concentrations $C\left( x,\tau\right) $ as $x\rightarrow0^{+}.$ }
Introducing
also the time scale $\bar{\tau}=\varepsilon^{2}\left(
\tau-\tau_{in}\right) $ where $\tau_{in}$ $=\ep T_{4}^{in}$ is the time when we assume the
Phase IV to begin, we obtain~\eqref{eq:CxEvol}.

In order to determine the boundary condition to be imposed at $x=0$ for the
solutions to \eqref{eq:CxEvol}, we use the second condition in \eqref{eq:Mresc}, that is
\begin{equation}
1=\sum_{k=1}^{\infty}\varepsilon C_{k}=\sum_{k=1}^{\infty}\varepsilon
C\left(\ep k,\bar{\tau}\right) \simeq%
\int_{0}^{\infty }C\left( x,\bar{\tau}\right) dx \ , \label{eq:CNorm}
\end{equation}
which implies that $\int_{0}^{\infty}C\left( x,\bar{\tau}\right) dx$ is
constant for all $\bar{\tau}$ and integration of \eqref{eq:CxEvol} yields \eqref{eq:CxBV}.

On the other hand, the first condition in \eqref{eq:Mresc} implies an additional normalization condition, namely 
\begin{align*}
1&=\varepsilon\left( V+W\right)
+\varepsilon^{2}\sum_{j=1}^{\infty}j C_{j}=\varepsilon\left( V+W\right)
+\sum_{j=1}^{\infty}\varepsilon j%
C\left( \varepsilon j,\bar{\tau}\right) \ep \\
&\approx\varepsilon\left( V+W\right) +\int_{0}^{\infty}xC\left( x,\bar{\tau}\right) dx \ .
\end{align*}

Using that $V$ and $W$ are close to $1$ we then obtain the following
normalization condition as $\varepsilon\rightarrow0$%
\begin{equation}
\int_{0}^{\infty}xC\left( x,\bar{\tau}\right) dx=1. \label{eq:CNorm1}
\end{equation}

The problem \eqref{eq:CxEvol}--\eqref{eq:CIn} yields the evolution of the
cluster concentrations during Phase IV. The steady states of the system \eqref%
{eq:CxEvol}--\eqref{eq:CxBV} have the form%
\begin{equation*}
C_{s}(x) =K\exp\left( -Kx\right),\ \quad x>0 \ ,
\end{equation*}
where $K>0$ is an arbitrary constant. Notice that $C_{s}(x)$
satisfies the normalization condition (\ref{eq:CNorm}) for any $K>0$.
In order to determine $K$ we use the normalization condition \eqref{eq:CNorm1}%
. We calculate%
\begin{equation*}
\int_{0}^{\infty}x\, C_{s}(x) dx=%
\frac{1}{K} \ .
\end{equation*}
Thus, (\ref{eq:CNorm}) implies $K=1.$ Therefore the equilibrium distribution
of clusters is given by%
\begin{equation}
C_{s}\left( x\right) =\exp\left( -x\right) ,\ \ x>0\ .
\label{eq:Ceq}
\end{equation}
\end{proof}
\section{Discussion}
\label{sec:discussion}
\subsection{Other choices of initial concentrations.}
\label{subsec:discussion:other}

{
It is worth to note that choosing initial cluster distributions different
from the ones considered in the previous subsections, i.e. the energy at initial time $E_{0}$ of order $%
1$ and the characteristic length at initial time $L_{0}$ of order $1$ (or more generally smaller than $\frac {1}{%
\varepsilon}$), the dynamics of the system differ from those described above in the four consecutive phases.

\paragraph{Initial condition widely spread along the size axis.} Certainly, we can skip some of the phases starting
with initial cluster contributions having, say $E_{0}\approx1$ and $%
L_{0}\approx\frac{1}{\varepsilon}.$ This would result in an evolution
without Phase I, starting directly in Phase II. If, in addition to $%
L_{0}\approx \frac{1}{\varepsilon}$ we have also $E_{0}\approx%
\varepsilon^{2},$ we could have evolutions starting directly in Phase III,
skipping the two previous phases.\\

\paragraph{Initial condition concentrated on the small sizes.} If we take as starting point values of $%
\left( w^0,v^0\right) $ for which the concentrations $c_{k}$ tend to move
towards smaller values of $k,$ we can obtain very large changes of the
energy just in the early transient states. For instance, if we begin with $E_0$
close to $1$ with $w^0=v^0<\varepsilon,$ and concentrations $c_{k}^0$ concentrated
in values of {$k$ of order one}, we obtain a concentrations wave
moving towards lower values of $k.$ This results in a large increase of $%
c_{1}$ and as a consequence large changes of the Energy $E$. Due to this the initial,
transient dynamics of the whole system can result in values of $\left(
w,v\right) $ that differ significantly from a LV orbit. After the values of $%
\left( w,v\right) $ reach the line $w=v>\varepsilon$ we obtain a dynamics
that can be described as indicated above, for suitable values of $E.$

\paragraph{Initial condition concentrated in a dirac mass along the size axis.} Moreover, we can make choices of the
initial cluster concentrations which differ more drastically of the previous
phases, because they cannot be characterized in a meaningful manner by a
single characteristic length $L_{0}.$ This would be the case, for instance
with distributions with the form%
\begin{equation*}
c_{j,0}=\varepsilon\delta_{j,R_{\varepsilon}}\ \ \text{where }R_{\varepsilon
}=\frac{a}{\varepsilon}\  ,\ \ \text{for some }a\in\left( 0,1\right) .
\end{equation*}

We will assume also that $v$ and $w$ are of order one. In this case we have
approximately $E_{0}\simeq1-a>0$ for small $\varepsilon.$ We will assume
then that initially $v=w\simeq\frac{1-a}{2}$ by definiteness. On the other
hand, it is not clear what should be the definition of $L_{0}.$ Taking
into account the set of values where the mass of the clusters is
concentrated we should take $L_{0}=R_{\varepsilon}=\frac{a}{\varepsilon}.$
On the other hand, there is not dispersion in the concentration
distributions and therefore it would be also reasonable to assume that $%
L_{0}=1.$ Actually, the evolution of the cluster concentrations in this
case differs from the one described in the previous sections.
For these initial concentrations, we obtain oscillations of the
concentrations $c_{j}$ in the space of cluster sizes $j$ in a manner similar
to the one described in subsections~\ref{subsec:sizedistrib},~\ref{subsec:following}, while
at the same time diffusion in the space of clusters takes place (cf. (\ref%
{eq:Coninit})). During the first oscillations, the values of $c_{1}$ are
exponentially small in $\frac{1}{\varepsilon}$ and, due to this, the energy $%
E_{n}$ which characterizes each of the LV cycles remains almost constant.
The diffusion in the space of cluster sizes increases slowly the width of
the cluster distributions. The values of $c_{1}$ become significant (i.e.
non-exponentially small), after $\frac{1}{\varepsilon}$ LV cycles.
Actually, after this number of cycles, the distribution of clusters has a
characteristic length of order $\frac{1}{\varepsilon}$ and the corresponding
evolution becomes similar to the one described in Phase II, with the only
difference that the cluster concentrations is not necessarily given by the
Gaussian distribution $\psi\left( x\right) $ in (\ref{eq:PsiAs}). We then
obtain an evolution similar to the one described in subsection~\ref%
{subsec:following}, but where an additional evolution of the concentrations $%
\varphi_{n}\left( x\right) $ by means of the iteration (\ref{eq:SemGroup})
must be included (cf. Figure \ref{fig:evo_dirac}).  During
this modified Phase II, the energy $E_{n}$ decreases until reaching values
of order $\varepsilon,$ when the system starts to evolve according to the
mechanism described in Phase III, and eventually, the concentrations
approach to the equilibrium as described in Phase IV.

\begin{figure}
     \centering
     \begin{subfigure}[b]{0.495\textwidth}
         \centering
         \includegraphics[width=\textwidth]{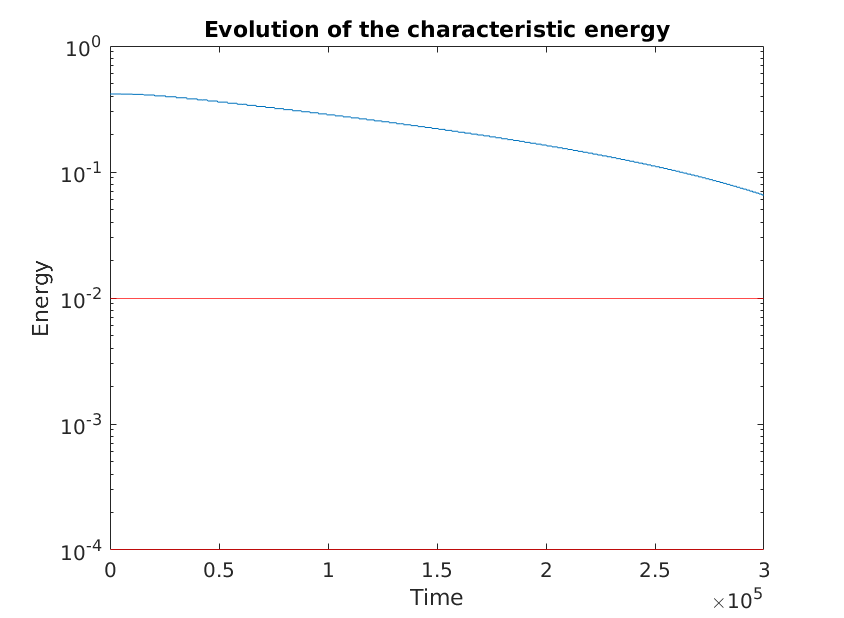}
         \caption{Energy (log scale)}
         \label{fig:ener_dirac}
     \end{subfigure}
     \hfill
     \begin{subfigure}[b]{0.495\textwidth}
         \centering
         \includegraphics[width=\textwidth]{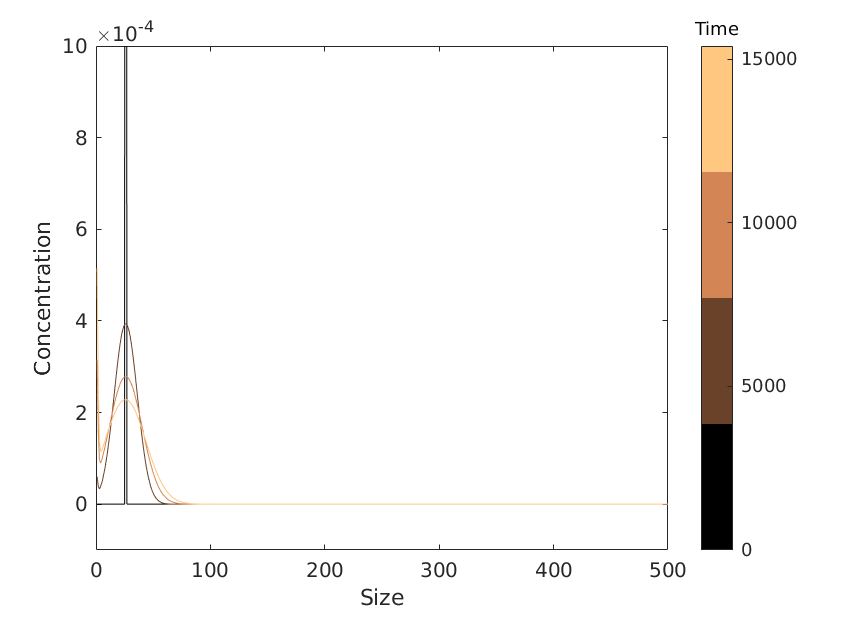}
         \caption{Size distribution}
         \label{fig:sd_dirac}
     \end{subfigure}
        \caption{{Numerical simulation of the evolution of the cluster size distribution and the energy taking as initial condition a Dirac mass located at the size $R_{\varepsilon
}=25$ and $ v=w \approx \tfrac{1-a}{2}$. The time evolution is indicated by the colors in Figure \ref{fig:sd_dirac}. The Figure \ref{fig:sd_dirac} is truncated on the y--axis at $0.001$. In the Figure \ref{fig:ener_dirac}, the horizontal red lines illustrate the thresholds when the energy would be inferior to $\varepsilon$.} }
        \label{fig:evo_dirac}
\end{figure}

\subsection{Concluding remarks}
\label{subsec:discussion:concluding}
In this paper, using asymptotic and numerical methods, we have described the
behaviour of the solutions of the system of equations \eqref{eq:BD2}--\eqref{eq:BD3}. This system of equations couples the dynamics of the classical
Lotka-Volterra oscillator with a generalized Becker-D\"oring model that
describes the concentrations of clusters with each given length. The coupled
system was introduced in \cite{DFMR}. This paper describes the way in which
the chemical concentrations $v,\ w,\ \left\{ c_{k}\right\} _{k\in\mathbb{N}} 
$ approach to their equilibrium values if the parameter $\varepsilon$ is
small. This equilibrium is reached by means of a mechanism in which the
concentrations of $v,\ w$ oscillate, with decreasing amplitude around a
center $\left( \varepsilon,\varepsilon\right) ,$ with the monomers spreading
in the set of chemical concentrations $\left\{ c_{k}\right\} _{k\in\mathbb{N}%
}$ in increasingly larger values of cluster sizes $k,$ until reaching sizes
of order $k\approx\frac{1}{\varepsilon}.$ In this paper we have derived a
set of equations that describes the decrease of the amplitude of the
oscillations in the space $\left( v,w\right) $ as well as a sequence of
iterative mappings that describes the evolution of the chemical
concentrations $\left\{ c_{k}\right\} _{k\in\mathbb{N}}.$ A characteristic
feature of the mechanism that yields the damping of oscillations that we
have obtained is the onset of some oscillations of the whole
set of chemical concentrations $\left\{ c_{k}\right\} _{k\in\mathbb{N}}$ in
the space of cluster sizes $k.$

{In the course of the analysis, we have raised a series of interesting asymptotic problems, e.g. the trend of the iterative system~\eqref{BL1}--\eqref{BL2}--\eqref{BL3} towards its steady state~\eqref{BLSt1}--\eqref{BLSt3}, or the study of the nonlinear free-boundary problem~\eqref{eq:CxEvol}--\eqref{eq:CIn}, or yet a rigorous justification for the heuristic description of the "collision" of the cluster concentration waves with regions where $k\approx 1$, in the fast-moving regions $t\in [T_n-\Delta t,T_n+\Delta t]$, see Lemma~\ref{lem:I:sizedistrib}.  Their proof is let for future work~\cite{DFMV2}.}

In this paper we have assumed that the reaction coefficients are constant
and we have modified the time scale to normalize them as one. It would be
relevant to understand if the damping mechanism for the chemical
oscillations obtained in this paper is also valid for more general choices
of the chemical coefficients. 
Constant coefficients are well-adapted for polymerisation by one or the two ends of fibrils, for instance, whereas linear or affine coefficients could take into account secondary nucleation, {\it i.e.} lateral polymerisation, and other more complex laws could take into account varying geometries of the clusters~\cite{prigent:hal-00778052,Radford}. In the case
considered in this paper, the model can be rewritten in an obvious manner
as a perturbation of the Lotka-Volterra model. However, in the case of more
general coefficients, it is not clear  if the model can be
rewritten as a perturbation of Lotka-Volterra or some more involved
oscillator.
\bibliographystyle{plain}
\bibliography{Biblio_Bimono.bib}

\appendix

\bigskip

\section{Extra computations for the LV system}
\label{sec:extra}

\label{app:LV} In this appendix, we prove the asymptotic expansions of Lemmas~\ref{lem:LV1} to~\ref{lem:LV4}. We 
discuss also various orders of magnitude, which are useful to get some
insight in the oscillatory behaviour of the chemical concentrations $v,\ w.$

\paragraph{Proof of Lemma~\ref{lem:LV1}: from $0$ to $t_{1}$.}\hfill \\
We develop
asymptotically \eqref{D8} as%
\begin{equation*}
\frac{\diff v}{\diff t}=-vw + \varepsilon v =
-v\left(1-v+O(\varepsilon\log(\varepsilon))\right).
\end{equation*}
Hence, 
\[
\frac{1}{v(1-v)}\frac{\diff v}{\diff t}=\frac{\diff }{\diff t}\left(\log(\f{v}{1-v})\right)=\f{\diff}{\diff t} \log(-1+\f{1}{1-v})=-1+O(\varepsilon\log(\varepsilon))
\]
and we deduce 
\[
\log (-1+\frac{1}{1-v})=-t
+O(\varepsilon\log(\varepsilon) t),\quad -1+\frac{1}{1-v}=e^{-t (1+O(\varepsilon\log(\varepsilon) ))},
\]
so finally
\begin{equation*}
v(t)=1-\frac{1}{1+e^{-t}} +O(e^{-t} \varepsilon \log(\varepsilon) (t+1)). 
\end{equation*}
We then compute that $v(t_1)=\varepsilon$ if $e^{-t_1}=\varepsilon
+O(e^{-t_1 \varepsilon \log(\varepsilon) (t_1+1)}),$ which implies 
\begin{equation*}
t_1\sim-\log(\varepsilon). 
\end{equation*}

The equation for $w$ may now be expanded: 
\begin{equation*}
\frac{\diff w}{\diff t}=vw -\varepsilon w=w\left(1-\frac{1}{1+e^{-t}} +O( \varepsilon \log(\varepsilon) )\right), 
\end{equation*}
hence
\[
\frac{\diff }{\diff t} \log (w)= 1-\frac{e^t}{e^t+1} +O(\varepsilon \log
(\varepsilon)) =\frac{\diff}{\diff t} (t- \log(e^t+1)) +O(\varepsilon \log
(\varepsilon)),
\]
so by integrating we get 
\[
\log w(t)=-\log 2 + t - \log(e^t+1) + \log(2)+ O(\varepsilon
\log(\varepsilon) t)
\]
and finally
\begin{equation*}
w(t)=\frac{e^t}{1+e^t} (1+O(\varepsilon \log(\varepsilon) t)), 
\end{equation*}
which provides us with $w(t_1)=1 + O(\varepsilon (\log\varepsilon)^2),$ and
gives~\eqref{estim:t1}. More precisely, we compute from the energy 
\begin{equation*}
1=w(t_1)-\varepsilon -\varepsilon \log(\frac{w(t_1)}{\varepsilon}) \implies
w(t_1)=1 -\varepsilon\log(\varepsilon) + \varepsilon+o(\varepsilon). 
\end{equation*}

\paragraph{Proof of Lemma~\ref{lem:LV2}: From $t_{1}$ to $t_{2}$.}

During this second phase, $v$ goes from $\ep$ to exponentially small values, whereas $w$ goes from approximately $1$ to $\ep.$ We thus
define an intermediate time $t_{1,2}$ by $v(t_{1,2}):=\varepsilon^{3/2},$ and  the energy equality becomes
\begin{equation*}
1=w(t_{1,2}) -\varepsilon \log(\frac{w}{\sqrt{\varepsilon}}) -2\varepsilon
+O(\varepsilon^{3/2}) \implies w(t_{1,2})=1-\frac{\varepsilon}{2}
\log(\varepsilon)+2\varepsilon +O(\varepsilon^{3/2}), 
\end{equation*}
from which we deduce that on $(t_1,t_{1,2}),$ 
\begin{equation*}
\frac{\diff v}{\diff t}=-v(w-\varepsilon)=-v(1+O(\varepsilon\log(\varepsilon)))%
\implies v(t)=\varepsilon
e^{-(t-t_1)+O(\varepsilon\log(\varepsilon)(t-t_1))}, 
\end{equation*}
hence we have 
\begin{equation*}
\varepsilon^{3/2} = \varepsilon
e^{-(t_{1,2}-t_1)+O(\varepsilon\log(\varepsilon)(t_{1,2}-t_1))} \implies
t_{1,2}-t_1\sim -\frac{1}{2}\log(\varepsilon)
\end{equation*}
and on $(t_{1,2},t_2)$ we have $v=O(e^{3/2})\ll \varepsilon$ hence 
\begin{equation*}
\frac{d w}{dt}=-w (\varepsilon + O(\varepsilon^{3/2})) \implies w(t)=\left(1-%
\frac{\varepsilon}{2} \log(\varepsilon)+2\varepsilon
+O(\varepsilon^{3/2})\right)e^{-\varepsilon(t-t_{1,2})}, 
\end{equation*}
we then use $w(t_2)=\varepsilon$ to obtain 
\begin{multline}\label{eqsim:t2}
\varepsilon= \left(1-\frac{\varepsilon}{2} \log(\varepsilon)+2\varepsilon
+O(\varepsilon^{3/2})\right)e^{-\varepsilon(t_2-t_{1,2})} \\
\implies t_2-t_{1,2} \sim-\frac{1}{\varepsilon}\log(\ep)\ ,
\end{multline}
so that 
\begin{equation*}  
t_2-t_1\sim -\frac{1}{\varepsilon}\log(\ep). 
\end{equation*}
 If we need an intermediate time $t_2^*$ such that $t_2^*-t_{1,2}\sim \f{1}{\sqrt{\ep}}$, useful for the typical symetric time interval during which mass accumulates in $c_1,$ we see that $w(t_2^*)\sim e^{-\sqrt{\ep}}\sim 1,$ justifying the approximations carried out for $\Delta t$ in Lemma~\ref{lem:I:sizedistrib}.

We can also compute explicitely 
\begin{equation*}
\frac{\diff v}{\diff t}=-v(w-\varepsilon)=-v\left(\left(1-\frac{\varepsilon}{2}
\log(\varepsilon)+2\varepsilon
+O(\varepsilon^{3/2})\right)e^{-\varepsilon(t-t_{1,2})} -\varepsilon\right), 
\end{equation*}
which implies 
\begin{equation*}
v(t)=\varepsilon^{3/2}\exp\left(-\left(\frac{1}{\varepsilon}-\frac{1}{2}
\log(\varepsilon)+2+O(\varepsilon^{1/2})\right)(1-e^{-%
\varepsilon(t-t_{1,2})}) +\varepsilon(t-t_{1,2})\right). 
\end{equation*}
Taken in $t=t_2$ and using~\eqref{eqsim:t2} we get 
\begin{multline*}
v(t_2)=
\varepsilon^{3/2}\exp\left(-\frac{1}{\varepsilon}+\frac{1}{2}
\log(\varepsilon)-2+O(\varepsilon^{1/2}) +\f{1}{\ep}\ep -  \log(\varepsilon)\right) \\
=\varepsilon^{3/2}\exp\left(-\frac{1}{\varepsilon}-\frac{1}{2}
\log(\varepsilon)-1+O(\varepsilon^{1/2}) 
\right) 
\sim \ep e^{-1}e^{-\f{1}{\ep}},
\end{multline*}
formula which we can also obtain from
the conservation of energy:
\begin{align*}
E & =v+w-2\varepsilon-\varepsilon\log\left( \frac{vw}{\varepsilon^{2}}\right)
\\
1 & =v+\varepsilon-2\varepsilon-\varepsilon\log\left( \frac{vw}{%
\varepsilon^{2}}\right) \\
1 & =-\varepsilon-\varepsilon\log\left( \frac{v\varepsilon}{\varepsilon^{2}}%
\right) +O(\ep^{3/2})\\
v & \sim \varepsilon e^{-1}\exp\left( -\frac{1}{\varepsilon}\right)
\end{align*}

\textbf{Proof of Lemma~\ref{lem:LV3}: From $t_{2}$ to $t_{3}$.}
We divide the trajectory connecting $\left( v\left(
t_{2}\right) ,w\left( t_{2}\right) \right) $ with $\left( v\left(
t_{3}\right) ,w\left( t_{3}\right) \right) $ into two (symmetric) parts, by defining $t_{2,3}$ by
\[v(t_{2,3})=w(t_{2,3} ).\]
For $t\in [ t_{2},t_{2,3}] $ we have $v\left(
t\right) \leq w\left( t\right) ,$ and for $t\in\left[
t_{2,3},t_{3}\right] $ we have $v\left( t\right) \geq w\left( t\right) .$
During the first interval we have that $v$ is much smaller than $\varepsilon$
and we approximate (\ref{D8}) as
\begin{equation*}
\frac{dv}{dt}=-vw+v\varepsilon\ \ ,\ \ \frac{dw}{dt}=-\varepsilon w \ .
\end{equation*}

We then obtain the approximations%
\begin{align*}
w & \sim\varepsilon\exp\left( -\varepsilon\left( t-t_{2}\right) \right) \\
v & \sim v\left( t_{2}\right) \exp\left( \varepsilon\left( t-t_{2}\right) - 
\left[ 1-\exp\left( -\varepsilon\left( t-t_{2}\right) \right) \right] \right)
\end{align*}
for $t\in\left[ t_{2},t_{2,3}\right] .$ We can determine $t_{2,3}$ by means
of the condition $v\left( t_{2,3}\right) =w\left( t_{2,3}\right) $ %
applied to the energy conservation: since $v=w=O(e^{-\f{1}{\ep}})$,
we compute 
\begin{equation*}
1=2v -2\ep -\ep\log(\f{v^2}{\ep^2})=-2\ep (1+\log(\f{v}{\ep}))+ O(e^{-\f{1}{\ep}}). \end{equation*}
Hence $v(t_{2,3})=w(t_{2,3})\sim \ep e^{-1} e^{-\frac{1}{2\varepsilon}},$
and using $w(t_{2,3})\sim \ep e^{-\ep(t_{2,3}-t_2)}$ we get
\begin{equation*}
t_{2,3}-t_2 \sim \frac{1}{2\varepsilon^2}. 
\end{equation*}

We can now compute $\left( t_{3}-t_{2,3}\right) $ using a same argument, due
to the symmetry of the equation. Notice that the energy formula \eqref{D5}
as well as $v\left( t_{3}\right) =\varepsilon$ imply $%
w\left( t_{3}\right) \sim\varepsilon e^{-1}\exp\left( -\frac{1}{\varepsilon}%
\right) .$ In the interval $t\in\left[ t_{2,3},t_{3}\right] $ we can
approximate \eqref{D8} as%
\begin{equation*}
\frac{dv}{dt}=\varepsilon v\  ,\ \ \frac{dw}{dt}=vw-\varepsilon w \ ,
\end{equation*}
because $w\ll\varepsilon$ for $t\in\left[ t_{2,3},t_{3}%
\right] .$ We then obtain the following approximation $\left(
t_{3}-t_{2,3}\right) \sim\frac {1}{2\varepsilon^{2}}$ for small $%
\varepsilon. $ Then%
\begin{equation*}
t_{3}\sim t_{2}+\frac{1}{\varepsilon^{2}}\ \ \text{as\ \ }\varepsilon
\rightarrow0. 
\end{equation*}

\textbf{Proof of Lemma~\ref{lem:LV4}: From $t_{3}$ to $t_{4}$.}

For $t>t_{3}$ both $v$ and $w$ increase. This stage can be analyzed in a
manner symmetric to the one from $t_{1}$ to $t_{2}.$ During this phase $v$
becomes of order one. This happens in times of order $\frac{1}{\varepsilon}$
(up to a logarithmic correction)$.$ We can define $t_{4}$ by means of $%
w\left( t_{4}\right) =\varepsilon.$ During this phase we can do the same computations as in $(t_1, t_2)$ so that we obtain:%
\begin{equation*}
t_{4}-t_{3}\sim\frac{1}{\varepsilon}\log\left( \frac{1}{\varepsilon }%
\right). 
\end{equation*}

\textbf{Proof of Lemma~\ref{lem:LV5}: From $t_{4}$ to $t_{5}$.}

Finally, the trajectory returns to the line $\left\{ w=v\right\} .$ 
Using the conservation of energy formula (\ref{D5}) we can approximate the
evolution of $\left( w,v\right) $ by means of the equation%
\begin{equation*}
\frac{dw}{dt}=w\left( 1-w\right)  ,\ \ v=1-w 
\end{equation*}

We have defined define $t_{5}$ by means of $w\left( t_{5}\right) =v\left(
t_{5}\right) >\varepsilon.$ Using that $w\left( t_{4}\right) =\varepsilon$
we obtain the following approximation, symmetric to the interval $(0,t_1)$: %
\begin{equation*}
t_{5}-t_4\sim \log\left( \frac{1}{\varepsilon}\right). 
\end{equation*}

\textbf{Proof of Lemma~\ref{lem:scaling2}}
Using (\ref{eq:Scal}) and (\ref{eq:TScal}) we obtain%
{
\begin{equation*}
\begin{aligned}
D\left( E,\varepsilon\right) &=\frac{E}{2}\int_{0}^{T\left( E,\varepsilon
\right) }\left( w+v\right) \left( Es;1,\frac{\varepsilon}{E}\right) ds
=\frac{%
1}{2}\int_{0}^{ET\left( E,\varepsilon\right) }\left( w+v\right) \left( s;1,%
\frac{\varepsilon}{E}\right) ds\\
&=\frac{1}{2}\int_{0}^{T\left( 1,\frac{%
\varepsilon}{E}\right) }\left( w+v\right) \left( s;1,\frac {\varepsilon}{E}%
\right) ds=D\left( 1,\frac{\varepsilon}{E}\right) 
\end{aligned}
\end{equation*}

\textbf{Proof of Lemma~\ref{lem:DYT}}

The equivalent $T(1,\ep)\sim \f{1}{\ep^2}$ comes directly from Lemmas~\ref{lem:LV1} to~\ref{lem:LV5}. 
It is known that for the Lotka-Volterra system, the mean value of $v$ and $w$ satisfies the equality
\[\f{1}{\ep}\int_0^T v(s)\diff s=\f{1}{\ep}\int_0^T w(s)\diff s= \ep,\]
hence their equivalent and the one for $D(1,\ep).$ This proves~\eqref{eq:Tasymp}.
The maximum value of $Y$ is given by
\[Y_{\text{max}} (t)=\int_0^{\f{T}{2}} (w(s)-v(s))\diff s = \int_0^{t_{2,3}}(w(s)-v(s))\diff s
\]
and with Lemmas~\ref{lem:LV1} to~\ref{lem:LV5} we get
\[\begin{array}{lll}
Y(t_1)=\int_0^{t_1} (w(s)-v(s))\diff s &\sim& \int_0^{t_1} \diff s\sim -\log(\ep),
\\ 
\int_{t_1}^{t_{1,2}}  (w(s)-v(s))\diff s &\sim& t_{1,2}-t_1\sim-\f{1}{2} \log(\ep),
\\
\int_{t_{1,2}}^{t_2}  (w(s)-v(s))\diff s &\sim& \int_{t_{1,2}}^{t_2} e^{-\ep(s-t_{1,2})} \diff s \sim\f{1}{\ep},
\\
\int_{t_{2}}^{t_{2,3}}  (w(s)-v(s))\diff s &\leq &\int_{t_{2}}^{t_{2,3}} w(s)\diff s \sim \int_{t_{2}}^{t_{2,3}} \ep e^{-\ep(t-t_2)} \diff s \leq 1.
\end{array}
\]
Finally, we see that the maximal contribution to the displacement of $Y$ is linked to the part $(t_{1,2},t_2),$ and we have $Y_{\text{max}}\sim\f{1}{\ep}.$
}

\section{Linearised stability around the positive steady state \label%
{app:linear}}

In this appendix we study the linear stability of the steady states (\ref{eq:steady}%
).

\begin{proposition}[Linear stability around the positive steady state] \hfill\\
For $M=1$ and $\ep \ll 1,$ the unique positive steady state defined by~\eqref{eq:steady} is locally asymptotically stable.
\label{prop:linear}
\end{proposition}

\begin{proof}
We provide here the main lines for the linear stability analysis, letting a fully rigorous and complete study for future work. We introduce new variables by means of 
\begin{equation}  \label{def:phi}
c_{j}=\bar c_{j}(1+\varphi_{j}),\qquad j\geq1,\qquad v=\bar v (1+V),\qquad
w=\bar w(1+W),
\end{equation}
and recall that the positive steady state $(\bar v,\bar w,\bar c_{j})$ is
defined by~\eqref{eq:steady}, namely 
\begin{equation*}
\begin{array}{l}
\theta:=1-\frac{1}{2}\biggl({\frac{1}{\varepsilon} - \sqrt{\Bigl(\frac {1}{%
\varepsilon}-2\Bigr)^{2}+{4}}}\biggr)\sim{\varepsilon}{}, \\ 
\bar c_{1}:=\varepsilon\theta\sim{\varepsilon^{2}}{},\qquad\bar c_{i}:
= (1-\theta)^{i-1} \bar c_{1}, \qquad\bar v:=\varepsilon,\qquad\bar
w:=\varepsilon-\bar c_{1}=\varepsilon(1-\theta).%
\end{array}%
\end{equation*}
Assuming $\vert\varphi_{j}\vert+\vert V\vert+\vert W\vert\ll1,$ neglecting
higher order terms, we get the following linearised system:

\begin{align}
\frac{dV}{dt} & =-\varepsilon(1-\theta)W-\varepsilon\theta\varphi_{1}, \\
\frac{dW}{dt} & =\varepsilon V, \\
\frac{d\varphi_{1}}{dt} & =-\varepsilon(1-\theta) (W
+\varphi_{1}-V-\varphi_{2}), \\
\frac{d\varphi_{j}}{dt} & =\varepsilon(\varphi_{j-1} - \varphi_{j}+W-V)
-\varepsilon(1-\theta) (\varphi_{j}-\varphi_{j+1}+W-V).
\end{align}
Introducing the time variable $\tau=\varepsilon t$ we get 
\begin{align}
\frac{dV}{d\tau} & =-(1-\theta)W-\theta\varphi_{1}, \\
\frac{dW}{d\tau} & = V, \\
\frac{d\varphi_{1}}{d\tau} & =-(1-\theta) (W +\varphi_{1}-V-\varphi_{2}), \\
\frac{d\varphi_{j}}{d\tau} & = (\varphi_{j-1} - \varphi_{j}+W-V) -
(1-\theta) (\varphi_{j}-\varphi_{j+1}+W-V).
\end{align}
We look for solutions to the eigenproblem, hence of the form $%
(V(t),W(t),\varphi_{j}(t))=e^{\lambda\tau} (V,W,\varphi_{j})$ with $%
\lambda\in\mathbb{C}.$ We get the system: 
\begin{align}  \label{eq:V:eigen}
\lambda V & =-(1-\theta)W-\theta\varphi_{1}, \\
\lambda W & = V,  \label{eq:W:eigen} \\
\lambda\varphi_{1} & =-(1-\theta) (W +\varphi_{1}-V-\varphi_{2}),
\label{eq:phi1:eigen} \\
\lambda\varphi_{j} & = (\varphi_{j-1} - \varphi_{j}+W-V) - (1-\theta)
(\varphi_{j}-\varphi_{j+1}+W-V).  \label{eq:phij:eigen}
\end{align}
We recall that $\theta\sim{\varepsilon}$ and, thus, $\theta\ll1$ as soon as $%
\varepsilon\ll 1$, hence we 
treat the solutions to~\eqref{eq:V:eigen}--\eqref{eq:phij:eigen} in a
perturbative manner with respect to $\theta$.

\textbf{At the limit $\theta=0,$} we get 
\begin{align}  \label{eq:V:eigen0}
\lambda V & =-W, \\
\lambda W & = V ,  \label{eq:W:eigen0} \\
\lambda\varphi_{1} & =-W -\varphi_{1}+V+\varphi_{2},  \label{eq:phi1:eigen0}
\\
\lambda\varphi_{j} & =\varphi_{j-1} - 2 \varphi_{j} +\varphi_{j+1}.
\label{eq:phij:eigen0}
\end{align}
We notice that we have two eigenvalues $\lambda_{0}=\pm i,$ with $V=\pm i W$%
, linked to an oscillatory regime of period $2\pi$ (coherent with the period found at the end of Phase~III), and, in the case $V=W=0$, a continuous
spectrum similar to the one of the discrete heat equation, whose
admissible generalised eigenvalues  all have a negative real part except $\lambda=0$
(associated to the constant sequence $\varphi_{j}=\varphi_{1}$ for all $j$).

\ \textbf{Case (a):} $\lambda=\pm i,$
$V=\pm i W.$ We only detail the computations for $\lambda=i$ since the case $\lambda=-i$ follows from taking the conjugate. For $W_{0}=-i,$ $V_{0}=1$ we 
obtain 
\begin{align*}
\varphi_{2} & =(i+1) \varphi_{1} -(1+i) , \\
\varphi_{j+1} & =(i+2) \varphi_{j}-\varphi_{j-1},\qquad j\geq2.
\end{align*}
We recognize a linear recurrent sequence of order 2 for $j\geq2$, whose
characteristic equation is 
\[
r^{2}-(i+2)r+1=0 \implies r_{\pm}=\frac{i+2\pm\sqrt{(i+2)^{2}-4}}{2}=1+\frac{%
i}{2}\pm\sqrt{i-\frac{1}{4}}.
\]
Since $\vert r_{-}\vert<1$ and $\vert r_{+}\vert>1$, the only admissible value is $r_-,$ that we denote $r=r_{-}$ for simplicity. 
We finally have
\begin{equation}  \label{eq:eigen:theta0}
V_{0}=1,\quad W_{0}=-i,\quad\varphi_{j,0}=\frac {1+i}{1-r}r^{j},\quad r=1+%
\frac{i}{2}-\sqrt{i-\frac{1}{4}}.
\end{equation}

 \textbf{%
Case (b):} $W=V=0.$ Then ~\eqref{eq:phi1:eigen0} implies $%
(\lambda+1)\varphi_{1}=\varphi_{2},$ and~\eqref{eq:phij:eigen0} implies for $%
j\geq2$ 
\begin{align*}
\lambda\varphi_{j} & =\varphi_{j-1} -2\varphi_{j} +\varphi_{j+1},\qquad
j\geq2, \\
\lambda\varphi_{1} & =-(\varphi_{1}-\varphi_{2}).
\end{align*}
We recognize the discrete heat equation. 
As previously, the characteristic equation is 
\[
r^{2} -(\lambda+2) r+1=0\implies r_{\pm}=\frac{\lambda+2\pm\sqrt {%
(\lambda+2)^{2}-4}}{2}=\frac{\lambda+2\pm\sqrt{\lambda(\lambda+4)}}{2}
\]
We see that for $\lambda=0$ we have $r_\pm=1$ linked to  the constant sequence $\varphi_j=1.$

If $\lambda\neq 0,$  generalised eigenvectors are given by linear combinations of the two sequences $(r_\pm ^{j-1})_{j\geq 1}$. Using the boundary condition at $j=1$ and $r_++r_-=\lambda+2,$ we obtain
\[\varphi_j=r_+^{j-1} - \f{1-{r_-}}{1-r_+}r_-^{j-1},\] 
so that to ensure $(\varphi_j)\in \ell^\infty (\mathbb{C})$ we need $\vert r_{\pm} \vert \leq 1.$ Since $r_+r_-=1,$ we define $r_\pm = e^{\pm i\beta}$ with $\beta\in[0,2\pi),$ and compute 
\[\lambda=e^{i\beta}+e^{-i\beta}
- 2=2(\cos(\beta) -1) \in[-4,0].
\] To conclude:
We have found that no admissible (generalised)
eigenvalue  has a positive real part, and the only eigenvalues with a
nonnegative real part are $\lambda=0,$ $\lambda=i$ and $\lambda=-i.$


\textbf{In the case $\theta\ll1,$} we consider a perturbation of the
eigenvalue $\lambda_{0}=i$ (the case $\lambda_0=-i$ is similar) and its eigenvector given by~%
\eqref{eq:eigen:theta0} by writing
the asymptotics 
\begin{equation*}
\lambda=\lambda_{0}+\theta\lambda_{1},\quad V=V_{0}+\theta V_{1},\quad
W=W_{0}+\theta W_{1},\quad\varphi_{j}=\varphi_{j,0}+\theta\varphi_{j,1}, 
\end{equation*}
then keeping only the first order terms in $\theta,$ the system~%
\eqref{eq:V:eigen}--\eqref{eq:phij:eigen} becomes 
\begin{align*}
\lambda_{0}V_{1}+\lambda_{1}V_{0} & =-W_{1}+W_{0}-\varphi_{1,0}, & \implies
& iV_{1}+\lambda_{1}=-W_{1}-i-\varphi_{1,0} \\
\lambda_{0}W_{1}+\lambda_{1}W_{0} & =V_{1}, & \implies & iW_{1}-i\lambda
_{1}=V_{1} \\
\implies & -W_{1}+2\lambda_{1}=-W_{1}-i-\varphi_{1,0} & \implies &
\lambda_{1}=-\frac{1}{2}(i+\frac{1+i}{1-r}r),
\end{align*}
and we check numerically that $\mathcal{R}e(\lambda_{1})<0,$ leading to
damping oscillations. 

We now consider the stability of the continuous spectrum, when $\lambda
_{0}=2(\cos(\beta)-1).$ Detailed computations are necessary for a fully rigorous justification, and are left for future work. In brief, we notice that as soon as $\theta \lesssim \beta,$  we have $\mathcal{R}e(\lambda)=\lambda_0+\theta \mathcal{R}e(\lambda_1) <0,$ so this part of the spectrum remains stable. The more delicate part is thus when $\beta \lesssim \theta \ll 1.$ We can carry out an asymptotic expansion as before, and conclude that the continuous spectrum remains in the half-plane with nonpositive real parts. 

To conclude, we note that the oscillations have a period of order
$2\pi$ in the variable $\tau$ and the damping occurs in a time of order $%
1/\theta\sim1/\varepsilon,$ which corresponds to a period of order ${ 2\pi}/{%
\varepsilon}$ and a damping as ${1}/{\varepsilon^{2}}$ in the original
variable $t:$ this is in line with the analysis carried out for the end of Phase~III.
\end{proof}


\section{Numerical computations}\label{sec:app_num}
In this appendix, we give details on the numerical simulation used to illustrate the evolution of the system (\ref{A1})-(\ref{A3}). The code is written in MATLAB and is published online (\href{https://github.com/mmezache/BiMono_Becker_Doring}{https://github.com/mmezache/BiMono\_Becker\_Doring}). The numerical study has an illustrative purpose and is divided based on the framework considered for the model, either the size-continuous framework or the size-discrete framework. 
\subsection{The continuous model}\label{sec:app_PhaseI}
We recall the approximation of the model by a diffusion-convection equation \eqref{eq:DiffApp}--\eqref{eq:DiffCoef} describing the evolution of clusters during Phases~I and II:

\begin{equation*}
\partial_t c(j,t) + \tilde{V}(t)\partial_j c(j,t)= \frac{d(t)}{2}\partial_j^2c(j,t), \quad j\in \Omega\subset\mathbb{R}_+^*, \; t\in [0,T)
\label{eq:evo_p1}
\end{equation*}
where 
\begin{equation*}
\tilde{V}(t) = w(t)-v(t)\; , \quad d(t) = w(t) +v(t)
\end{equation*}
 and where $v$ and $w$ are the solution of the Lotka-Volterra system \eqref{D8}:
 
\begin{equation*}
\left\lbrace
\begin{array}{l@{}l@{}r}
\frac{d}{dt}v(t) & = v(t)(\epsilon -w(t)) ,&\quad  v(0)=0.6,\\
\\
\frac{d}{dt}w(t) & = w(t)(v(t) - \epsilon) ,&\quad w(0)=0.6,
\end{array}
\right.
\label{eq:LV_p1}
\end{equation*}  
where $\epsilon = \int_\Omega c(j,t)dj$.
\subsubsection*{Computation of the monomers}
The system \eqref{D8} depends on the parameters $\epsilon \ll 1$ which makes it stiff when one attempts to solve it numerically. The main reason for the instability of the numerical schemes is that the trajectories of $v$ and $w$, which form the following closed-curve 
$$E = v(t) + w(t) -2\epsilon -\epsilon \log\left(\frac{v(t) w(t)}{\epsilon^2} \right), $$
pass through strictly positive values but very close to $0$. The issue is then to find the balance between the accuracy of the numerical schemes and preserving the positivity properties (see for instance \cite{blanes2022positivity}). To illustrate this phenomenon, let $\epsilon = 10^{-2}$ and we define the function 
$$g: x\in \mathbb{R}_+^* \mapsto 2(x-\epsilon -\epsilon(\log(x) -\log(\epsilon)))-1.$$ 
The function $g$ corresponds to the Hamiltonian $E$ \eqref{D5} minus 1. Then, the root $x^* \in (0,\epsilon)$ such that $g(x^*)=0$ gives insights on the minimal values taken by $v$ and $w$.
We have 
\begin{align*}
g(x^*) & = 0 \\
 - 2 \epsilon\log(x^*) & = 1-2\epsilon\log(\epsilon) +2\epsilon - 2x^* \\
-2\epsilon \log(x^*) & \geq 1-2\epsilon\log(\epsilon)\\
x^* & \leq \exp\left(\tfrac{1-2\epsilon\log(\epsilon)}{-2\epsilon} \right) \leq 2\times10^{-24}.
\end{align*}  
Since the solution of system \eqref{D8} goes very close to $0$, the numerical scheme needs to be capable of preserving the positivity. In order to ensure the positivity of the numerical solutions of system \eqref{D8}, the structure of the Lotka-Volterra system that forces the solutions to remain positive can be used. We apply the following change of variable $x = \log(v)$ and $y = \log(w)$. Hence, we have
\begin{equation}
\left\lbrace
\begin{array}{l@{}l@{}r}
\frac{d}{dt}x(t) & = \epsilon -e^{y(t)} ,&\quad  x(0)=\log(0.6),\\
\\
\frac{d}{dt}y(t) & = e^{x(t)} - \epsilon ,&\quad y(0)=\log(0.6).
\end{array}
\right.
\label{eq:LV_exp_p1}
\end{equation}  
The numerical solution of \eqref{eq:LV_exp_p1} is then computed with the high accuracy Runge-Kutta scheme at the 8th order.

\subsubsection*{Numerical scheme for the diffusion-convection equation}
We propose to use an implicit scheme to solve the diffusion-convection equation \eqref{eq:DiffApp} with a centered difference operator for the convection term. Consider a constant step time discretisation of the interval $[0,T]$ $t_n = n\Delta t$, $n=0,\hdots,N_1,$ where $T=t_{N_1}$ and $\Delta t$ is the time step. Consider a constant step "space" discretisation of the interval of sizes $[0,L]$ $j_k = k\Delta j$, $k=0,\hdots,N_2,$ where $L=j_{N_2}$ and $\Delta j$ is the space step. We look for an approximation $c_k^n$ of $c(k\Delta j, n\Delta t)$. The scheme is then the following:
\begin{equation}
\frac{c^{n+1}_k -c^{n}_k}{\Delta t} + \tilde{V}^{n+1}\frac{c^{n+1}_{k+1} -c^{n+1}_{k-1}}{2\Delta j} -\frac{d^{n+1}}{2}\frac{c^{n+1}_{k+1}-2c^{n+1}_{k} +c^{n+1}_{k-1}}{(\Delta j)^2} = 0,
\label{eq:evo_p1_num}
\end{equation} 
for $n=0, \hdots, N_1-1$ and $k=1,\hdots,N_2-1$. The approximation of the integral terms are given by
\begin{align}
\epsilon^n &= \Delta j \sum\limits_{k=1}^{N_2} c_k^n, \label{eq:epsilon_num}\\
M^n &= \Delta j \sum\limits_{k=1}^{N_2} j_k c_k^n. \label{eq:M1_num}
\end{align}

\subsubsection*{Conservation of the total concentration of clusters}
\begin{proposition}\label{prop:bdd_num_p1}
The numerical scheme defined in \eqref{eq:evo_p1_num} is conservative for the quantity $\epsilon$ if the following equation holds for the boundary conditions : 
\begin{equation}
\begin{split}
 0 = &\, c_0^n - (a^{n+1} + b^{n+1} +1 )c_0^{n+1} - (a^{n+1} - b^{n+1})c_1^{n+1} + c_{N_2}^n\\
 &+ (a^{n+1} - b^{n+1} -1 )c_{N_2}^{n+1} + (a^{n+1} + b^{n+1})c_{N_2-1}^{n+1},   
\end{split}
\label{eq:bdd_cond_num}
\end{equation}
where $$a^{n+1} = \frac{\Delta t \tilde{V}^{n+1}}{2\Delta j} \quad \text{and} \quad b^{n+1} = \frac{\Delta t d^{n+1}}{2(\Delta j)^2}.$$
\end{proposition}
\begin{proof}
For the sake of clarity, we use the following notations $a^{n+1}= a$ and $b^{n+1}=b$.
The approximation of the integral term is defined by \eqref{eq:epsilon_num}, dividing by the space step, we have 
\begin{align*}
\sum\limits_{k=0}^{N_2} c_k^n & = c_0^n  +\sum\limits_{k=1}^{N_2-1} c_k^n + c_{N_2}^n ,\\
& = c_0^n + \sum\limits_{k=1}^{N_2-1} \big( -(a+b)c_{k-1}^{n+1} +(1+2b)c_k^{n+1} +(a-b)c_{k+1}^{n+1}\big) + c_{N_2}^n,\\ 
&=  c_0^n + c_{N_2}^{n+1} + \sum\limits_{k=2}^{N_2-2} c_k^{n+1} -(a+b)(c_0^{n+1} +c_1)^{n+1}\\
& \quad + (1+2b) (c_{1}^{n+1} +c_{N_2 -1}^{n+1}) + (a-b)(c_{N_2 - 1}^{n+1} +c_{N_2}^{n+1}),\\
 & =  c_0^n - (a + b +1 )c_0^{n+1} - (a - b)c_1^{n+1} \\
 & \quad + c_{N_2}^n + (a - b -1 )c_{N_2}^{n+1} + (a + b)c_{N_2-1}^{n+1} +\sum\limits_{k=0}^{N_2} c_k^{n+1}. 
\end{align*}
Hence, if the equation \eqref{eq:bdd_cond_num} holds then the total cluster concentration defined by \eqref{eq:epsilon_num} is the same at each time iteration. 
\end{proof}

As a direct result of Proposition \ref{prop:bdd_num_p1}, the boundary conditions can be chosen as the following in order to keep the tridiagonal structure of the matrix for the numerical scheme \eqref{eq:evo_p1_num} :
\begin{equation*}
\begin{cases}
c_0^n = (1+2a^{n+1})c_0^{n+1} + \Delta j (a^{n+1}-b^{n+1})\frac{c_1^{n+1} - c_0^{n+1}}{\Delta j},\\
c_{N_2}^{n} = (1-2a^{n+1})c_{N_2}^{n+1} + \Delta j (a^{n+1}+b^{n+1})\frac{c_{N_2}^{n+1} - c_{N_2-1}^{n+1}}{\Delta j}. 
\end{cases}
\end{equation*}

\subsubsection*{Discussion about the total mass of the system \eqref{D8}--\eqref{eq:DiffApp}}
The system composed by equations \eqref{D8} and \eqref{eq:DiffApp} is a continuous approximation of the behaviour of the more complex polymerisation/depolymerisation chemical reactions with two monomers and the catalytic depolymerisation.
As such, the system does not ensure the conservation of the total mass. For instance, let $\Omega = \mathbb{R}_+$, assume sufficient regularity conditions for $(c,v,w)$ and the correct boundary conditions such that $\epsilon$ remains constant.
Then  
\begin{align*}
\frac{d}{dt}\left( v(t) + w(t) + \int_\Omega j c(j,t)dj \right)& = -\tilde{V}(t)\epsilon  + \tilde{V}(t) \int_\Omega c(j,t)dj -\frac{d(t)}{2}\int_\Omega \partial_j c(j,t) dj,\\
& = \frac{d(t)}{2}c(0,t). 
\end{align*}
This implies that the solution of \eqref{eq:DiffApp} has a homogeneous Dirichlet boundary conditions. Moreover, the conservation of the total concentration of clusters imposes a homogeneous Robin boundary conditions. Taking into account the two boundary conditions, the solution is locally constant and equal to $0$ in a neighbourhood of $\partial \Omega = 0$. 
Hence, the numerical scheme proposed above is conservative only for the total concentration of clusters. 
{The correction imposed by the mass conservation law in the equations for the monomers is neglected because of its small order of magnitude in the continuous approximation.
The total mass increases at each time step.\\ 
Using the previous notations, consider $N_2 = \infty$, for the sake of simplicity. The discretization scheme \eqref{eq:evo_p1_num}, with the choice of numerical quadratures \eqref{eq:epsilon_num} and \eqref{eq:M1_num} imply that the total mass is increasing by the following amount at each iteration
$$ \frac{d^{n+1} - \Delta j \tilde{V}^n+1}{2}c_0^{n+1}.$$
One thing to note is that the purpose of the discretization of the continuous approximation is to illustrate the transient dynamics of the clusters over one cycle of the Lotka-Volterra oscillations. Hence, the increase of the total mass does not significantly affect the interpretation that can be made of the graphs.
}

\subsubsection*{Numerical simulation of the advection-diffusion equation}
The numerical simulations of \eqref{eq:DiffApp} and \eqref{D8} are illustrated in Figures 
\ref{fig:size_distr_p1} and \ref{fig:size_distr_p1_bis}. 
The parameters are $T=10^4,\; \Delta t = 0.05, \; L = 250$ and  $\Delta j = 0.5$  
The initial conditions are $v(0) = w(0) = 0.6$ and $$c_0(x) = e^{-\tfrac{x^2}{2\sigma}}\mu\sqrt{\tfrac{2}{\pi \sigma}}$$
where $\sigma = 10$ and $\mu = 0.02$. 
Hence, $\epsilon^0 \approx 0.0212$ and $v^0 +w^0 + M^0 \approx 1.2503$.

\subsection{The discrete model}
The numerical computation of the discrete model is using the 8th order Runge Kutta scheme. In order to ensure the nonnegativity of the concentrations of monomers, we apply the change of variables $$v\to e^v \quad \text{and} \quad w \to e^w.$$
The maximal size of the clusters is set to $500$. The code can be found on the following link: \href{https://github.com/mmezache/BiMono_Becker_Doring}{\text{https://github.com/mmezache/BiMono\_Becker\_Doring}}.
\bigskip

\newpage

\end{document}